\documentclass[11pt]{amsart}
\usepackage{amssymb,amsmath}
\usepackage{caption}

\usepackage{epsfig}

\def\sppan{\operatorname{span}}
\def\Bbb{\mathbb}
\def\Cal{\mathcal}
\def\Dt{\partial_t}

\def\eb{\varepsilon}

\def\R {\mathbb{R}}

\def\<{\left<}
\def\>{\right>}
\def\Ree{\operatorname{Re}}

\def\Dx{\Delta_x}
\def\({\left(}
\def\){\right)}

\def\sgn{\operatorname{sgn}}

\def\R{\Bbb R}
\def\Dx{\Delta_x}

\def\Dt{\partial_t}

\def\eb{\varepsilon}
 \makeatletter
 \expandafter\def\csname r@tocindent4\endcsname{0pt}
 \makeatother

\newtheorem{proposition}{Proposition}[section]
\newtheorem{theorem}[proposition]{Theorem}
\newtheorem{corollary}[proposition]{Corollary}
\newtheorem{lemma}[proposition]{Lemma}

\theoremstyle{definition}
\newtheorem{definition}[proposition]{Definition}

\newtheorem{conjecture}[proposition]{Conjecture}
\newtheorem{remark}[proposition]{Remark}
\newtheorem{example}[proposition]{Example}

\numberwithin{equation}{section}

\def \no#1#2#3 {{\bf #1} (#3), #2.}
\def \eds#1#2#3 {#1, #2, #3.}
\begin{document}
\title[] {Inertial manifolds and finite-dimensional reduction for dissipative PDEs}

\author[] {Sergey Zelik}

\thanks{The author is  grateful to V. Chepyzhov, A. Eden, A. Ilyin, V. Kalantarov, J. Robinson, A. Romanov and D. Turaev for stimulating discussions.}

\address{University of Surrey, Department of Mathematics, \newline
Guildford, GU2 7XH, United Kingdom.}
\email{s.zelik@surrey.ac.uk}

\begin{abstract} These notes are devoted to the problem of finite-dimensional reduction for parabolic PDEs.
We give a detailed exposition of the classical theory of inertial manifolds as well as various attempts to generalize it based on the so-called  Man\'e projection theorems. The recent counterexamples which show that the underlying dynamics may be in a sense infinite-dimensional if the spectral gap condition is violated as well as the discussion on the most important open problems are also included.
\end{abstract}

 \maketitle
\tableofcontents
\section{Introduction}\label{s0}
This manuscript is an extended version of  the lecture notes  taught by the author as a part of the crash course in the Analysis of Nonlinear PDEs at Maxwell Center for Analysis and Nonlinear PDEs (Edinburgh, November, 8-9, 2012). It is devoted to one of the central problems of the modern theory of dissipative systems generated by partial differential equations (PDEs), namely, whether or not the underlying dynamics is effectively finite-dimensional and can be described by a system of ordinary differential equations (ODEs).
\par
As hinted from the denomination, dissipative systems consume energy (in contrast to the so-called conservative systems where the total energy is usually preserved), so, in order to have the non-trivial dynamics, the energy income should be taken into the account (in other words, the considered system must be open and should interact with the external world). On the physical level of rigor,  the rich and complicated dynamical structures  (often referred as {\it dissipative structures} following I. Prigogine, see \cite{pri}) arise as the result of interaction of the following three mechanisms:
\par
1) Energy decay, usually more essential in higher "Fourier modes";
\par
2) Energy income, usually through the lower "Fourier modes";
\par
3) Energy flow from lower to higher "modes" provided by the non-linearities.
\par
Moreover, it is typical at least for the dissipative PDEs in {\it bounded} domains that the number of lower "modes" where the the energy income is possible is {\it finite}. So it is natural to expect that these modes will be in a sense dominating (the so-called order parameters in Prigogine's terminology) and the higher modes are slaved by the lower ones. This supports the hypothesis that the effective dynamics in such systems is finite-dimensional up to some transient behavior (can be completely  described by the finitely many order parameters) and somehow explains why the ideas and techniques of the classical finite-dimensional theory of dynamical systems (DS) are also effective for describing the dissipative dynamics in such PDEs, see \cite{tem,BV,CV,sell-book} and references therein.
\par
However, the above arguments are very non-rigorous from the mathematical point of view and it is extremely difficult/impossible to make them meaningful. It is even unclear  what are the "modes" in the above statements. Indeed, they  are rigorously defined as Fourier/spectral modes in the {\it linear} theory and a priori have sense only for systems somehow close to the linear ones. Moreover, as we will see in Paragraph \ref{s.Floquet}, such modes are natural for the linearization near  {\it equilibria} only and may not exist at all when the linear equations with time {\it periodic} coefficients are considered. Thus, despite the common usage of modes and related lengthscales for highly non-linear systems in physical literature (e.g., in  turbulence, see \cite{firsh} and references therein), the precise meaning of them is usually unclear.
\par
That is the reason why the mathematical theory of dissipative PDEs is  based on related but different concepts which at least can be rigorously defined. Namely, let a dissipative system be given by a semigroup $S(t)$, $t\ge0$, acting in a Hilbert or Banach space $H$. Usually, in applications, $H$ is some infinite-dimensional functional space where the initial data lives (say, $H=L^2(\Omega)$) and $S(t)$ is a solution operator of the PDE considered. Then, the dissipativity is often understood as the validity of the following {\it dissipative} estimate:
\begin{equation}\label{0.dis}
\|S(t)u_0\|_H\le Q(\|u_0\|_H)e^{-\alpha t}+C_*
\end{equation}
for the appropriate positive constants $C_*$ and $\alpha$ and monotone increasing function $Q$ which are independent of $u_0$ and $t$. Roughly speaking, this estimate shows that there is a balance between the energy income and dissipation, so the energy cannot grow to infinity. Moreover, for very high energy levels, the dissipation is dominating and the energy  decays.
\par
Next central concept of the theory is the so-called global {\it attractor} $\Cal A$ which is a {\it compact} subset of the phase space $H$ which is invariant with respect to the semigroup $S(t)$ and attracts the images (under the map $S(t)$) of all bounded sets when time tends to infinity. Thus, on the one hand, the attractor $\Cal A$ consists of the full trajectories of the DS considered and contains all its non-trivial dynamics. On the other hand, it is essentially smaller than the initial phase space. In particular, the compactness assumption shows that the higher modes are indeed suppressed, so the existence of a global attractor already somehow supports the hypothesis on the finite-dimensionality.
\par
The theory of attractors for dissipative PDEs has been intensively developing during the last 30 years and the existence of a global attractor is known now-a-days for many various classes of equations, including the Navier-Stokes system, reaction-diffusion equations, damped hyperbolic equations, etc., see \cite{tem,BV,CV,MirZe} and references therein. So, the assumption that the dissipative system considered possesses a global attractor does not look as a great restriction and this assumption is implicitly or explicitly made throughout of the present notes.
\par
We also restrict ourselves to consider only the dissipative systems generated by {\it parabolic} problems, namely, we assume that our dissipative system is governed by the following abstract parabolic equation in a Hilbert space $H$:
\begin{equation}\label{0.eq}
\Dt u+Au=F(u), \ u\big|_{t=0}=u_0,
\end{equation}
where $A$ is a linear self-adjoint (unbounded) operator (which is often the Laplacian or some elliptic operator in applications) and $F:H\to H$ is a non-linearity. It will be also assumed that $F$ is globally bounded and globally Lipschitz continuous in $H$ with the Lipschitz constant~$L$. Note that the global boundedness is not a restriction if the existence of a global attractor is already established since one can cut off the non-linearity outside of a large ball making it globally bounded, see Paragraph \ref{s.app} for the examples of physically relevant equations which can be presented in that form and Paragraph \ref{s.gen} for more general non-linearities. The finite-dimensional reduction for non-parabolic, say, damped hyperbolic equations is more difficult (since, e.g., the inertial manifolds exist only in exceptional cases there) and is out of scope of these notes.
\par
A perfect mathematical definition of the finite-dimensionality of dissipative dynamics is given by the concept of an {\it inertial manifold}. By definition, it is a finite-dimensional smooth submanifold $\Cal M$ of the phase space which is invariant with respect to the semigroup $S(t)$, contains the global attractor and possess the so-called exponential tracking property, namely, any trajectory which starts outside of the manifold approaches exponentially fast some "trace" trajectory belonging to the manifold. If that object exists then restricting equation \eqref{0.eq} to the invariant manifold, we get a system of ODEs which captures the dynamics on the attractor.
\par
The classical theory of inertial manifolds which is considered in details in Section \ref{s1} claims that the $N$-dimensional inertial manifold indeed exists for problem \eqref{0.eq} if the following {\it spectral gap} assumption is satisfied:
$$
\lambda_{N+1}-\lambda_N>2L,
$$
where $L$ is a Lipschitz constant of the non-linearity $F$. The sharpness of that condition and corresponding counterexamples are discussed in Section \ref{s3}. Note that even in that ideal case, the finite-dimensional reduction decreases drastically the smoothness of the considered system since the smoothness of the inertial manifold is usually not better than $C^{1+\eb}$ for some small $\eb>0$, see Paragraph \ref{s.smo}.
\par
An alternative way to justify the finite-dimensionality which is based on the so-called Man\'e projection theorem is considered in Section \ref{s.bsg}. Indeed, one of the main result of the theory of attractors is the fact that the fractal dimension of a global attractor is {\it finite} at least in the case of "good" equations in bounded domains (the study of degenerate/singular equations or equations in unbounded domains where this dimension can naturally be infinite, see \cite{MirZe} and references therein, is also out of scope of these notes):
$$
\dim_F(\Cal A,H)=d<\infty.
$$
 Then, the Man\'e projection theorem claims that a generic orthoprojector to the finite-dimensional plane of dimension $N>2d+1$ is one-to-one on the attractor, see Section \ref{s.bsg} for details. This allows us to reduce the dynamics on the attractor to a finite-dimensional system of ODEs {\it without} the spectral gap assumption. However, this reduction is far from being perfect since the reduced equations are only H\"older continuous which is not enough even for the uniqueness and all attempts to increase the smoothness of that equations are in fact failed. So, the drastic and  in a sense unacceptable loss of smoothness under that reduction looks unavoidable. Moreover, as recent counterexamples of \cite{EKZ} show even the log-Lipschitz regularity  may be lost.
\par
The counterexamples which show the sharpness of the spectral gap assumption as well as the fact that  we cannot expect more  than  the H\"older continuity of the inverse Man\'e projections are considered in Section \ref{s3}. These examples are mainly based on the counterexamples to the Floquet theory for the infinite-dimensional linear parabolic equations with time-periodic coefficients which in fact were known for a long time, see e.g. \cite{kuch}, but nevertheless have been somehow overseen by the experts in dissipative dynamics. Surprisingly, in contrast to the finite-dimensional case, one may construct a linear equation of the form \eqref{0.eq} with time-periodic coefficients such that the associated period map will be the Volterra type operator in $H$, so no Floquet exponents exist and any solution approaches zero super-exponentially in time. This can be hardly accepted as finite-dimensional phenomenon and even rises up a question about the actual {\it infinite-dimensionality} of the underlying dissipative dynamics.
\par
Finally, the concluding remarks as well as some discussion of the open problems is given in Section \ref{s4}.

\section{Spectral gaps  and inertial manifolds}\label{s1}

This section presents the classical theory of the inertial manifolds. The related notations and the statement of the main result is given in Paragraph \ref{s1.1}. The key estimates for the solutions of linear non-homogeneous equation in a saddle point are given in Paragraph \ref{s1.2}.
These estimates are used in Paragraph \ref{s1.3} for the detailed study of the non-linear problem near the saddle point, in particular, the stable and unstable manifolds are constructed here and the exponential tracking property for them is verified. In Paragraph \ref{s1.4}, we reduce the problem of finding an inertial manifold to the analogous problem for the unstable manifold near the saddle and using the results of the previous paragraph, we complete the proof of the inertial manifold theorem.
\par
The inertial manifold theorem for more general non-linearities $F$ which can contain the spatial derivatives and decrease the smoothness is considered in Paragraph \ref{s.gen} and the smoothness of the inertial manifolds is studied in Paragraph \ref{s.smo}.
\par
The alternative way of constructing the inertial manifolds based on the invariant cones is presented in Paragraph \ref{s.cone}. This method is applied in Paragraph \ref{s1.8} for constructing the inertial manifold using the so-called spatial averaging principle (which  from our point of view is the most beautiful and non-trivial result in the theory of inertial manifolds). Finally, the applications of the proved theorems to concrete equations of mathematical physics are discussed in Paragraph \ref{s.app}.

\subsection{Preliminaries and  main result}\label{s1.1}
We study the following abstract semilinear parabolic equation in a separable Hilbert space $H$:
\begin{equation}\label{1.eqmain}
\Dt u+Au=F(u),\ \ u\big|_{t=0}=u_0,
\end{equation}
where $A: D(A)\to H$ is a linear positive self-adjoint operator with compact inverse (as usual, $D(A)$ stands for the domain of the unbounded operator $A$), $u_0\in H$ is a given initial data and $F:H\to H$ is a given non-linear map which is {\it globally} Lipschitz continuous with the Lipschitz constant $L$, i.e.,
\begin{equation}\label{1.lip}
\|F(u)-F(v)\|_H\le L\|u-v\|_H,\ \ u,v\in H.
\end{equation}
It is well-known, see \cite{tem, hen}, that under the above assumptions, for every $u_0\in H$, problem \eqref{1.eqmain} possess a unique  solution $u(t)$, $t\in\R_+$
which belongs to the space $C(0,T;H)\cap L^2(0,T;D(A^{1/2}))$ for all $T>0$ and satisfy the equation \eqref{1.eqmain} in the sense of the equality in $L^2(0,T;D(A^{-1/2}))$. Equivalently, this solution can be defined using the variation of constants formula
\begin{equation}\label{1.vc}
u(t)=e^{-At}u_0+\int_0^te^{-A(t-s)}F(u(s))\,ds,
\end{equation}
where $e^{-At}$ stands for the analytic semigroup generated by the operator $A$ in $H$, see \cite{hen}.
\par
Thus, problem \eqref{1.eqmain} is globally well-posed and generates a non-linear semigroup $S(t)$ in $H$ via
\begin{equation}\label{1.sem}
S(t)u_0:=u(t),\ u(t) \ {\rm solves}\ \  \eqref{1.eqmain} ,\ \ S(t+h)=S(t)\circ S(h),\ \ t,h\ge0, \ S(0)=Id.
\end{equation}
We  note that, due to the Hilbert-Schmidt theorem, the operator $A$ possesses the complete orthonormal in $H$ system of eigenvectors $\{e_n\}_{n=1}^\infty$ which correspond to the eigenvalues $\lambda_n>0$ numerated in the non-decreasing way:
\begin{equation}\label{1.ei}
Ae_n=\lambda_n e_n,\ \ 0<\lambda_1\le\lambda_2\le\lambda_3\le\cdots
\end{equation}
and, due to the compactness of $A^{-1}$, we also know that $\lambda_n\to\infty$ as $n\to\infty$. For instance, in applications, $A$ is usually the Laplacian (or, more general, the uniformly elliptic operator of order $2l$) in a bounded domain $\Omega$ of $\R^d$. Then, by the classical Weyl formula,
\begin{equation}\label{1.weyl}
\lambda_n\sim C n^{2l/d}
\end{equation}
for some constant $C$ depending on $\Omega$ and on the elliptic operator $A$, see \cite{tri}.
\par
Thus, every $u\in H$ can be presented in the form
\begin{equation}\label{1.four}
u=\sum_{n=1}^\infty u_n e_n,\ \ u_n:=(u,e_n)
\end{equation}
and, due to the Parseval equality,
\begin{equation}\label{1.parseval}
\|u\|^2_H=\sum_{n=1}^\infty u_n^2.
\end{equation}
Moreover, the norm in spaces $H^s:=D(A^{s/2})$, $s\in\R$, is given by
\begin{equation}\label{1.hs}
\|u\|_{H^s}^2:=\sum_{n=1}^\infty\lambda_n^su_n^2.
\end{equation}
Recall that, for $s>0$, $H^s$ is a dense subspace of $H$ and, for $s<0$, it is defined as a completion of $H$ with respect to this norm.
We also introduce the orthoprojector $P_N$ to the first $N$ Fourier modes:
\begin{equation}\label{1.pn}
P_N u:=\sum_{n=1}^N(u,e_n)e_n
\end{equation}
and denote by $Q_N:=Id-P_N$, $H_+:=P_NH$ and $H_-:=Q_NH$. Then, for every fixed $N$, equation \eqref{1.eqmain} can be presented as a coupled system with respect to functions $u_+(t):=P_Nu(t)$ and $u_-(t):=Q_Nu(t)$:
\begin{equation}\label{1.coupled}
\Dt u_++Au_+=F_+(u_++u_-),\ \ \Dt u_-+Au_-=F_-(u_++u_-),
\end{equation}
where $F_+(u):=P_NF(u)$ and $F_-(u):=Q_NF(u)$. Moreover, by definition of the spaces $H_\pm$, they are invariant with respect to the operator $A$ and we have
\begin{equation}\label{1.pregap}
\begin{cases}
(Au,u)\le\lambda_N \|u\|_H^2,\ \ u\in H_+,\\
(Au,u)\ge\lambda_{N+1}\|u\|^2_H,\ \ u\in D(A)\cap H_-.
\end{cases}
\end{equation}
As will be shown below, if the spectral gap $\lambda_{N+1}-\lambda_N$ is large enough in comparison with the Lipschitz constant $L$ of the non-liearity $F$, the nonlinearity $F$ in system \eqref{1.coupled} can be considered as a "small" perturbation of the decoupled equations which correspond to the case $F=0$ and the invariant manifold $H_+$ of the unperturbed equations persists in the perturbed system as well.
Namely, we say that the submanifold $\Cal M$ in $H$ is an {\it inertial} manifold for problem \eqref{1.eqmain} if the following conditions are satisfied:
\par
1. The manifold $\Cal M$ is invariant with respect to the solution semigroup \eqref{1.sem}: $S(t)\Cal M=\Cal M$;
\par
2. It can be presented as a graph of a Lipschitz continuous function $\Phi:H_+\to H_-$:
\begin{equation}\label{1.graph}
\Cal M:=\{u_++\Phi(u_+), \ u_+\in H_+\};
\end{equation}
\par
3. The exponential tracking property holds, i.e., there are positive constants $C$ and $\alpha$ such that, for every $u_0\in H$ there is $v_0\in \Cal M$ such that
\begin{equation}\label{1.phase}
\|S(t)u_0-S(t)v_0\|_H\le Ce^{-\alpha t}\|u_0-v_0\|_H.
\end{equation}
In other words, for any trajectory $u(t)$ of \eqref{1.eqmain} there is a trajectory $v(t)$ on the inertial manifold $\Cal M$ which is exponentially close to $u(t)$ as $t\to\infty$.
\par
Thus, if an inertial manifold exists, we have $u_-(t)=\Phi(u_+(t))$ for every trajectory of \eqref{1.coupled} belonging to the manifold and, therefore, this trajectory is determined by the $N$-dimensional system of ODEs:
\begin{equation}\label{1.inertial}
\Dt u_++Au_+=F_+(u_++\Phi(u_+))
\end{equation}
and any other trajectory attracts exponentially fast to one of such trajectories. In that sense the long-time dynamics of the initial infinite-dimensional system \eqref{1.eqmain} is described by the finite-dimensional system of ODEs \eqref{1.inertial} which is often called the {\it inertial form} of \eqref{1.eqmain}.
\par
The following classical theorem gives the sufficient conditions for the existence of the above defined inertial manifold.

\begin{theorem}\label{Th1.main} Let the above assumptions on the operator $A$ and the non-linearity $F$ hold and let, in addition, for some $N\in\mathbb N$, the following spectral gap condition hold:
\begin{equation}\label{1.gap}
\lambda_{N+1}-\lambda_N>2L,
\end{equation}
where $L$ is a Lipschitz constant of the non-linearity $F$. Then, there exists an $N$-dimensional inertial manifold $\Cal M$ which can be presented as a graph of a Lipschitz continuous function $\Phi:H_+\to H_-$ ($H_+:=P_NH$) and the exponential tracking \eqref{1.phase} holds with $\alpha=\lambda_N$.
\end{theorem}
To the best of our knowledge, the existence of an inertial manifold for equation \eqref{1.eqmain} has been firstly proved in \cite{FST} with the non-optimal constant $C$ in the right-hand side of assumption \eqref{1.gap}. The result with the sharp value $C=2$ of this constant has been obtained independently in \cite{mik} and \cite{rom-man}.
\par  
The proof of this key theorem will be given in the next subsections. As we will also see below, the spectral gap condition \eqref{1.gap} is sharp and the manifold may not exist if it is violated.

\subsection{Linear saddles and dichotomies}\label{s1.2} In  this and next paragraph, we study the dynamics near the multi-dimensional (infinite-dimensional) saddles. As we will see below, the proof of Theorem \ref{Th1.main} is reduced in a straightforward way to verifying the existence of stable/unstable manifolds near such saddle point. In this paragraph, we will study the linear non-homogeneous equation of the form
\eqref{1.eqmain} and prepare some technical tools to study its non-linear perturbations.
\par
We assume here that our Hilbert space $H$ is split into the orthogonal sum of two subspaces $H_+$ and $H_-$:
\begin{equation}\label{1.ort}
H=H_+\oplus H_-
\end{equation}
and the corresponding orthogonal projectors are denoted by $P_+$ and $P_-$ respectively. The self-adjoint operator $\tilde A:D(A)\to H$ is assumed to have the form
\begin{equation}\label{1.diag}
\tilde A=\operatorname{diag}(\tilde A_+,\tilde A_-)
\end{equation}
where the operators $\tilde A_+$ and $\tilde A_-$ satisfy
\begin{equation}\label{1.ssadle}
\begin{cases}
(\tilde A_+u,u)\le -\theta\|u\|^2_H,\ \ u\in H_+,\\ (\tilde A_-u,u)\ge\theta\|u\|^2_H,\ \ u\in H_-,
\end{cases}
\end{equation}
for some fixed positive $\theta$. We emphasize that the operator $\tilde A$ in this subsection {\it differs} from the operator $A$ of the previous section. In particular, it is  assumed neither that $\tilde A$ is positive/bounded from below nor that $\tilde A^{-1}$ is compact, so we are in more general situation here. In the proof of Theorem \ref{Th1.main}, we will use
\begin{equation}\label{1.shift}
\tilde A:=A-\frac{\lambda_N+\lambda_{N+1}}2.
\end{equation}
Then, the validity of \eqref{1.ssadle} with $\theta=\frac{\lambda_{N+1}-\lambda_N}2$ is guaranteed by \eqref{1.pregap}.
\par
We now consider the linear non-homogeneous equation associated with the operator $\tilde A$:
\begin{equation}\label{1.lin}
\Dt u+\tilde Au=h(t),
\end{equation}
where $h(t)$ is a given external force. Equation \eqref{1.lin} is equivalent to
\begin{equation}\label{1.llin}
\Dt u_++\tilde A_+u_+=h_+(t),\ \ \Dt u_-+\tilde A_-u_-=h_-(t)
\end{equation}
with $h_{\pm}(t):=P_{\pm}h(t)$ and $u_{\pm}(t):=P_{\pm}u(t)$.
\par
Let us start with the homogeneous case $h\equiv 0$. Then, equations \eqref{1.llin} can be solved as follows
$$
u_+(t)=e^{-\tilde A_+t}u_+(0),\ \ u_-(t)=e^{-\tilde A_- t}u_-(0)
$$
and, due to  assumptions \eqref{1.ssadle},
\begin{equation}\label{1.dich}
\begin{cases}
\|e^{-\tilde A_+t}\|_{L(H,H)}\le e^{\theta t},\ \ t\le 0,\\
\|e^{-\tilde A_- t}\|_{L(H,H)}\le e^{-\theta t},\ \ t\ge0.
\end{cases}
\end{equation}
Thus, the solutions of \eqref{1.lin} starting from the {\it unstable} space $H_+$ tend exponentially
to zero as $t\to-\infty$ and the solutions starting from the {\it stable} space $H_-$ tend exponentially to zero as $t\to+\infty$, so the dynamics of \eqref{1.lin} indeed looks as  a multi-dimensional saddle. It worth to point out however that, in contrast to the finite dimensional saddles, the first equation of \eqref{1.llin} (unstable component) is well-posed only backward in time and maybe ill-posed forward in time. Analogously, the stable component (the second equation of \eqref{1.llin}) may be ill-posed backward in time. Note also that the property \eqref{1.dich} is often referred as  {\it exponential dichotomy}, see \cite{hart} for more details.
\par
We now turn to the non-homogeneous  case $h\ne0$.
The following simple lemma is however the key technical tool for the theory.

\begin{lemma}\label{Lem1.norm} Let the above assumptions hold and let $h\in L^2(\R,H)$. Then, there exists a unique solution $u\in L^2(\R,H)$ of problem \eqref{1.lin} and the following estimate holds
\begin{equation}\label{1.l2main}
\|u\|_{L^2(\R,H)}\le\frac1\theta\|h\|_{L^2(\R,H)},
\end{equation}
where the constant $\theta>0$ is the same as in \eqref{1.ssadle}. Thus, the solution operator $\mathcal T: L^2(\R,H)\to L^2(\R,H)$ is well defined and its norm does not exceed $\frac1\theta$.
\end{lemma}
\begin{proof} We give below only the derivation of estimate \eqref{1.l2main} which automatically implies the uniqueness and the existence of such solution can be verified in a standard way, see e.g. \cite{hen}.
\par
Indeed, multiplying the first equation of \eqref{1.llin} by $u_+$ and using \eqref{1.ssadle}, we get
$$
\frac12\frac d{dt}\|u_+\|^2-\theta\|u_+\|^2= (h_+,u_+)-[(\tilde A_+u_+,u_+)+\theta\|u_+\|^2]\ge (h_+,u_+).
$$
Multiplying it by $2e^{-\theta t}$, we arrive at
\begin{equation}\label{1.proint}
\frac d{dt}\(e^{-2\theta t}\|u_+(t)\|^2\)\ge -2e^{-\theta t}(h_+(t),u_+(t)).
\end{equation}
Integrating this inequality over the time interval $[t,\infty)$, we have
$$
\|u_+(t)\|^2\le 2\int_t^\infty e^{2\theta (t-s)}(h_+(s),-u_+(s))\,ds
$$
and integrating the last inequality over $t\in\R$ and using the Fubini theorem, we finally arrive at
\begin{multline*}
\|u_+\|_{L^2(\R,H)}^2\le 2\int_{\R}\int_t^\infty e^{2\theta(t-s)}(h_+(s),-u_+(s))\,ds\,dt=\\=2\int_{\R}\(\int_{-\infty}^s e^{2\theta(t-s)}\,dt\) (h_+(s),-u_+(s))\,ds=\\=\frac1\theta\int_R(h_+(s),-u_+(s))\,ds\le \frac1{2\theta^2}\|h_+\|^2_{L^2(\R,H)}+\frac12\|u_+\|^2_{L^2(\R,H)}
\end{multline*}
and, thus,
\begin{equation}\label{1.pos}
\|u_+\|^2_{L^2(\R,H)}\le\frac1{\theta^2}\|h_+\|^2_{L^2(\R,H)}.
\end{equation}
Analogously, multiplying the second equation of \eqref{1.llin} by $2e^{\theta t}u_-(t)$, using \eqref{1.ssadle} and integrating over $(-\infty,t]$, we have
$$
\|u_-(t)\|^2\le 2\int_{-\infty}^t e^{-2\theta(t-s)}(h_-(s),u_-(s))\,ds
$$
and integrating this inequality over $t\in\R$ and using the Fubini theorem, analogously to \eqref{1.pos}, we arrive at
\begin{equation}\label{1.neg}
\|u_-\|^2_{L^2(\R,H)}\le \frac1{\theta^2}\|h_-\|^2_{L^2(\R,H)}.
\end{equation}
Taking the sum of \eqref{1.pos} and \eqref{1.neg} and using the obvious fact that
\begin{equation}\label{1.key}
\|v\|^2_{L^2(\R,H)}=\|v_+\|^2_{L^2(\R,H)}+\|v_-\|^2_{L^2(\R,H)},
\end{equation}
we end up with \eqref{1.l2main} and finish the proof of the lemma.
\end{proof}
\begin{remark}\label{Rem1.linf} As we can see from the proof of the lemma, the desired solution $u(t)$ is defined by the following version of the variation of constants formula:
\begin{equation}\label{1.var}
u_+(t)=-\int_t^\infty e^{-\tilde A_+(t-s)}h_+(s)\,ds,\ \ u_-(t)=\int_{-\infty}^t e^{-\tilde A_-(t-s)}h_-(s)\,ds
\end{equation}
and, using \eqref{1.dich}, one can easily show that if $h\in C_b(\R,H)$ (i.e., bounded and continuous with values in $H$), then the solution $u(t)$ defined by \eqref{1.var} also belongs to $C_b(\R,H)$ and
\begin{equation}\label{1.linf}
\|u\|_{C_b(\R,H)}\le \frac C{\theta}\|h\|_{C_b(\R,H)}
\end{equation}
for some positive constant $C$ which is independent of $\theta$. However, we do not have the analogue of \eqref{1.key} for that spaces and cannot take $C=1$ in \eqref{1.linf}. By this reason, the usage of "more natural" space $C_b(\R,H)$ instead of $L^2(\R,H)$ in the proof of the inertial manifold existence (see below) leads to the artificial and non-sharp constant $2\sqrt 2$ (instead of $2$) in the spectral gap condition \eqref{1.gap}. To the best of our knowledge, the advantage of using the $L^2$-norms has been firstly observed in \cite{mik}.
\end{remark}
To conclude this subsection, we state two corollaries of the proved lemma: the first is a kind of smoothing property  and estimates the value $u(t)$
 in a fixed point $t$ through the $L^2$-norms and the second one gives the weighted analogue of \eqref{1.l2main}.

\begin{corollary}\label{Cor1.sm} Under the assumptions of Lemma \ref{Lem1.norm} the solution $u(t)$ belongs to $C_b(\R,H)$ and
\begin{equation}\label{1.smo}
\|u\|_{C_b(\R,H)}\le C\|h\|_{L^2(\R,H)},
\end{equation}
where the constant $C$ is independent of  $h$.
\end{corollary}
\begin{proof}Indeed, for any fixed $T\in\R$, multiplying inequality \eqref{1.proint} by $(T+1-t)$ and integrating over $[T,T+1]$, we have
$$
e^{-2\theta T}\|u_+(T)\|^2\le \int_{T}^{T+1}e^{-2\theta t}(T+1-t)[-2(h_+(t),u_+(t))+\|u_+(t)\|^2]\,dt
$$
and, therefore, due to \eqref{1.l2main}
$$
\|u_+(T)\|^2\le C\int_T^{T+1}\|u_+(t)\|^2+\|h_+(t)\|^2\,dt\le C_1\|h\|^2_{L^2(\R,H)}.
$$
Thus, the desired estimate for the $u_+$ component is proved. The $u_-$ component can be estimated analogously and the corollary is proved.
\end{proof}
\begin{corollary}\label{Cor1.weight} Let $\eb\in\R$ be such that $|\eb|<\theta$, $\tau\in\R$ is arbitrary and $\varphi_{\eb,\tau}(t):=e^{-\eb|t-\tau|}$ be the corresponding weight function. Define the weighted space $L^2_{\varphi_{\eb,\tau}}(\R,H)$ by the following norm:
\begin{equation}\label{1.wnorm}
\|u\|_{L^2_{\varphi_{\eb,\tau}}(\R,H)}^2:=\int_\R\varphi^2_{\eb,\tau}(t)\|u(t)\|^2_H\,dt.
\end{equation}
Then, the solution operator $\mathcal T$ defined in Lemma \ref{Lem1.norm} is bounded in the space $L^2_{\varphi_{\eb,\tau}}(\R,H)$ and
\begin{equation}\label{1.expnorm}
\|\mathcal Th\|_{L^2_{\varphi_{\eb,\tau}}(\R,H)}\le \frac1{\theta-|\eb|}\|h\|_{L^2_{\varphi_{\eb,\tau}}(\R,H)}.
\end{equation}
\end{corollary}
\begin{proof}Indeed, let $v(t):=\varphi_{\eb,\tau}(t) u(t)$. Then, the $L^2(\R,H)$-norm of $v$ is the same as the $L^2_{\varphi_{\eb,\tau}}(\R,H)$-norm of $u$, so we only need to estimate the non-weighted norm of $v$. The function $v$ obviously satisfies the equation
\begin{equation}\label{1.wv}
\Dt v+\tilde Av=\varphi_{\eb,\tau}(t)h(t)+\varphi_{\eb,\tau}'(t)(\varphi_{\eb,\tau}(t))^{-1}v
\end{equation}
Applying estimate \eqref{1.l2main} to this equation and using that
$$
|\varphi'_{\eb,\tau}(t)|=|\eb|\varphi_{\eb,\tau}(t),
$$
we have
$$
\|v\|_{L^2(R,H)}\le\frac1\theta\|h\|_{L^2_{\varphi_{\eb,\tau}}(\R,H)}+\frac{|\eb|}{\theta}\|v\|_{L^2(\R,H)}.
$$
Thus,
$$
\|v\|_{L^2(\R,H)}\le\frac1{\theta-|\eb|}\|h\|_{L^2_{\varphi_{\eb,\tau}}(\R,H)}
$$
and the corollary is proved.
\end{proof}

\subsection{Nonlinear saddles: stable and unstable manifolds}\label{s1.3} In that paragraph, we consider the perturbed version of equation \eqref{1.lin}
\begin{equation}\label{1.seq}
\Dt u+\tilde A u=\tilde F(t,u),
\end{equation}
where the operator $\tilde A$ satisfies the assumptions of the previous subsection and the non-linearity $\tilde F$ is globally Lipschitz continuous with the Lipschitz constant $L$:
\begin{equation}\label{1.tl}
\|\tilde F(t,u_1)-\tilde F(t,u_2)\|_H\le L\|u_1-u_2\|_H
\end{equation}
uniformly with respect to $t\in\R$ and satisfies also the condition
\begin{equation}\label{1.nul}
\tilde F(t,0)\equiv 0.
\end{equation}
The aim of this paragraph is to show that the saddle structure generated by the linear equation \eqref{1.seq} (with $\tilde F=0$) persists if the Lipschitz constant $L$ is not large enough. To this end, we first need to define the stable and unstable sets of equation \eqref{1.seq}. Note also that since the equation considered is {\it non-autonomous}, these sets will also depend on time.

\begin{definition}\label{Def1.stunst} Let $\tau\in\R$ be fixed. The unstable set $\Cal M_+(\tau)\subset H$ consists of all $u_\tau\in H$ such that there exists a {\it backward} trajectory $u(t)$, $t\le\tau$, such that
\begin{equation}\label{1.uns}
u(\tau)=u_\tau,\ \ \|u\|_{L^2((-\infty,\tau],H)}<\infty.
\end{equation}
Analogously, the stable set $\Cal M_-(\tau)\subset H$ consists of all $u_\tau\in H$ such that there exists a {\it forward} trajectory $u(t)$, $t\ge\tau$, such that
\begin{equation}\label{1.stab}
u(\tau)=u_\tau,\ \ \|u\|_{L^2([\tau,+\infty),H)}<\infty.
\end{equation}
\end{definition}
The following theorem can be considered as the main result of this paragraph.

\begin{theorem}\label{Th1.manex} Let the above assumptions hold and let, in addition, the following spectral gap condition holds:
\begin{equation}\label{1.s-gap}
\theta>L,
\end{equation}
where $L$ is a Lipschitz constant of the nonlinearity $\tilde F$. Then, for every $\tau\in\R$, the set $\Cal M_+(\tau)$ is a Lipschitz manifold
over $H_+$ which is a graph of the uniformly in $\tau$ Lipschitz continuous function $M_+(\tau,\cdot):H_+\to H_-$, i.e.
\begin{equation}\label{1.uns-man}
\Cal M_+(\tau)=\{u_++M_+(\tau,u_+),\ \ u_+\in H_+\},\ \ \|M_+(\tau,v_1)-M_+(\tau,v_2)\|_{H_-}\le K\|v_1-v_2\|_{H_+}
\end{equation}
for some positive constant $K$.
Analogously, the stable set $\Cal M_-(\tau)$ is a graph of the uniformly in $\tau$ Lipschitz continuous function $M_-(\tau,\cdot):H_-\to H_+$.
\end{theorem}
\begin{proof} We will consider below only the unstable set $\Cal M_+(\tau)$. The proof for the stable set is analogous. To verify \eqref{1.uns-man}, it is sufficient to show that for every $u_\tau\in H_+$ there is a unique solution $u\in L^2((-\infty,\tau])$ of the problem
\begin{equation}\label{1.unss}
\Dt u+\tilde Au=\tilde F(t,u),\ \ P_+u\big|_{t=\tau}=u_\tau
\end{equation}
and that this solution depends on $u_\tau$ in a Lipschitz continuous way. In that case, the desired map $M_+(\tau,u_\tau)$ is defined via
\begin{equation}\label{1.man}
M_+(\tau,u_\tau):=P_-u(\tau).
\end{equation}
To solve \eqref{1.unss}, we introduce a function $v(t):=e^{-\tilde A_+(t-\tau)}u_\tau$, $t\le\tau$ and $w(t):=u(t)-v(t)$. Then, the function $w$ solves
\begin{equation}\label{1.unss1}
\Dt w+\tilde A w=\tilde F(t,w+v(t)),\ \ P_+w\big|_{t=\tau}=0
\end{equation}
and, due to \eqref{1.dich},
\begin{equation}\label{1.uns1}
\|v\|_{L^2((-\infty,\tau],H)}\le \frac1\theta\|u_\tau\|.
\end{equation}
As the next step, we transform equation \eqref{1.unss1} to the equivalent equation on the whole line $t\in\R$. To this end, we introduce the function
\begin{equation}\label{1.unss2}
\bar F(t,u_\tau,w):=\begin{cases} \tilde F(t,w+v(t)),\ \ t<\tau,\\ 0,\,\ t\ge\tau. \end{cases}
\end{equation}
We claim that \eqref{1.unss1} is equivalent to the following problem: find $w\in L^2(\R,H)$ such that
\begin{equation}\label{1.unss3}
\Dt w+\tilde Aw=\bar F(t,u_\tau,w).
\end{equation}
Indeed, any solution $w\in L^2((-\infty,\tau],H)$ of problem \eqref{1.unss1} can be extended to the solution of \eqref{1.unss3} by setting
$$
w_+(t)\equiv0,\ \ w_-(t)=e^{-\tilde A_-(t-\tau)}w_-(\tau)
$$
for $t\ge\tau$ and \eqref{1.dich} guarantees that $w\in L^2(\R,H)$. Vise versa, if $w\in L^2(\R,H)$ solves \eqref{1.unss3} then, for $t\ge\tau$
$$
\Dt w_++\tilde A_+w_+=0.
$$
Since, due to \eqref{1.dich}, only zero forward solution can be square integrable, we conclude that $w_+(\tau)=0$ and $w$ satisfies \eqref{1.unss1}.
\par
Finally, using the solution operator $\Cal T$ defined in Lemma \ref{Lem1.norm}, we transform \eqref{1.unss3} to the equivalent fixed point equation
\begin{equation}\label{1.fixed}
w=\Cal T\circ\bar F(\cdot,u_\tau,w).
\end{equation}
We claim that the right-hand side of \eqref{1.fixed} is a {\it contraction} on the Banach space $L^2(\R,H)$ (for every fixed $u_\tau$). Indeed, according to \eqref{1.tl} and \eqref{1.uns1}, we have
\begin{multline}\label{1.contr}
\|\bar F(\cdot,u_\tau^1,w^1)-\bar F(u_\tau^2,w^2)\|_{L^2(\R,H)}\le L\|w^1+u_\tau^1-w^2-u_\tau^2\|_{L^2((-\infty,\tau],H)}\le\\\le L\(\|w^1-w^2\|_{L^2(\R,H)}+\frac1\theta\|u_\tau^1-u_\tau^2\|_H\)
\end{multline}
and due to \eqref{1.nul}, $\bar F(\cdot,0,0)=0$, therefore, taking $w^2=u_\tau^2=0$ in \eqref{1.contr}, we see also that $\bar F(\cdot,u_\tau,w)\in L^2(\R,H)$ if $u_\tau\in H_+$ and $w\in L^2(\R,H)$.
\par
Thus, due to Lemma \ref{Lem1.norm}, estimate \eqref{1.contr} and the spectral gap condition \eqref{1.s-gap}, the right-hand side of \eqref{1.fixed} is indeed a contraction on $L^2(\R,H)$ with the contraction factor $\frac L\theta<1$ for every fixed $u_\tau\in H_+$. Therefore, by the Banach contraction theorem, there is a {\it unique} solution
$w=W(u_\tau)\in L^2(\R,H)$ of \eqref{1.fixed}. Moreover, since the right-hand side of \eqref{1.fixed} is Lipschitz also in $u_\tau$, the function
$u_\tau\to W(u_\tau)$ is Lipschitz continuous as well. Since $P_-v(t)\equiv0$, the desired function $M_+(\tau,u_\tau)$ is now defined via
$$
M_+(\tau,u_\tau):=W(u_\tau)\big|_{t=\tau}
$$
and the Lipschitz continuity of that function is guaranteed by Corollary \ref{Cor1.sm}. So, the theorem is proved.
\end{proof}
\begin{remark}\label{Rem1.manifold} The Definition \eqref{Def1.stunst} of stable and unstable sets looks slightly different from the traditional one since not only the convergence of the corresponding trajectories to zero as $t\to\pm\infty$, but also their square integrability is required. However, this difference is not essential at least under the assumptions of Theorem \ref{Th1.manex}. Indeed, as not difficult to see using Corollary \ref{Cor1.weight}, the right-hand side of equation \eqref{1.fixed} will be a contraction not only in the space $L^2(\R,H)$, but also in all weighted space $L^2_{\varphi_{\eb,\tau}}(\R,H)$ if $|\eb|<\theta-L$. Using this observation and setting $\eb$ being small {\it positive}, we see, in particular, that any solution satisfying \eqref{1.uns} decays {\it exponentially} as $t\to-\infty$ and
\begin{equation}\label{1.expconv}
\|u(t)\|_H\le Ce^{\eb(t-\tau)}\|u_\tau\|_H,\ \ u_\tau\in\Cal M(\tau),\ \ t\le\tau.
\end{equation}
On the other hand, fixing $\eb$ being {\it negative}, we see that any solution $u(t)$ which is a priori only {\it bounded} as $t\to-\infty$ (or even a priori exponentially growing with sufficiently small exponent) is a posteriori converges exponentially to zero as $t\to-\infty$ and is generated by the initial data belonging to the unstable manifold $\Cal M_+(\tau)$. The analogous is true for the stable manifolds $\Cal M_-(\tau)$ as well.
\end{remark}
We conclude this paragraph by proving the exponential tracking for the stable/unstable manifolds constructed.

\begin{theorem}\label{Th1.track} Let the assumptions of Theorem \ref{Th1.manex} hold, $\tau\in\R$ be fixed and $u(t)$, $t\ge\tau$ be the forward solution of \eqref{1.seq} belonging to $C_{loc}([\tau,\infty),H)$. Then, there exists a solution $v(t)$, $t\in\R$, of problem \eqref{1.seq} such that
\begin{equation}\label{1.trr}
1. \ \ v(t)\in\Cal M_+(t),\ \ t\in\R,\ \ \ 2. \ \ \|u(t)-v(t)\|_H\le C\|u\|_{L^2([\tau,\tau+1],H)}e^{-\eb (t-\tau)}, \ \ t\ge\tau
\end{equation}
for some positive $C$ and $\eb$ which are independent of $t$, $\tau$ and $u$.
\end{theorem}
\begin{proof} Let $\psi(t)$ be the smooth cut-off function such that $\psi(t)\equiv0$ for $t\le\tau$ and $\psi(t)\equiv1$ for $t\ge\tau+1$. We seek for the desired solution $v(t)$ of problem \eqref{1.seq} in the form
\begin{equation}\label{1.good}
v(t):=\psi(t)u(t)+w(t).
\end{equation}
Then, the function $w$ solves
\begin{equation}\label{1.goodeq}
\Dt w+\tilde A w=\tilde F(t,w+\psi(t)u(t))-\psi(t)\tilde F(t,u(t))-\psi'(t)u(t):=\hat F(t,w).
\end{equation}
Due to the \eqref{1.tl},  function $\hat F$ is globally Lipschitz continuous with the Lipschitz constant $L$:
\begin{equation}\label{1.llip}
\|\hat F(t,w_1)-\hat F(t,w_2)\|_H=\|\tilde F(t,w_1+\psi u)-\tilde F(t,w_2+\psi u)\|_H\le L\|w_1-w_2\|_H.
\end{equation}
Moreover,  $\hat F(t,w)=\tilde F(t,u(t)+w)-\tilde F(t,u(t))$ for $t\ge\tau+1$ and $\hat F(t,w)=\tilde F(t,w)$ for $t\le\tau$ and, therefore,
due to \eqref{1.nul}, $\hat F(t,0)=0$ if $t\notin(\tau,\tau+1)$. Thus, due to \eqref{1.llip} and global Lipschitz continuity of $\tilde F$, we have
\begin{equation}\label{1.llip1}
\|\hat F(t,w)\|_H\le L\|w\|_H+C\chi_{(\tau,\tau+1)}(t)\|u(t)\|_H,
\end{equation}
where $C$ is independent of $\tau$ and $u$ and $\chi_{(\tau,\tau+1)}(t)$ is a characteristic function of the interval $(\tau,\tau+1)$.
\par
On the other hand, thanks to \eqref{1.trr} and \eqref{1.good}, we should have $w\in L^2(\R,H)$, so applying the solution operator $\Cal T$ defined in Lemma \ref{Lem1.norm} to equation \eqref{1.goodeq}, we transform it to the fixed point equation
\begin{equation}\label{1.fixed1}
w=\Cal T\circ\hat F(\cdot,w).
\end{equation}
Thanks to \eqref{1.llip} and \eqref{1.llip1}, the spectral gap condition \eqref{1.s-gap} and Lemma \ref{Lem1.norm}, the right-hand side of \eqref{1.fixed1} is a contraction in the space $L^2(\R,H)$. Moreover, due to Corollary \ref{Cor1.weight}, it is also the contraction in the weighted space $L^2_{\phi_{\eb,\tau}}(\R,H)$ if $|\eb|<\theta-L$. Thus, due to the Banach contraction theorem, there is a solution $w$ of equation \eqref{1.fixed1} which satisfies
$$
\|w\|_{L^2_{\varphi_{\eb,\tau}}(\R,H)}\le C\|u\|_{L^2([\tau,\tau+1],H)}.
$$
This estimate together with the smoothing property of Corollary \ref{Cor1.sm} and formula \eqref{1.good} gives the desired properties \eqref{1.expconv} and finish the proof of the theorem.
\end{proof}
\begin{remark}\label{Rem1.stable} The analogous exponential tracking property for {\it backward} solutions also holds and can be verified exactly as in Theorem \ref{Th1.track}.
\end{remark}
\begin{remark}\label{Rem1.invar} Recall that we assume the well-posedness of the initial value problem for equation \eqref{1.seq}
 neither backward nor forward in time. However, if we assume in addition that this equation is well-posed {\it forward} in time then the two-parametric family of solution operators $U(t,\tau):H\to H$, $t\ge\tau$, is well-defined by
 $$
 U(t,\tau)u_\tau:=u(t).
 $$
 In that case, it follows immediately from the definition that the unstable sets $\Cal M_+(\tau)$ are invariant with respect to these solution operators:
 \begin{equation}\label{1.inv-uns}
 U(t,\tau)\Cal M_+(\tau)=\Cal M_+(t),\ \ t\ge\tau.
 \end{equation}
 In contrast to that, the stable manifolds are only semi-invariant:
 $$
 U(t,\tau)\Cal M_-(\tau)\subset\Cal M_-(t).
 $$
 The absence of the strict invariance is related with the fact that, in general, not all forward trajectories can be extended backward in time.
 \end{remark}

 \subsection{Existence of an inertial manifold}\label{s1.4} We are now ready to prove the key Theorem \ref{Th1.main}. The idea of the proof is, following to \cite{gor-chep}, to present equation \eqref{1.eqmain} in the form of \eqref{1.seq} by the proper change of the dependent variable $u$ and then to obtain the inertial manifold of \eqref{1.eqmain} as an unstable manifold of \eqref{1.seq}. For simplicity, we give the proof under the extra assumption:
 \begin{equation}\label{1.bad}
 F(0)=0.
 \end{equation}
Then, setting
\begin{equation}\label{1.change}
\tilde u(t):=e^{\alpha t}u(t),\ \ \alpha:=\frac{\lambda_{N+1}+\lambda_N}2,
\end{equation}
we transform \eqref{1.eqmain} to
\begin{equation}\label{1.seq1}
\Dt\tilde u+\tilde A\tilde u=\tilde F(t,\tilde u),\ \ \tilde F(t,\tilde u):=e^{\alpha t}F(e^{-\alpha t}\tilde u),
\end{equation}
where the operator $\tilde A$ is defined by \eqref{1.shift}. Then, as not difficult to see, the function $\tilde F(t,\tilde u)$ is globally Lipschitz continuous with the same Lipschitz constant $L$. Moreover, assumption \eqref{1.dich} follows from \eqref{1.pregap} and the spectral gap condition
\eqref{1.s-gap} is guaranteed by \eqref{1.gap}. Finally, \eqref{1.nul} follows from our extra assumption \eqref{1.bad}. Thus, all of the assumptions of Theorem \ref{Th1.manex} are satisfied and, therefore, there exist the unstable manifolds $\Cal M_+(t)$, $t\in\R$, for problem \eqref{1.seq1}. Moreover, by the uniqueness part of Banach contraction theorem, we also know that
$$
M_+(t,u_+)=e^{\alpha t}M_+(0,e^{-\alpha t}u_+).
$$
Thus, $\Cal M_+(0)$ is invariant with respect to the solution semigroup $S(t)$ of problem \eqref{1.eqmain}. Finally, the exponential tracking property \eqref{1.phase} is an immediate corollary of Theorem \ref{Th1.track}. So, $\Cal M_+(0)$ is indeed the inertial manifold for problem \eqref{1.eqmain} and Theorem \ref{Th1.main} is proved.

\begin{remark}\label{Rem1.0} The extra assumption \eqref{1.bad} is purely technical and can be removed. Indeed, if the abstract elliptic equation
\begin{equation}\label{1.ell}
Av=F(v)
\end{equation}
possess a solution in $H$, the change of the dependent variable $\bar u(t):=u(t)-v$ reduces the problem to the case when \eqref{1.bad} is satisfied.
Although the solvability of \eqref{1.ell} does not follow from the assumptions of Theorem \ref{Th1.main}, it holds in the most part of applications.
\par
Alternatively, analyzing the proof of Theorems \ref{Th1.manex} and \ref{Th1.track}, one can see that assumption \eqref{1.nul} can be replaced by  the weaker assumption that $\tilde F(\cdot,0)\in L^2_{\varphi_{\eb,\tau}}((-\infty,\tau),H)$ and that assumption does not require the extra assumption \eqref{1.bad} to be satisfied.
\end{remark}
\begin{remark}\label{Rem1.dis} Note that assumptions of Theorem \ref{Th1.main} allow the solutions of \eqref{1.eqmain} to grow as $t\to\infty$. In that case, the {\it global} Lipschitz continuity of the non-linearity $F$ is indeed necessary. However, usually this theorem is applied to {\it dissipative} equations where the following estimate holds:
\begin{equation}\label{1.dis}
\|S(t)u_0\|_H\le Q(\|u_0\|_H)e^{-\beta t}+C_*,\ \ t\ge0
\end{equation}
for some positive $\beta$ and $C_*$ and a monotone increasing function $u$. In that case, every forward in time trajectory of \eqref{1.eqmain} enters the {\it absorbing} ball $B:=\{u\in H, \ \|u\|_H\le 2C_*\}$ after some time and remains inside for all larger times, see \cite{tem,BV,CV} and references therein. By this reason, since we are mainly interested in the behavior of solutions belonging to the absorbing ball, we may cut-off the nonlinearity outside of this ball without increasing the Lipschitz constant. Thus, instead of the global Lipschitz constant, we may use the {\it local} Lipschitz constant of $F$ on the absorbing ball $B$ only, see e.g., \cite{tem}.
\par
Note also that the inertial manifold $\Cal M=\Cal M_N$ is {\it unique} if $N\in\Bbb N$ satisfying the spectral gap condition is fixed and if $F$ is globally Lipschitz continuous (due to the uniqueness part of the Banach contraction theorem) and is generated by all solutions $u(t)$, $t\in\R$ which grow not faster than $e^{-\alpha t}$, $\alpha:=\frac{\lambda_{N+1}+\lambda_N}2$ as $t\to-\infty$. However, in the case when $F$ is not {\it globally} Lipschitz and the cut-off procedure is needed, like the usual center manifolds, the inertial manifold depends on the way how we cut-off the non-linearity and becomes not unique.
\end{remark}
\begin{remark}\label{Rem1.gen} As we can see from the proof, the assumptions of the main Theorem \ref{Th1.main} can be essentially relaxed. For instance, the extension to case when the non-linearity $F$ depends explicitly on time is immediate (the key Theorem \ref{Th1.manex} is already given for the non-autonomous case). Moreover, neither the finite-dimensionality of $H_+$ nor the fact that $A$ is self-adjoint are essentially used. Thus, the above described scheme can be applied for the non-selfadjoint case, e.g., when $A$ is sectorial or  when it corresponds to the hyperbolic or even elliptic equations, see \cite{tem,mik,gor-chep,chep-hyp,bab,kok1,foseti} for the details. However, the dichotomy \eqref{1.dich} and the spectral gap condition \eqref{1.gap} are {\it crucial} and the inertial manifold may not exist if they are violated, see counterexamples below.
\par
Nevertheless, one should be very careful in applications of the abstract theory involving the non-selfadjoint operators. In a fact, this abstract theory contain some "natural" assumptions (similar to \eqref{1.dich}) which are immediate for the selfadjoint case, but can be surprisingly violated even for the simplest model non-selfadjoint examples. Indeed, let us consider the case of the phase space $H\times H$ with the operator
$$
\Bbb A:=\(\begin{matrix}1&1\\0&1\end{matrix}\)A,
$$
where $A$ is a self-adjoint operator in $H$ satisfying the assumptions of Theorem \ref{Th1.main}. Then, after the transform \eqref{1.change}, the corresponding non-homogeneous linear equation reads:
\begin{equation}\label{1.strange}
\Dt u+(A-\alpha)u+Av=h^1(t),\ \ \Dt v+(A-\alpha)v=h^2(t),
\end{equation}
where $(u,v)\in H\times H$ and, in particular, for the $N$th component, we have the following equations
$$
\Dt u_N-\theta u_N+\lambda_N v_N=h_N^1(t),\ \ \Dt v_N-\theta v_N=h^2_N(t)
$$
which contain the huge (in comparison with $\theta\sim\lambda_{N+1}-\lambda_N$) non-diagonal term $\lambda_N v_N$. By this reason, the solution operator of problem \eqref{1.strange} cannot satisfy estimate \eqref{1.l2main} with the constant $\frac1\theta$, but we may expect this estimate to be valid only with much larger constant $\frac1\theta\(1+\frac{\lambda_N}\theta\)\sim\frac{\lambda_N}{\theta^2}$ and, therefore, essentially stronger spectral gap condition than \eqref{1.gap} is required. Exactly this fact has been overseen in the famous "proof" of the inertial manifold existence for the 2D Navier-Stokes equations given by Kwak (see \cite{kwak,tem-wrong}, see also \cite{rom-wrong} for the analogous mistake in the case of reaction-diffusion equations).
\end{remark}

\subsection{More general class of non-linearities}\label{s.gen} The aim of this paragraph is to discuss briefly the more general case where
 the nonlinearity $F$ decreases the regularity which corresponds in applications to the case
 where the nonlinear terms contain spatial derivatives. Namely, we assume now that the nonlinearity $F$ is globally Lipschitz not as a map in $H$, but as a map from $H^\gamma$ to $H^\beta$ for some $\beta<\gamma$:
 \begin{equation}\label{1.dlip}
 \|F(u_1)-F(u_2)\|_{H^\beta}\le L\|u_1-u_2\|_{H^\gamma},\ \ u_1,u_2\in H^\gamma,
 \end{equation}
 where $H^s:=D(A^{s/2})$. Obviously, without loss of generality, we may assume that $\gamma=0$ (otherwise, we just take $H^\gamma$ as the phase space instead of $H$). Then, since $A^{-1}$ maps $H^s$ to $H^{s+2}$ for all $s$, the non-linearity $F$ will be subordinated to the linear terms of equation \eqref{1.eqmain} if $\beta>-2$. In that case, the standard arguments give the existence and uniqueness of global solutions of \eqref{1.eqmain} and, therefore, the solution semigroup $S(t)$ associated with equation \eqref{1.eqmain} is well-defined, see \cite{hen}.
  Moreover, we may speak about the inertial manifolds for that problem which can be defined exactly as in Paragraph \ref{s1.1}. The following theorem extends Theorem \ref{Th1.main} to that case.

 \begin{theorem}\label{Th1.dmain} Let the operator $A$ is the same as in Theorem \ref{Th1.main}, $F$ satisfies \eqref{1.dlip} with $\gamma=0$ and
 $-2<\beta\le0$. Let also $N\in\Bbb N$ be such that the following spectral gap condition is satisfied:
 \begin{equation}\label{1.g-gap}
 \frac{\lambda_{N+1}-\lambda_N}{\lambda_{N+1}^{-\beta/2}+\lambda_N^{-\beta/2}}>L.
 \end{equation}
 Then, problem \eqref{1.eqmain} possesses the $N$ dimensional inertial manifold with exponential tracking property which is a graph of a Lipshitz function $\Phi:H_+\to H_-$ with $H_+:=P_NH$.
 \end{theorem}
The proof of this theorem is analogous to the proof of Theorem \ref{Th1.main}. The only difference is that instead of Lemma \ref{Lem1.norm}, we need to use the following natural generalization of it.

\begin{lemma}\label{Lem1.d-norm} Let the operator $A$ satisfy the assumptions of Theorem \ref{Th1.dmain},
and $-2<\beta\le0$ and let
\begin{equation}\label{1.shift1}
\tilde A:=A-\alpha,\ \ \alpha:=\lambda_{N+1}\frac{\lambda_N^{-\beta/2}}{\lambda_N^{-\beta/2}+\lambda_{N+1}^{-\beta/2}}+
\lambda_{N}\frac{\lambda_{N+1}^{-\beta/2}}{\lambda_N^{-\beta/2}+\lambda_{N+1}^{-\beta/2}}\in(\lambda_N,\lambda_{N+1}).
\end{equation}
Then, for every $h\in L^2(\R,H^{\beta})$, equation \eqref{1.lin} possesses a unique solution $u\in L^2(\R,H)$ and the solution operator $\Cal T$ satisfies the following estimate:
\begin{equation}\label{1.dest}
\|\Cal T h\|_{L^2(\R,H)}\le \frac{\lambda_{N+1}^{-\beta/2}+\lambda_N^{-\beta/2}}{\lambda_{N+1}-\lambda_N}\|h\|_{L^2(\R,H^{\beta})}.
\end{equation}
\end{lemma}
\begin{proof} The $n$th component $u_n(t):=(u(t),e_n)$ solves
\begin{equation}\label{1.n}
\frac d{dt}u_n(t)+(\lambda_n-\alpha)u_n(t)=h_n(t).
\end{equation}
Multiplying \eqref{1.n} by $u_n$ and arguing as in the proof of Lemma \ref{Lem1.norm}, we have
\begin{equation}\label{1.ggood}
\|u_n\|^2_{L^2(\R)}\le \frac1{(\lambda_n-\alpha)^2}\|h_n\|^2_{L^2(\R)}.
\end{equation}
Taking a sum for all $n\le N$, we get
\begin{multline}\label{1.positive}
\|u_+\|^2_{L^2(\R,H)}=\sum_{n=1}^N\frac{\lambda_n^{-\beta}}{(\lambda_n-\alpha)^2}\cdot\lambda_n^{\beta}\|h_n\|^2_{L^2(\R,H)}\le\\\le
\sup_{n\le N}\bigg\{\frac{\lambda_n^{-\beta}}{(\lambda_n-\alpha)^2}\bigg\}\|h_+\|^2_{L^2(\R,H^{\beta})}=\frac{\lambda_N^{-\beta}}{(\lambda_N-\alpha)^2}\|h_+\|^2_{L^2(\R,H^\beta)},
\end{multline}
where we have implicitly used that the function $f(x):=\frac{x^{-\beta}}{(x-\alpha)^2}$ is monotone increasing for $x<\alpha$ since
$$
f'(x)=x^{-\beta-1}\frac{(2+\beta)x-\beta\alpha}{(\alpha-x)^3}
$$
and $f'(x)>0$ for $x<\alpha$.
Analogously, taking a sum $n\ge N$, we get
\begin{multline}\label{1.negative}
\|u_-\|^2_{L^2(\R,H)}=\sum_{n=N+1}^\infty\frac{\lambda_n^{-\beta}}{(\lambda_n-\alpha)^2}\cdot\lambda_n^{\beta}\|h_n\|^2_{L^2(\R,H)}\le\\\le
\sup_{n\ge N+1}\bigg\{\frac{\lambda_n^{-\beta}}{(\lambda_n-\alpha)^2}\bigg\}
\|h_-\|^2_{L^2(\R,H^{\beta})}=\frac{\lambda_{N+1}^{-\beta}}{(\lambda_{N+1}-\alpha)^2}\|h_-\|^2_{L^2(\R,H^\beta)},
\end{multline}
where we used that $f(x)$ is decreasing for $x>\alpha$. It remains to note that the exponent $\alpha$ is chosen in such way that
$$
\frac{\lambda_N^{-\beta}}{(\lambda_N-\alpha)^2}=\frac{\lambda_{N+1}^{-\beta}}{(\lambda_{N+1}-\alpha)^2}=
\(\frac{\lambda_{N+1}^{-\beta/2}+\lambda_N^{-\beta/2}}{\lambda_{N+1}-\lambda_N}\)^2.
$$
Therefore, taking a sum of \eqref{1.positive} and \eqref{1.negative}, we end up with \eqref{1.dest} and finish the proof of the lemma.
\end{proof}
The rest of the proof of Theorem \ref{Th1.dmain} repeats word by word what is done in Theorem \ref{Th1.main} and by this reason is omitted.

\begin{remark}\label{Rem1.beta} It is  known that the spectral gap condition \eqref{1.g-gap} is also sharp, see, e.g., \cite{rom-man}. In addition, although we state Theorem \ref{Th1.dmain} for $\beta\in(-2,0]$ only, it can be used for positive $\beta$ as well which corresponds to the case of smoothing nonlinearities $F$. The only difference here is that the function $f(x)$ is not monotone increasing on $x\in(0,\alpha)$, but has a minimum at $x=\frac\beta{2+\beta}\alpha$. Thus, together with \eqref{1.g-gap}, we need to assume, in addition, that
$$
\frac{\alpha-\lambda_1}{\lambda_1^{-\beta/2}}>L.
$$
\end{remark}

\subsection{Smoothness of  inertial manifolds}\label{s.smo} In this paragraph, we obtain the extra smoothness of the function $\Phi: H_+\to H_-$ determining the inertial manifold. To this end, we assume  in addition that the nonlinearity $F$ in \eqref{1.eqmain} belongs to $C^{1+\eb}(H,H)$
 for some  $\eb\in(0,1)$, i.e.,
 \begin{equation}\label{1.holder}
 \|F(u_1)-F(u_2)-F'(u_1)(u_1-u_2)\|_H\le C\|u_1-u_2\|_H^{1+\eb},\ \ u_1,u_2\in H,
 \end{equation}
where $F'(u)\in\Cal L(H,H)$ is the Frechet derivative of $F(u)$ in $H$. Note also that, due to the global Lipschitz continuity assumption on $F$, we also have
\begin{equation}\label{1.derL}
\|F'(u)\|_{\Cal L(H,H)}\le L, \ u\in H.
\end{equation}
The main result of this section is the following theorem.

\begin{theorem}\label{Th1.holder} Let the assumptions of Theorem \ref{Th1.main} hold and let also \eqref{1.holder} be valid for some $\eb>0$ such that $\alpha\eb<\theta-L$.
Then, the inertial manifold $\Cal M$ constructed in Theorem \ref{Th1.main} is $C^{1+\eb}$-smooth.
\end{theorem}
\begin{proof}
To verify the extra regularity of the map $\Phi$, we need to differentiate equation \eqref{1.unss}, say, with $\tau=0$, with respect to the initial data $u_0$ and this, in turn, requires us to differentiate the function $\tilde F(t,u)$ defined via \eqref{1.seq1}. According to \eqref{1.holder}, we get
\begin{equation}\label{1.wholder}
\|\tilde F(t,u_1)-\tilde F(t,u_2)-\tilde F'(t,u_1)(u_1-u_2)\|_H\le Ce^{-\eb\alpha t}\|u_1-u_2\|^{1+\eb}_H,
\end{equation}
where $\tilde F'(t,u):=F'(e^{-\alpha t}u)$.
\par
Let us fix now $u_0^1\in H_+$ and let $v(t):=e^{-\tilde A_+t}u_0^1$ denote also by $w_1:=W(u_0^1)$ the solution of \eqref{1.unss3} which corresponds to the initial data $u_0^1$.
\par
 Therefore, differentiating (at this stage formally) equation \eqref{1.unss3} with respect to $u_0\in H_+$, we see that the derivative $\bar w:=W'(u_0^1)\xi$ in the direction $\xi\in H_+$ should satisfy the following equation:
\begin{equation}\label{1.dif1}
\Dt \bar w+\tilde A\bar w=\chi_{t\le0}(t)\(\tilde F'(t,w_1+v_1(t))(\bar w+\xi(t))\),\ \ \xi(t):=e^{-\tilde A_+t}\xi.
\end{equation}
Since the spectral gap condition is satisfied and the norm of the derivative $F'(u)$ is bounded by $L$, applying the Banach contraction theorem to \eqref{1.dif1} analogously to the proof of Theorem \ref{Th1.manex}, we see that \eqref{1.dif1} is solvable for every $\xi\in H_+$ and the following estimate holds:
\begin{equation}\label{1.derest}
\|W'(u_0^1)\xi\|_{C(\R,H)}\le C\|\xi\|_H,
\end{equation}
where we have also implicitly used Corollary \ref{Cor1.sm} in order to obtain the $C(\R,H)$-norm from the control of the $L^2(\R,H)$-norm.
\par
To complete the proof, it only remains to show that the map $W'(u_0^1)$ thus defined is indeed the Frechet derivative. To this end, we need to take another point $u_0^2\in H_+$ and estimate the difference $z:=W(u_0^1)-W(u_0^2)-W'(u_0^1)(u_0^1-u_0^2)$ which solves the following equation:
\begin{equation}\label{1.dereq}
\Dt z+\tilde A z=\chi_{t\le0}(t)\tilde F'(t,w_1+v_1(t))z+H(t),
\end{equation}
where
$$
H(t):=\chi_{t\le0}(t)\(\tilde F(t,w_1+v_1)-\tilde F(t,w_2+v_2)-\tilde F'(t,w_1+v_1)(w_1-w_2+v_1-v_2)\).
$$
Using estimate \eqref{1.holder} together with the analogue of \eqref{1.derest} for the difference $w_1-w_2$, we have
\begin{equation}\label{1.hest}
\|H\|_{L^2_{\varphi_{\alpha\eb,0}}(\R,H)}\le C\(\|w_1-w_2\|_{C(\R,H)}^{1+\eb}+\|v_1-v_2\|^{1+\eb}_{C(\R,H)}\)\le C_1\|u^1_0-u^2_0\|^{1+\eb}_H.
\end{equation}
Applying Corollary \ref{Cor1.sm} to equation \eqref{1.dereq} and using \eqref{1.hest}, we end up with
$$
\|z\|_{L^2_{\varphi_{\alpha\eb,0}}(\R,H)}\le \frac L{\theta-\alpha\eb}\|z\|_{L^2_{\varphi_{\alpha\eb,0}}(\R,H)}+C\|u_0^1-u_0^2\|_{H}^{1+\eb}
$$
which due to the assumption on $\eb$ gives
\begin{equation}\label{1.derfin}
\|W(u_0^1)-W(u_0^2)-W'(u_0^1)(u_0^1-u_0^2)\|_{L^2_{\varphi_{\alpha\eb,0}}(\R,H)}\le C\|u_0^1-u_0^2\|_H^{1+\eb}.
\end{equation}
Using that $\Phi(u_0):=W(u_0)\big|_{t=0}$ together with Corollary \eqref{Cor1.sm}, we finally arrive at
$$
\|\Phi(u_0^1)-\Phi(u_0^2)-\Phi'(u_0^1)(u_0^1-u_0^2)\|_H\le C\|u_0^1-u_0^2\|_H^{1+\eb},
$$
where $\Phi'(u_0)\xi:=W'(u_0)\xi\big|_{t=0}$. This estimate shows that $\Phi\in C^{1+\eb}(H_+,H_-)$ and finishes the proof of the theorem.
\end{proof}

\begin{remark}\label{Rem1.smooth} Note that $\tilde F$ is {\it not differentiable} as the map from, say, $C_b(\R,H)$ to itself even if $F\in C^\infty(H,H)$ and the differentiability holds only in the weighted spaces. For instance, analogously to \eqref{1.wholder}, $\tilde F$ will be $n+\eb$ times continuously differentiable only as the map from $C(\R,H)$ to $C_{\varphi_{(n-1)\alpha+\eb,0}}(\R,H)$. By this reason, higher regularity of the manifold requires larger  spectral gaps (to guarantee that equation \eqref{1.dereq} is uniquely solvable in the weighted space $L^2_{\varphi_{(n-1)\alpha+\eb,0}}(\R,H)$) and, as a rule, the center/inertial manifold has only {\it finite} smoothness even in the case when the initial system is $C^\infty$-smooth, see e.g. \cite{hart}.
\par
In particular, as shown in \cite{kok}, under assumptions of Theorem \ref{Th1.dmain}, the inertial manifold is $C^k$-smooth if
\begin{equation}\label{1.nopt}
\lambda_{N+1}-k\lambda_N>\sqrt2 L(\lambda_{N+1}^{-\beta/2}+k\lambda_N^{-\beta/2}).
\end{equation}
It seems that the factor $\sqrt2$ is artificial here and can be removed using the spaces $L^2(\R,H)$ instead of $C(\R,H)$, but we did not check the details and failed to find the proper reference.
\end{remark}
\begin{remark}\label{Rem1.bad} In particular, for the $C^2$-regularity of the inertial manifold, we need (in the simplest case $\beta=0$)
\begin{equation}\label{1.huge}
\lambda_{N+1}-2\lambda_N>2L
\end{equation}
which is hardly compatible with applications where $A$ is an elliptic differential operator and, to the best of our knowledge, there are no examples of such operators where that spectral gap exists for all $L$. Thus, usually the inertial manifolds are only $C^{1+\eb}$-smooth for some {\it small} exponent $\eb$ and the higher regularity ($C^2$ and more) holds only in the exceptional cases (e.g., under the bifurcation analysis where the dimension  $N$ and the Lipschitz constant $L$ are both small).
\par
We also mention that the differentiability assumption \eqref{1.holder} can be weakened as follows:
 \begin{equation}\label{1.holder1}
 \|F(u_1)-F(u_2)-F'(u_1)(u_1-u_2)\|_H\le C\|u_1-u_2\|_{H^1}^{1+\eb},\ \ u_1,u_2\in H^1
 \end{equation}
 which is easier to verify in applications. Indeed, the proof of Theorem \ref{Th1.holder} remains almost unchanged only instead of Corollary \ref{Cor1.sm}, we need to use its analogue with the $C(\R,H^1)$-norm which also can be easily verified using the parabolic smoothing property.
\end{remark}

\subsection{Invariant cones and inertial manifolds}\label{s.cone} In this paragraph, we briefly discuss an alternative a bit more general approach to inertial manifolds based on the invariant cones and squeezing property which is very popular in the theory, see \cite{fen,bates,rom-man,tem,sell-book} and references therein. The advantage of this approach is that the necessary conditions are formulated not in terms of  the nonlinearity $F$, but directly in terms of the trajectories of the associated dynamical system which sometimes allows us to use the specific properties of the considered system and the specific "cancelations" appeared under its integration (the most important example here is the so-called spatial averaging principle which will be discussed below).
\par
Let us introduce the following indefinite quadratic form in the phase space $H$:
\begin{equation}
V(\xi):=\|Q_N\xi\|^2_H-\|P_N\xi\|^2_H,\ \ \xi\in H.
\end{equation}
The next simple Lemma is of fundamental significance for what follows.

\begin{lemma}\label{Lem1.cone} Under the assumptions of Theorem \ref{Th1.main}, the following estimate holds for every two solutions $u_1(t)$ and $u_2(t)$:
\begin{equation}\label{1.best}
\frac12\frac d{dt}V(u_1(t)-u_2(t))+\alpha V(u_1(t)-u_2(t))\le -\mu\|u_1(t)-u_2(t)\|^2_H,
\end{equation}
where $\alpha:=\frac{\lambda_{N+1}+\lambda_N}2,\ \mu:=\frac{\lambda_{N+1}-\lambda_N}2-L>0$.
\end{lemma}
\begin{proof} Indeed, for any two solutions $u_1(t)$ and $u_2(t)$ of equation \eqref{1.eqmain}, the function $v(t):=u_1(t)-u_2(t)$ satisfies
\begin{multline}\label{1.conproof}
\frac12\frac d{dt}V(v(t))+\alpha V(v(t))=-((A_--\alpha)v_-,v_-)-((\alpha-A_+),v_+,v_+)+\\+(F(u_1)-F(u_2),v_--v_+)\le-\frac{\lambda_{N+1}-\lambda_N}2\|v\|^2_H+L\|v\|^2_H=-\mu\|v\|^2_H
\end{multline}
and the lemma is proved.
\end{proof}
Vice versa, as we will se below, the validity of \eqref{1.best} implies the existence of an inertial manifold. To state the main result of this paragraph, we need to define the following  cones:
\begin{equation}\label{1.defcon}
K^+:=\{\xi\in H, \ \ V(\xi):=\|Q_N\xi\|^2_H-\|P_N\xi\|^2_H\le 0\}.
\end{equation}
\begin{corollary}\label{Cor1.consq} Let the nonlinearity $F$ be globally Lipschitz and let \eqref{1.best} be satisfied for all solutions $u_1$ and $u_2$ of problem \eqref{1.eqmain}. Then, the following properties hold:
\par
1. Cone property: the cone $K^+$ is invariant in the following sense:
\begin{equation}\label{1.con-inv}
\xi_1,\xi_2\in K^+\ \  \Rightarrow \ \ S(t)\xi_1-S(t)\xi_2\in K^+,\ \ \text{ for all $t\ge0$},
\end{equation}
where $\xi_1,\xi_2\in H$ and $S(t)$ is a solution semigroup associated with \eqref{1.eqmain}.
\par
2. Squeezing property: there exists positive $\gamma$ and $C$ such that
\begin{equation}\label{1.squee}
S(T)\xi_1-S(T)\xi_2\notin K^+\ \ \Rightarrow\ \ \|S(t)\xi_1-S(t)\xi_2\|_H\le Ce^{-\gamma t}\|\xi_1-\xi_2\|_H,\ \ t\in[0,T].
\end{equation}
\end{corollary}
\begin{proof}
  Indeed, the invariance of the cone $K^+$ is an immediate corollary of \eqref{1.best}. Moreover, since $F$ is globally Lipschitz,
  $$
  \frac12\frac d{dt}\|u_1(t)-u_2(t)\|^2_H\le L\|u_1(t)-u_2(t)\|^2_H,
  $$
  Multiplying this inequality on $\eb>0$ and taking a sum with \eqref{1.best}, we see that
   $$
   V_\eb(u_1(t)-u_2(t)):=\eb\|u_1(t)-u_2(t)\|^2_H+V(u_1(t)-u_2(t))
   $$
    also satisfies \eqref{1.best}
  (with the new constant  $\mu=0$)  if $\eb(L+\alpha)\le\mu$. Thus, if $u_1$ and $u_2$ are such that $u_(T)-u_2(T)\notin K^+$, then
   due to the cone property, $u_1(t)-u_2(t)\notin K_+$ for all $t\in[0,T]$ and, therefore,
 \begin{multline*}
  \eb\|u_1(t)-u_2(t)\|^2_H\le \eb\|u_1(t)-u_2(t)\|^2_H+V(u_1(t)-u_2(t))\le\\\le V_\eb(0)e^{-2\alpha t}\le(1+\eb)\|u_1(0)-u_2(0)\|^2_He^{-2\alpha t}
  \end{multline*}
  and the corollary is proved.
  \end{proof}
  The following classical theorem is the main result of this paragraph.
\begin{theorem}\label{Th1.maincon} Let the nonlinearity $F$ be globally Lipschitz and globally bounded, i.e,
\begin{equation}\label{1.fbound}
\|F(u)\|_{H}\le C,\ \ u\in H.
\end{equation}
Assume also that the solution semigroup $S(t)$ associated with equation \eqref{1.eqmain} satisfies the cone and squeezing properties \eqref{1.con-inv} and \eqref{1.squee} (with some positive constants $\gamma$ and $C$). Then, equation \eqref{1.eqmain} possesses an $N$-dimensional inertial manifold which is a graph of a Lipschitz function $\Phi: P_N H\to Q_N H$ and the exponential tracking also holds.
\end{theorem}
\begin{proof} We split the proof in 4 steps.
\par
{\it Step 1.} We claim that the boundary value problem
\begin{equation}\label{1.boundary}
\Dt u+Au=F(u),\ P_Nu\big|_{t=0}=u_0^+,\ \ Q_Nu\big|_{t=-T}=0
\end{equation}
has a unique solution for any $T<0$ and any $u_0^+\in H_+$. Indeed, consider the map $G_T:H_+\to H_+$ given by
$$
G_T(v):=P_N S(T)v, \ \ v\in H_+,
$$
where $S(t)$ is a solution operator of problem \eqref{1.eqmain}. Then, obviously, $G_T$ is continuous and,  due to the cone condition, for any
two points $v_1,v_2\in H_+$ and associated trajectories $u_i(t)=S(t)v_i$, we have
$$
\|w_-(t)\|_H\le\|w_+(t)\|_H,\ \ t\in(0,T),\ \ w(t):=u_1(t)-u_2(t)
$$
and, therefore,
\begin{multline*}
\frac12\frac d{dt}\|w_+(t)\|^2_H\ge -(A_+w_+,w_+)-(F(u_1)-F(u_2),w_-)\ge\\\ge -\lambda_N\|w_+\|^2_H-L(\|w_+\|_H+\|w_-\|_H)\|w_+\|_H\ge -(\lambda_N+2L)\|w_+\|^2_H.
\end{multline*}
Integrating this inequality, we see that
$$
\|v_1-v_2\|_H\le e^{(\lambda_N+2L)T}\|G_T(v_1)-G_T(v_2)\|_H.
$$
Thus, the map $G_T: H_+\to H_+$ is injective and the inverse is Lipschitz on its domain. Since $H_+\sim\R^N$, by the topological arguments (e.g., one can easily check that $G_T(H_+)$ has empty boundary), we conclude that $G_T$ is a homeomorphism on $H_+$ and, therefore, $G_T(v)=u_0^+$ is uniquelly solvable for all $u_0^+\in H_+$. It remains to note that $u(t)=S(t)G_T^{-1}(u_0^+)$ solves \eqref{1.boundary}.
\par
{\it Step 2.} Let $u_{T,u_0^+}(t)$ be the solution of the boundary value problem \eqref{1.boundary}. We claim that the limit
\begin{equation}\label{1.liminf}
u_{u_0^+}(t):=\lim_{T\to-\infty}u_{T,u_0^+}(t)
\end{equation}
exists for all $t\in(-\infty,0]$ and solves the the problem \eqref{1.boundary} with $T=-\infty$. Indeed, let $u_i(t):=u_{T_i,u_0^+}(t)$ and $w(t):=u_1(t)-u_2(t)$. Then, since $w_+(0)=0$, we know that $w(0)\notin K^+$ and, due to the cone property, $w(t)\notin K^+$ for all $t\le T:=\max\{T_1,T_2\}$ and, due to the squeezing property,
\begin{equation}\label{1.sq-good}
\|u_1(t)-u_2(t)\|_H\le Ce^{-\gamma(t-T)}\|u_1(T)-u_2(T)\|_H\le2C e^{-\gamma(t-T)}\|w_-(-T)\|_H.
\end{equation}
We now use that $F$ is globally bounded. Then, for any solution $u(t)$ of \eqref{1.eqmain}, we have
$$
\frac12\frac d {dt} \|u_-\|^2_H\le -\lambda_{N+1}\|u_-\|^2_H+(F(u),u_-)\le \frac12\lambda_{N+1}\|u_-\|^2_H+C_*
$$
and
\begin{equation}\label{1.qbound}
\|u_-(t)\|^2_H\le e^{-\lambda_{N+1}t}\|u_-(0)\|^2_H+\tilde C,\ \ t\ge0
\end{equation}
for some $\tilde C$ which is independent of $u$. In particular, since the solution of \eqref{1.boundary} starts from $u_-(T)=0$, we conclude that
$$
\|Q_N u_{T,u_0^+}(t)\|^2_H\le \tilde C
$$
for all $T\le0$ and $u_0^+\in H_+$. Then, estimate \eqref{1.sq-good} guarantees that $u_{T,u_0^+}$ is a Cauchy sequence, say, in $C_{loc}((-\infty,0),H)$. Thus, the limit \eqref{1.liminf} exists and $u_{u_0^+}(t)$ is a backward solution of \eqref{1.eqmain}.
\par
{\it Step 3.} Define a set $\Cal S\subset C_{loc}(\R,H)$ as the set of all solutions of \eqref{1.eqmain} obtained as a limit \eqref{1.liminf}. Then, by the construction, this set is strictly invariant
$$
T(h)\Cal S=\Cal S,\ \ (T(h)u)(t):=u(t+h),\ h\in\R
$$
and for any two trajectories $u_1,u_2\in\Cal S$, we have
\begin{equation}\label{1.conS}
u_1(t)-u_2(t)\in K^+,\ \ t\in\R.
\end{equation}
The last property can be easily shown using the approximations $u_{T,u_0^1}$ and $u_{T,u_0^2}$ for which we have $Q_Nu_{T,u_0^1}(T)=Q_Nu_{T,u_0^2}(T)=0$ and passing to the limit $T\to-\infty$.
\par
From \eqref{1.conS} we conclude that the map $\Phi:H_+\to H_-$
$$
\Phi(u_0):=Q_N u(0),\ \ u\in\Cal S,\ P_Nu(0)=u_0
$$
is well-defined and Lipschitz continuous with Lipschitz constant one and the invariance of $\Cal S$ implies the invariance of the Lipschitz manifold
$$
\Cal M:=\{u_0+\Phi(u_0),\ \ u_0\in H_+\}.
$$
Thus, the desired invariant manifold is constructed.
\par
{\it Step 4.} To complete the proof, we need to check the exponential tracking. Indeed, let $u(t)$, $t\ge0$ be a forward trajectory of \eqref{1.eqmain}. Let $T>0$ and $u_T\in\Cal S$ be the solution of \eqref{1.eqmain} belonging to the inertial manifold such that
\begin{equation}\label{1.good-1}
P_N u(T)=P_Nu_T(T).
\end{equation}
Then, obviously, $u(T)-u_T(T)\notin K^+$ and, consequently, $u(t)-u_T(t)\notin K^+$ for all $t\in[0,T]$. We also know that $Q_Nu_T(0)$ is uniformly bounded with respect to $T$. Moreover, since
$$
\|P_N(u(0)-u_T(0)\|_H\le \|Q_N(u(0)-u_T(0)\|_H,
$$
we see that the sequence $u_T(0)$ is uniformly bounded as $T\to\infty$. On the other hand, due to squeezing property,
\begin{equation}\label{1.trrack}
\|u(t)-u_T(t)\|_H\le Ce^{-\gamma t}\|u(0)-u_T(0)\|_H,\ \ t\in[0,T].
\end{equation}
Since $H_+$ is finite-dimensional and $u_T(0)$ is bounded, we may assume without loss of generality that $u_T(0)\to\bar u(0)$ and then the corresponding trajectory $\bar u(t)\in\Cal S$ and satisfies \eqref{1.trrack} for all $t\ge0$. Thus, the theorem is proved.
\end{proof}
\begin{remark}\label{Rem1.var} There are a lot of slightly different versions of the assumptions of Theorem \ref{Th1.maincon}, see \cite{kok,rom-man,sell-book,fen,bates} and references therein. In particular, if the non-linearity $F$ is smooth, then it is more natural to state the cone condition in terms of equation of variations associated with the problem \eqref{1.eqmain} and if the slightly stronger {\it uniform} cone condition is satisfied, the manifold $\Cal M$ will be $C^{1+\eb}$-smooth, and normally hyperbolic, see e.g. \cite{sell}. More interesting generalizations are related with the situation when the cones $K^+=K^+(u)$ depend on the point $u\in H$, see \cite{fen,bates,Hirsh} and the literature cited there. To the best of our knowledge, such construction has been never used so far for establishing the existence of inertial manifolds for dissipative PDEs although there are natural candidates for that, see Example \ref{Ex2.grad} below.
\end{remark}
\begin{remark}\label{Rem1. nonuniform} Arguing as in the proof of Corollary \ref{Cor1.consq}, one can easily check that the cone and squeezing properties are satisfied if for every two trajectories $u_1$ and $u_2$, there exists an exponent $\alpha(t):=\alpha_{u_1,u_2}(t)$ such that
\begin{equation}\label{1.alphabounds}
0<\alpha_-\le\alpha(t)\le\alpha_+<\infty
\end{equation}
(where $\alpha_\pm$ are independent of $u_1$ and $u_2$) such that
\begin{equation}\label{1.bestcone}
\frac d{dt}V(u_1(t)-u_2(t))+\alpha(t) V(u_1(t)-u_2(t))\le-\mu\|u_1(t)-u_2(t)\|^2.
\end{equation}
We will check exactly this condition in order to verify the cone and squeezing property in the next section.
\end{remark}
\subsection{Spatial averaging: an abstract scheme}\label{s1.8} Let us assume that our nonlinearity is differentiable with the derivative $F'(u)\in\Cal L(H,H)$. Then, the standard spectral gap condition is stated in terms of the norm $\|F'(u)\|\le L$. However, this assumption can be refined using the fact that only the entries $(F'(u)e_i,e_j)$ of the "matrix" $F'(u)$ with numbers $i,j$ "close" to $N$ are really essential. In particular, if this truncated finite-dimensional matrix is close to the scalar one, the inertial manifold will exist even if the spectral gap condition is violated. To state the precise result, we need the following projectors: let $N\in\Bbb N$ and $k>0$ be such that $k<\lambda_N$, then
\begin{multline}\label{1.proj}
\bar P_ku:=\sum_{\lambda_n<\lambda_N-k}(u,e_n)e_n, \\  \bar P_{N,k}u:=\sum_{\lambda_{N}-k\le \lambda_n\le\lambda_{N+1}+k}(u,e_n)e_n,\ \bar Q_ku:=\sum_{\lambda_n>\lambda_{N+1}+k}(u,e_n)e_n.
\end{multline}
Then, analogously to \cite{sell}, the following result holds.
\begin{theorem}\label{Th1.sp} Let the function $F$ be globally Lipschitz with the Lipschitz constant $L$, globally bounded and  differentiable and let the numbers $N$ and $k>4L$ be such that
\begin{equation}\label{1.sp}
\|\bar P_{N,k}\circ F'(u)\circ \bar P_{N,k}v-a(u)\bar P_{N,k}v\|_H\le\delta\|v\|^2_H,\ \ u,v\in H,
\end{equation}
where $a(u)\in\R$ is a scalar depending continuously on $u\in H$ and $\delta<L$. Assume also that
\begin{equation}\label{1.spcond}
\frac{2L^2}{k-4L}+\delta<\frac\theta2,\ \ \alpha-2L>0,
\end{equation}
where as before $\alpha:=\frac{\lambda_{N+1}+\lambda_N}2$ and $\theta:=\frac{\lambda_{N+1}-\lambda_N}2$. Then, inequality \eqref{1.bestcone} holds and, therefore, there exists a Lipschitz  $N$-dimensional inertial manifold  with exponential tracking.
\end{theorem}
\begin{proof} Indeed, let $u_1$ and $u_2$ be two solutions of \eqref{1.eqmain} and let
$$
l(t):=\int_0^1F'(su_1(t)+(1-s)u_2(t))\,ds.
$$
Then, obviously, $l(t)$ satisfies \eqref{1.sp} with the same $\eb$ and $a(t):=\int_0^1a(su_1(t)+(1-s)u_2(t))\,ds$.
Let also $v(t):=u_1(t)-u_2(t)$. Then, analogously to \eqref{1.conproof}, but using also that
$$
\|\bar P_kv\|^2_H\le (\alpha-\lambda_N+k)^{-1}((\alpha-A_+)v_+,v_+)\le k^{-1}((\alpha-A_+)v_+,v_+)
$$
and , similarly,
$$
\|\bar Q_kv\|^2_H\le k^{-1}((A_--\alpha)v_-,v_-),
$$
we have
\begin{multline}\label{1.breddd}
\frac12\frac d{dt}V(v(t))+\alpha V(v(t))=-[(\alpha-A_+)v_+,v_+)+((A_--\alpha)v_-,v_-)]+\\+(l(t)v,v_--v_+)\le-\frac\theta2\|v\|^2_H-
\frac12[(\alpha-A_+)v_+,v_+)+((A_--\alpha)v_-,v_-)]+(l(t)v,v_--v_+)\le\\\le
-\frac\theta2\|v\|^2_H-\frac k2(\|\bar P_kv\|^2_H+\|\bar Q_kv\|^2_H)+(l(t)v,v_--v_+).
\end{multline}
Furthermore,
\begin{multline}
(l(t)v,v_--v_+)=(\bar P_{N,k}l(t)v,v_--v_+)+(l(t)v,\bar Q_k v-\bar P_k v)=\\=(\bar P_{N,k}\circ l(t)\circ \bar P_{N,k}v,v_--v_+)+(l(t)v,\bar Q_kv-\bar P_k v)+(\bar P_kv+\bar Q_k v,l^*(t)\circ \bar P_{N,k}(v_--v_+))\le\\\le(\bar P_{N,k}\circ l(t)\circ\bar P_{N,k}v,v_--v_+)+ 2L\|v\|_H(\|\bar P_kv\|^2_H+\|\bar Q_k v\|^2_H)^{1/2}\le\\\le (\bar P_{N,k}\circ l(t)\circ\bar P_{N,k}v,v_--v_+)+\frac{k-4L}2(\|\bar P_kv\|^2_H+\|\bar Q_kv\|^2_H)+\frac{2L^2}{k-4L}\|v\|^2_H.
\end{multline}
Using now assumption \eqref{1.sp} and the obvious fact that $|a(u)|\le L+\delta<2L$, we obtain
\begin{multline}
(\bar P_{N,k}\circ l(t)\circ \bar P_{N,k}v, v_--v_+)\le a(t) V(v(t))+\delta\|v(t)\|^2_H+\\+|a(t)|(\|\bar P_k v\|^2_H+\|\bar Q_kv\|^2_H)\le a(t)V(v(t))+\delta\|v(t)\|^2+2L(\|\bar P_k v\|^2_H+\|\bar Q_k v\|^2_H).
\end{multline}
Inserting the obtained estimates into the right-hand side of \eqref{1.breddd} and using \eqref{1.spcond}, wee see that \eqref{1.bestcone} is indeed satisfied with
$$
\mu:=\frac\theta2 -\frac{2L^2}{k-4L}-\delta>0,\ \ \alpha(t):=\alpha-a(t)>\alpha-2L>0
$$
and the theorem is proved.
\end{proof}
\begin{remark}\label{Rem1.smo} As not difficult to check inequality \eqref{1.bestcone} implies also the so-called uniform cone condition, so the obtained inertial manifold is not only Lipschitz, but also $C^{1+\kappa}$-smooth (for some small $\kappa>0$) if the non-linearity $F$ is smooth enough. Mention also that the name "spatial averaging" for that method comes from the fact that in the key application to the 3D reaction-diffusion equation, see Example \ref{Ex1.3DRDE},
 $F'(u)$ is a multiplication operator on the function $f'(u(t,x))$ and the scalar factor $a(u)$ is the averaging of $f'(u(t,x))$ with respect to the spatial variable $x$.
 \end{remark}
The rest of this paragraph is devoted to adapting Theorem  \ref{Th1.sp} to the more realistic case where the condition \eqref{1.sp} is not uniform with respect to all $u\in H$. Then, we need the following definition.

\begin{definition}\label{Def1.sp-av} We say that an operator $A$ and a bounded, globally Lipschitz and differentiable non-linearity $F:H\to H$ satisfies the spatial averaging condition if there exist a positive exponent $\kappa$  and a positive constant $\rho$ such that, for every $\delta>0$, $R>0$ and $k>0$ there exists {\it infinitely many}
values of $N\in\Bbb N$ satisfying
\begin{equation}\label{1.smallgap}
\lambda_{N+1}-\lambda_N\ge\rho
\end{equation}
and
\begin{equation}\label{1.sp-avcon}
\sup_{\|u\|_{H^{2-\kappa}}\le R}\biggl\{\|\bar P_{N,k}\circ F'(u)\circ \bar P_{N,k}v-a(u)\bar P_{N,k}v\|_H\biggr\}\le\delta\|v\|^2_H,
\end{equation}
for some scalar multiplier $a(u)=a_{N,k,\delta}(u)\in\R$.
\end{definition}
The main difficulty here is that assumption \eqref{1.sp-avcon} involves higher norms of $u$ and we cannot use the straightforward cut-off procedure (say, multiplication of $F$ on $\varphi(\|u\|_{H^{2-\kappa}})$) in order to reduce the situation to Theorem \ref{Th1.sp} since in that case the modified map will no more act from $H$ to $H$, but only from $H^{2-\kappa}$ to $H$ (and, as we already know, much stronger spectral gap assumptions are required for such maps). Thus, the cut-off procedure should be much more delicate. To describe this procedure, we need to remind some elementary
properties of equation~\eqref{1.eqmain}, see also Paragraph \ref{s2.1} below.

\begin{proposition}\label{Prop1.obvious} Let the non-linearity $F$ be globally bounded and operator $A$ satisfies the above assumptions. Then, the following properties hold for any solution $u(t)$ of problem \eqref{1.eqmain}:
\par
1) Dissipativity in $H^s$ and $0\le s\le2$:
\begin{equation}\label{1.diskappa}
\|u(t)\|_{H^s}\le C e^{-\lambda_1 t}\|u(0)\|_{H^s}+R_*,
\end{equation}
for some positive constants $R_*$ and $C$;
\par
2) Smoothing property:
\begin{equation}\label{1.smosmo}
\|u(t)\|_{H^{2}}\le Ct^{-1}\|u(0)\|_H+R_*,\ \ t>0;
\end{equation}
3) Dissipativity of the $Q_N$-component:
\begin{equation}\label{1.disqq}
\|Q_Nu(t)\|_{H^{2-\kappa}}\le e^{-\lambda_{N+1} t}\|Q_Nu(0)\|_{H^{2-\kappa}}+R_*
\end{equation}
for all $N\in\Bbb N$.
\end{proposition}
\begin{proof}
Indeed, all of these properties are standard and follow from the well-known analogous properties of the linear equation ($F=0$) and the fact that
the linear non-homogeneous problem
$$
\Dt v+Av=h(t),\ \ v(0)=0
$$
satisfies
\begin{equation}\label{1.good-good}
\|v\|_{C_b(\R_+,H^{2-\kappa})}\le C_\kappa\|h\|_{C(\R_+,H)}.
\end{equation}
(exactly here we utilize the assumption that $F$ is globally bounded), see \cite{hen} for details. We only mention here that estimate \eqref{1.good-good} does not hold for $\kappa=0$ and, in order to verify the dissipative estimate for the $H^2$-norm, one should use the obvious fact that
\begin{equation}\label{1.dth2}
\|\Dt u\|_H-C\le\|u\|_{H^2}\le\|\Dt u\|_H+C
\end{equation}
and estimate the norm of $v=\Dt u$ using the differentiated equation
$$
\Dt v+Av=F'(u)v.
$$
\end{proof}
\par
Thus,  being mainly interested  in the long-time behaviour of solutions of \eqref{1.eqmain}, we can freely modify the equation outside of the
absorbing ball $\|u\|_{H^{2}}\le 2R_*$. A bit less trivial, but crucial observation is that, due to \eqref{1.disqq}, we need the cone and squeezing properties to be satisfied {\it only} for the solutions $u_1$ and $u_2$ satisfying $\|Q_Nu_i(t)\|_{H^{2-\kappa}}\le 2R_*$. Indeed, to construct the manifold, we only use the solutions of the boundary value problem \eqref{1.boundary} and all such solutions have uniformly bounded
$Q_N$-component due to \eqref{1.disqq} and zero initial data for $Q_Nu(-T)$. The exponential tracking also can be checked for the trajectories belonging to the absorbing ball only (since any trajectory spends only finite time outside of this ball).
Thus, we need not to modify the $Q_N$-component of \eqref{1.eqmain} and may just assume without loss of generality that
\begin{equation}\label{1.no-restriction}
\|Q_Nu\|_{H^{2-\kappa}}\le 2R_*
\end{equation}
(remind that $R_*$ is independent of $N$) in all forthcoming estimates. In contrast to that, the $P_N$-component of $u$ is {\it unavoidably} unbounded on the trajectories on inertial manifold as $T\to-\infty$, so the proper cut-off should be applied to it.
\par
Following \cite{sell}, we introduce the cut-off function $\varphi(\eta)\in C^\infty(\R)$ such that
$$
1. \ \varphi(\eta)\equiv 1,\,\ \eta \le(2R_*)^{1/2};\ \ 2.\ \ \varphi(\eta)\equiv\frac12,\ \ \eta\ge R_1^{1/2}
$$
for some $R_1>R_*$ and
\begin{equation}\label{1.phigood}
\varphi'(\eta)\le0,\ \ \frac12\varphi(\eta)+\eta\varphi'(\eta)>0,\ \ \eta\in\R.
\end{equation}
Obviously, such cut-off function exists. However, \eqref{1.phigood} gives the restriction
$$
\varphi(\eta)\ge \frac{(2R_*)^{1/4}}{\sqrt{\eta}},\ \ \eta\ge(2R_*)^{1/2},
$$
so $R_1$ should actually  satisfy $R_1\ge32R_*$.  Finally, for every $N\in\Bbb N$, introduce the following cut-off versions of equations \eqref{1.eqmain}:
\begin{equation}\label{1.cutted}
\Dt u+Au=F(u)+AP_Nu-\varphi(\|AP_N u\|^2_H)AP_N u.
\end{equation}
Then, by construction of the cut-off function, we see that equations \eqref{1.cutted} and \eqref{1.eqmain} coinside in the absorbing ball $\|u\|_{H^2}\le 2R_*$. Since $H_+=P_NH$ is finite-dimensional, we also conclude that the modified nonlinearity of \eqref{1.cutted} remain globally  Lipschitz, however, its Lipschitz constant is now growing as $N\to\infty$. Nevertheless, as the next theorem shows, for some large $N$s these modified equations will possess the inertial manifolds if the spatial averaging condition is satisfied. Roughly speaking the new nonlinearity is constructed in such way that on the one hand, it does not decrease the spectral gap at any point $u\in H^{2-\kappa}$, but on the other hand, it drastically increases this gap if the $H^2$-norm of $u_+$ is large enough, so for large $u$, we can just use the Lipschitz continuity of $F$ and  the spatial averaging estimate is required only for $u$ belonging to the bounded set in $H^{2-\kappa}$.

\begin{theorem}\label{Th1.spa-final} Let the operator $A$ and the non-linearity $F$ satisfy the spatial averaging assumption. Then, there exist infinitely many $N$s such that the differential inequality \eqref{1.bestcone} is satisfied for the modified equation \eqref{1.cutted} and, therefore, it possesses an $N$-dimensional inertial manifold with exponential tracking.
\end{theorem}
\begin{proof} To prove the theorem, analogously to \cite{sell}, we need to control the impact of the new nonlinearity $T(u):=\varphi(\|AP_Nu\|^2_H)AP_Nu$ and for that we need the following lemma.
\begin{lemma}\label{Lem1.T} Under the above assumptions, the following estimate holds:
$$
(T'(u)v,v)\le
\frac12\lambda_N\|P_N v\|^2_H+\frac12(AP_Nv,P_Nv).
$$
\end{lemma}
\begin{proof} We use the  version of the Cauchy-Schwartz inequality for any 3 vectors $v,w,y\in H$:
\begin{equation}\label{1.CS}
2(v,y)(w,y)\ge\|y\|^2_H((v,w)-\|v\|_H\|w\|_H).
\end{equation}
Indeed, this inequality is equivalent to
$$
(v-2\frac{(v,y)}{\|y\|^2_H}y,w)\le \|v\|_H\|w\|_H
$$
and the last inequality follows from the Cauchy-Schwartz inequality and the fact that  the map $v\to v-\frac{2(v,y)}{\|y\|^2_H}y$ is a reflection with respect to the plane orthogonal to $y$ and, therefore, it is an isometry.
\par
Taking,  $\eta:=\|Au_+\|^2_H$ and $v=v_+\in H_+:=P_NH$ and using inequalities \eqref{1.phigood} and \eqref{1.CS}, we get
\begin{multline*}
(T'(u)v,v)=2\varphi'(\eta)(Au_+,Av)(Au_+,v)+\varphi(\eta)(Av,v)\le -\varphi'(\eta)\eta(\|Av\|_H\|v\|_H-(Av,v))+\\+\varphi(\eta)(Av,v)\le \frac12\varphi(\eta)\lambda_N\|v\|^2_H+\frac12\varphi(\eta)(Av,v)\le\frac12(\lambda_N\|v\|^2_H+(Av,v))
\end{multline*}
and the lemma is proved.
\end{proof}
Thus, by the mean value theorem, we have
\begin{equation}\label{1.t1mt2}
(T(u_1)-T(u_2),v)\le \frac12(\lambda_N\|v\|^2_H+(AP_Nv,v)),
\end{equation}
where $v:=u_1-u_2$.
Moreover, if both $\|P_Nu_i\|_{H^2}\ge R_1$, $i=1,2$, we have $T(u_i)=1/2AP_Nu_i$ and, therefore,
$$
(T(u_1)-T(u_2),v)=\frac12(AP_Nv,v).
$$
Combining these two estimates and using again the integral mean value theorem, we conclude that the better estimate
\begin{equation}\label{1.betterlip}
(T(u_1)-T(u_2),v)\le \frac14\lambda_N\|v\|^2_H+\frac12(AP_Nv,v)
\end{equation}
holds if $\|P_Nu_i\|_{H^2}\ge2R_1$ for $i=1$ or $i=2$.
\par
We are now ready to verify inequality \eqref{1.bestcone} for equation \eqref{1.cutted}. Indeed, analogously to \eqref{1.breddd}, we have
\begin{multline}\label{1.bred-bred}
\frac12\frac d{dt}V(v(t))+\alpha V(v(t)=-[((\alpha-A_+)v,v)+((A_--\alpha)v,v)]+\\+(T(u_1)-T(u_2),v)-(A_+v,v)+(F(u_1)-F(u_2),v_--v_+).
\end{multline}
where $u_1(t)$ and $u_2(t)$ are two solutions and $v(t)=u_1(t)-u_2(t)$. Fix an arbitrary point $t\ge0$.
Assume first that $\|P_Nu_1(t)\|_{H^2}\le 2R_1$ and $\|P_Nu_2(t)\|_{H^2}\le 2R_1$. Then, using estimate \eqref{1.t1mt2} for the new nonlinearity, we arrive at
\begin{equation}\label{1.bred-br}
\frac12\frac d{dt}V+\alpha V\le -\frac12[(\lambda_{N+1}-A_+)v,v)+((A_--\alpha)v,v)]+(l(t),v_--v_+).
\end{equation}
Since $(\lambda_{N+1}-A_+)v,v)=((\alpha-A_+)v,v)+\theta\|v_+\|^2$, inequality \eqref{1.bred-br} has exactly the same structure as \eqref{1.breddd}. Moreover, using also \eqref{1.no-restriction}, we conclude that, in that case,
$$
\|u_i\|_{H^{2-\kappa}}\le \lambda_1^{-\kappa}\|P_Nu_i\|_{H^2}+\|Q_Nu_i\|_{H^2}\le\lambda_1^{-\kappa}2R_1+R_*\le C,
$$
where $C$ is independent of $N$. Thus, using the spatial averaging assumption and arguing exactly as in the proof of Theorem \ref{Th1.sp}, we may find a sequence of $N$s such that
$$
\frac12\frac d{dt}V(v(t))+\alpha(t)V(t)\le -\mu\|v(t)\|^2_H,\ \mu>0,\ \ \alpha(t)>\alpha_+>0.
$$
Let us consider now the case where $\|P_Nu_i(t)\|_{H^2}\ge2R_1$  for  $i=1$ or $i=2$. Then, using better estimate \eqref{1.betterlip} and the fact that $F$ is globally Lipschitz, we have
$$
\frac12\frac d{dt}V+\alpha V\le -\frac14\lambda_N\|v_+\|^2_H+L\|v\|^2_H= -\frac18\lambda_N V-\frac18\lambda_N\|v\|^2_H+L\|v\|^2_H
$$
and if $\frac18\lambda_N>L$, the inequality \eqref{1.bestcone} is verified and the theorem is proved.
\end{proof}

\begin{remark}\label{Rem1.krav0} In the Definition \ref{Def1.sp-av}, we introduce slightly unusual condition $\kappa>0$ (instead of the traditional $\kappa=0$). This is related with the fact that only global boundedness of $F$ does not allow to obtain the key estimate \eqref{1.disqq} for $\kappa=0$, so some more restrictions and extra cut-off functions are necessary in order to achieve this estimate for $\kappa=0$ as well, see \cite{sell}.
\end{remark}

\subsection{Applications}\label{s.app} The aim of that paragraph is to give a number of examples of parabolic PDEs arising in mathematical physics where we can apply the above theory and establish the existence of inertial manifolds. We start with the simplest case of one spatial variable.

\begin{example}\label{Ex1.1DRDS} Consider the 1D reaction-diffusion equation:
\begin{equation}\label{1.1DRDS}
\Dt u-\partial_x(a(x)\partial_x u)+\alpha u=-f(x,u),\ \ x\in[-L,L],
\end{equation}
where $a(x)>a_0$ is smooth, $\alpha>0$ and $f$ is $C^1$ and satisfies the dissipativity assumption:
\begin{equation}\label{1.disas}
f(x,u).u\ge-C,\ \ u\in\R
\end{equation}
for some $C>0$. Problem \eqref{1.1DRDS} is endowed by the standard Dirichlet, Newmann or periodic boundary conditions.
Then, arguing in a standard way, see e.g., \cite{tem,BV}, one can show that \eqref{1.1DRDS} is well-posed in $H:=L^2(-L,L)$ and satisfies the dissipative estimate \eqref{1.dis}. Moreover, using the maximum principle, one can show the existence of the absorbing set in $C[-L,L]$. Thus, we may cut-off the non-linearity for large $u$ and assume without loss of generality that $f$ and $f'_u$ are globally bounded. We set
$$
Au:=-\partial_x(a(x)\partial_x u)+\alpha u,\ \ F(u):=f(\cdot, u).
$$
Then $A$ is positive self-adjoint and $F$ is globally Lipschitz as the map from $H$ to $H$. Moreover, due to the Weyl asymptotic for the eigenvalues, $\lambda_N\sim CN^2$ and, therefore, there are infinitely many $N$s satisfying
\begin{equation}\label{1.1Dgap}
\lambda_{N+1}-\lambda_N\ge \eb N\le \bar \eb\lambda_N^{1/2}
\end{equation}
for some positive $\eb$ and $\bar\eb$. Indeed, if \eqref{1.1Dgap} is violated for all large $N$s ($N>N_\eb$) for all $\eb$, then taking a sum, we will have
$$
\lambda_N\le \lambda_{N_\eb}+\eb\frac{N(N+1)}2\le C_\eb+\frac\eb2N^2
$$
which contradicts the Weyl asymptotic. Thus, the spectral gap condition \eqref{1.gap} is satisfied for all sufficiently large $N$ satisfying \eqref{1.1Dgap} and, therefore, the inertial manifold exists.
\end{example}
\begin{example}\label{Ex1.KS} Consider the so-called Kuramoto-Sivashinksy equation
\begin{equation}\label{1.KS}
\Dt u+\partial_x^4u+2a\partial_x^2 u=\partial_x(u^2),\ \ a\in\R
\end{equation}
say, on the interval $\Omega:=[0,\pi]$ with Dirichlet boundary conditions $u\big|_{\partial\Omega}=\partial_x^2u|_{\partial\Omega}=0$.
Then, as known, \eqref{1.KS} is well-posed and dissipative in $H:=L^2(0,\pi)$. Moreover, the solutions are smooth for $t>0$ and, in particular,
there is an absorbing ball in $H^2(-L,L)\subset C[-L,L]$, see \cite{tem} and references therein. Thus, after the proper cut-off, we may write the equation in the form
\begin{equation}\label{1.KScut}
\Dt u +Au=F(u):=f_1(u)+\partial_x f_2(u)
\end{equation}
with $A:=\partial_x^4+2a\partial_x^2+a^2+1$ and $f_i(u)$ satisfying $|f_i(u)|+|f'_i(u)|\le C$. In that case, due to the presence of spatial derivatives, $F$ does not map $H$ to $H$ and we can only guarantee that $F$ is globally Lipschitz as a map from $H=L^2(0,\pi)$ to $H^{-1}(0,\pi)$, so we need to use Theorem \ref{Th1.dmain} instead. We note that, $A$ is positive and self-adjoint with
$$
H^2:=D(A)=\{u\in H^4(-L,L),\ u\big|_{\partial\Omega}=\partial_x^2u|_{\partial\Omega}=0\}.
$$
Thus, $H^{-1}(-L,L)=D(A^{-1/4})=H^{-1/2}$ and we need to take $\beta=-\frac12$ in \eqref{1.g-gap}. The spectrum of $A$ can be found explicitly:
$\lambda_N=(N^2+a)^2+1$, so
$$
\frac{\lambda_{N+1}-\lambda_N}{\lambda_{N+1}^{1/4}+\lambda_N^{1/4}}\sim 2N^2
$$
and the spectral gap condition \eqref{1.g-gap} holds for large $N$. Thus, due to Theorem \ref{Th1.dmain}, there is an inertial manifold for the 1D Kuramoto-Sivashinsky equation.
\end{example}
We now turn to the 2D examples. In that case, say, for the second order elliptic operators, the Weyl asymptotic reads $\lambda_N\sim CN$ which gives only that
\begin{equation}\label{1.sma}
\lambda_{N+1}-\lambda_N\ge \eb
\end{equation}
for infinitely many $N$s and some {\it fixed} (small) $\eb$ and that is {\it not enough} to guarantee the validity of the spectral gap assumptions for general Lipschitz non-linearities. However, as we will see below, the Weyl asymptotic $\lambda_N\sim N$ does not forbid the arbitrarily large spectral gaps to exist (it will be so, e.g., if we there are sufficiently many multiple eigenvalues like for the Laplace-Beltrami operator on a unit sphere). To the best of our knowledge, the problem of existence of arbitrarily large spectral gaps is open for general 2D domains even in the simples case of the Laplacian with Dirichlet boundary conditions. So, we discuss below only the case of a unit 2D sphere $\Omega=\Bbb S^2$ (with Laplace-Beltrami operator) and the 2D torus with the usual Laplacian, see also \cite{kwean} for some other domains.

\begin{example}\label{Ex1.2DRDS} Consider the 2D analogue of \eqref{1.1DRDS}
\begin{equation}\label{1.2DRDS}
\Dt u -\Dx u+\alpha u=f(x,u),\  \ x\in\Omega,
\end{equation}
where $\alpha>0$ and $f$ satisfies the dissipativity assumption \eqref{1.disas} and $\Omega=\Bbb S^2$ or $\Bbb T^2:=[-\pi,\pi]^2$ (with periodic boundary conditions). As in 1D case, it is known that the problem is well-posed in $H:=L^2(\Omega)$ and possesses the absorbing ball in $C(\Omega)$, see \cite{tem,BV}, so we may assume that $f$ and $f'_u$ are globally bounded and, therefore, the associated operator $F$ is globally Lipschitz in $H$.
So, we only need to look at the gaps in the spectrum of the Laplacian.
\par
The ideal case is $\Omega=\Bbb S^2$. Then, the spectrum of the Laplace-Beltrami operator is well-known: the eigenvalues are $\Lambda_n:=n(n+1)$ and each of them has multiplicity $2n+1$. Thus, the "one-dimensional" inequality \eqref{1.1Dgap} holds for infinitely many $N$s (with $\bar\eb<2$). So, the spectral gap condition is satisfied and the inertial manifold exists.
\par
Let now $\Omega=\Bbb T^2$ with periodic boundary conditions. Then the spectrum of the 2D Laplacian consists of all $N\in\Bbb N$ which can be presented as a sum of two squares:
$$
\sigma(A)=\{N\in\Bbb N,\ N=n^2+k^2,\ \ k,n\in\Bbb Z\}
$$
with the eigenvectors $e^{inx+iky}$. Thus, the spectral gap problem is related with the number theoretic question on the  subsequent integers which are not sums of two squares. In particular, from the Gauss theorem, we know that $N\in\Bbb N$ belongs to $\sigma(A)$ if and only if
every prime factor $p$ of $N$ which equals to $3\mod 4$ has an even degree. Moreover, as shown, e.g., in \cite{sell}, there are arbitrary many subsequent integers which are not sums of two squares and, therefore, the spectral gap condition is satisfied. In fact, it is possible to show that
there are infinitely many solutions of
$$
\lambda_{N+1}-\lambda_N\ge c\log\lambda_N
$$
for some $c>0$ although, to the best of our knowledge,  there are no larger than logarithmical gaps, see \cite{gaps}.
\end{example}
\begin{example}\label{Ex1.2DCH} Consider the 2D Cahn-Hilliard problem:
\begin{equation}\label{1.2DCH}
\Dt u+\Dx(-\Dx u+f(x,u))=0,\ \ x\in\Omega,
\end{equation}
where $\Omega=\Bbb S^2$ or $\Omega=\Bbb T^2$. It is known that this equation is well-posed in $H=L^2(\Omega)$ if the non-linearity $f$ satisfies the dissipativity assumption \eqref{1.disas} together, say, with the assumption $f'_u(x,u)\ge-C$, see \cite{tem}. Moreover, one also has the absorbing ball in $C(\Omega)$ (of course, by modulus of the mass conservation law $\int_\Omega u(t,x)\,dx=const$), so we can again assume without los of generality that $f$ and $f'_u$ are globally bounded. In that case, the map $F(u):=-\Dx f(x,u)$ is globally Lipschitz from $H$ to $H^{-2}(\Omega)$.
The operator $A$ now is a bi-Laplacian $A=\Dx^2$, so $D(A)=H^4(\Omega)$, $H^{-2}(\Omega)=D(A^{-1/2})$, and we should use $\beta=-1$ in the spectral gap condition \eqref{1.g-gap}. Since $\lambda_N=\mu_N^2$, where $\mu_N$ are the eigenvalues of the Laplacian, we have
$$
\frac{\lambda_{N+1}-\lambda_N}{\lambda_{N+1}^{1/2}+\lambda_N^{1/2}}=\mu_{N+1}-\mu_N
$$
and the spectral gap condition is {\it exactly} the same as for the reaction-diffusion problem considered above. Thus, for both choices $\Omega=S^2$ and $\Omega=\Bbb T^2$, the inertial manifold exists.
\end{example}
We now turn to the 3D case. In that case, the Weyl asymptotic gives $\lambda_N\sim CN^{2/3}$ for the case of second order operators which is clearly insufficient to have the spectral gap (although, in some very exceptional cases like the Laplace-Beltrami operator on the $n$-dimensional sphere, we have the "one-dimensional" spectral gap \eqref{1.1Dgap} in any space dimension), however, for the 4th order operators, we still have $\lambda_N\sim N^{4/3}$ and the spectral gap exists.

\begin{example}\label{Ex1.3DSH} Consider the 3D Swift-Hohenberg equation:
\begin{equation}\label{1.3DSH}
\Dt u=-(\Dx+1)^2u-\alpha u+f(x,u), \ x\in\Omega,\ \ u\big|_{\partial\Omega}=\Dx u\big|_{\partial\Omega}=0,
\end{equation}
where $\alpha>0$, $\Omega$ is a bounded domain of $\R^3$ and $f(x,u)$ satisfies the dissipativity assumption \eqref{1.disas}. Then the problem is well-posed and dissipative in $H=L^2(\Omega)$, so as before we may assume without loss of generality that $f$ and $f'_u$ are globally bounded. Then $F$ is globally Lipschitz in $H$.
\par
We set $A:=(\Dx +1)^2+\alpha$ (with the Dirichlet boundary conditions). Then, by the Weyl asymptotic $\lambda_N\sim CN^{4/3}$ and the spectral gap assumption
$$
\lambda_{N+1}-\lambda_N\ge \eb\lambda_N^{1/4}
$$
is satisfied for infinitely many $N$s. Thus, the inertial manifold exists.
\end{example}
 We conclude by the most interesting example where the spatial averaging technique is used.

 \begin{example}\label{Ex1.3DRDE} Consider equation \eqref{1.2DRDS} on a 3D torus $\Bbb T^3=[-\pi,\pi]^3$ (endowed by the periodic boundary conditions). As before, we assume that $f$ and $f'_u$ are globally bounded, so the map $F$ is bounded and globally Lipschitz in $H=L^2(\Bbb T^3)$.
 The spectrum of the Laplacian on $\Bbb T^3$ obviously consists of integers which can be presented as a sum of 3 squares ($\lambda=n^2+m^2+k^2$ and the associated eigenfunction is $e^{inx+imy+ikz}$). By the Gauss theorem, all natural numbers which are not in the form of $4^l(8L+7)$ can be presented as a sum of 3 squares, so we see that, in contrast to the 2D case, there are no more spectral gaps of length larger than 3 and Theorem \ref{Th1.main} does not work. Nevertheless, using the  number theoretic results on the distribution of integer points in 3D spherical layers, one can check the spatial averaging conditions and prove the existence of an inertial manifolds. Indeed, let $N\in\Bbb N$, $\kappa,\rho>0$ and
 $$
 \Cal C_N^\kappa:=\{(n,m,l)\in\Bbb Z^3,\, N-k \le n^2+m^2+l^2\le N+k\},\ \Cal B_\rho:=\{(n,m,l)\in\Bbb Z^3, \, n^2+m^2+l^2\le \rho^2\}.
 $$
 Then, the following result holds.
 \begin{lemma}\label{Lem1.NT} For any $k>0$ and $\rho>0$ there are infinitely many $N\in\Bbb N$ such that
 \begin{equation}\label{1.empty}
 (\Cal C_{N+\frac12}^k-\Cal C_{N+\frac12}^k)\cap\Cal B_\rho=\{(0,0,0)\}
 \end{equation}
 \end{lemma}
The proof of this lemma is given in \cite{sell}.
\par
Let us check now that the spatial averaging assumption is satisfied. Indeed, the derivative $F'(u)v=f'(x,u(x))v(x)$ is a multiplication operator on the function $w(x):=f'_u(x,u(x))$. In Fourier modes this operator is a convolution
$$
[F'(u)v]_{n,m,l}=\sum_{(n',m',l')\in\Bbb Z^3}w_{n-n',m-m',l-l'}v_{n',m',l'},\ u(x):=\sum_{(n,m,l)\in\Bbb Z^3}u_{n,m,k}e^{inx+imy+ilz}.
$$
Denote by $\bar P_{N,k}$ the orthoprojector to the Fourier modes belonging to $\Cal C_{N+\frac12}^k$ and by $w_{>\rho}(x)$ the projection of function $w$ to all Fourier modes which are outside of $\Cal B_\rho$. Then, due to condition \eqref{1.empty},
$$
\bar P_{N,k}((w-w_{0,0,0})\bar P_{N,k} v)=\bar P_{N,k}(w_{>\rho}\bar P_{N,k}v)
$$
and, therefore,
$$
\|\bar P_{N,k}((w-w_{0,0,0})\bar P_{N,k} v)\|_{H}\le \|w_{>\rho}\|_{L^\infty}\|v\|_H.
$$
Furthermore, due to the interpolation, for $\kappa<\frac14$,
$$
\|w_{>\rho}\|_{L^\infty}\le C\|w_{>\rho}\|^{1-\theta}_H\|w_{>\rho}\|^\theta_{H^{2-\kappa}}\le C_1\rho^{-(1-\theta)(2-\kappa)}\|w\|_{H^{2-\kappa}},
$$
where $\theta=\frac3{4(2-\kappa)}$. Finally, since $H^{2-\kappa}$ is an algebra and $f$ is smooth,
$$
\|\bar P_{N,k}((w-w_{0,0,0})\bar P_{N,k} v)\|_{H}\le C\rho^{-(1-\theta)(2-\kappa)}Q(\|u\|_{H^{2-\kappa}})\|v\|_{H}
$$
for some monotone increasing function $Q$, and the spatial averaging estimate \eqref{1.sp-avcon} holds with
$$
a(u):=w_{0,0,0}=\<f'_u(x,u(x))\>=\frac1{(2\pi)^3}\int_{\Bbb T^3}f'_u(x,u(x))\,dx.
$$
Thus, the existence of an inertial manifold is verified.
 \end{example}

\section{Beyond the spectral gap}\label{s.bsg}
In this section, we present an alternative approach to the finite-dimensional reduction which is based on the so-called Man\'e projection theorem and does not require the restrictive spectral gap assumption to be satisfied.
\par
In Paragraph \ref{s2.1}, we recall some basic facts from the theory of attractors, see e.g., \cite{tem,BV,CV} for more details. The basic properties of Hausdorff and fractal dimensions are recalled in Paragraph \ref{s2.2}. Moreover, we show here why the fractal dimension of the
attractor is finite, indicate the proof of the Man\'e projection theorem and build up the finite-dimensional reduction based on that theorem.
\par
The so-called Romanov theory which gives necessary and sufficient conditions for the existence of bi-Lipschitz Man\'e projections is given in Paragraph \ref{s2.3}. Finally, the recent attempts to solve the uniqueness problem for the reduced equations by verifying the log-Lipschitz continuity of the inverse Man\'e projections are discussed in Paragraph \ref{s2.4}.

\subsection{Global attractors}\label{s2.1} The aim of this paragraph is to recall briefly the basic facts from the attractor theory which will be used in the sequel. Since we are not pretending  on a more or less complete survey of the theory (see \cite{tem,BV,CV,MirZe} for the more detailed exposition), we pose here the simplest and most convenient assumptions on the nonlinearity $F$, namely, we assume that $F$ is continuously Frechet differentiable as the map from $H$ to $H$ and
\begin{equation}\label{2.Fass}
1. \ \ \|F(u)\|_H\le C,\ \ 2.\ \ \|F'(u)\|_{\Cal L(H,H)}\le L,\ \ 3.\ \ \|F'(u)\|_{\Cal L(H^1,H^1)}\le Q(\|u\|_{H^2})
\end{equation}
for some positive constants $C$ and $L$ and monotone increasing function $Q$. Note that we do not assume that $F$ is Frechet differentiable as the map from $H^1$ to $H^1$, we even do not need that $F$ maps $H^1$ to $H^1$, only the derivative $F'(u)$ should be bounded as a map from $H^1$ to $H^1$.
\par
 We start with three standard lemmas describing some analytic properties of equation \eqref{1.eqmain} which are crucial for the theory. The first lemma gives the dissipativity in $H$.

\begin{lemma}\label{Lem2.dis} Let the nonlinearity $F$ satisfies \eqref{2.Fass} and $A$ is the same as in the previous chapter. Then, any solution $u(t)$ of problem \eqref{1.eqmain} satisfies the following estimate:
\begin{equation}\label{2.dis}
\|u(t)\|^2_H+\int_t^{t+1}\|u(s)\|_{H^1}^2\,ds\le Ce^{-\alpha t}\|u(0)\|^2_H+C_*,
\end{equation}
for some positive constants $C$, $\alpha$ and $C_*$.
\end{lemma}
\begin{proof} Indeed, taking a scalar product of \eqref{1.eqmain} with $u(t)$ and using that $F$ is bounded, we have
$$
\frac12\frac d{dt}\|u(t)\|^2_H+\|u(t)\|^2_{H^1}\le C\|u(t)\|_H.
$$
Using the Poincare inequality $\|u\|_H\le\lambda_1^{-1}\|u\|_{H^1}$ together with the Gronwall inequality, we arrive at \eqref{2.dis} and finish the proof of the lemma.
\end{proof}
Next straightforward lemma gives the smoothing property  and dissipativity in $H^2$ for the solutions of \eqref{1.eqmain}.
\begin{lemma}\label{Lem2.smo} Let the assumptions of Lemma \ref{Lem2.dis} hold. Then, any solution $u(t)$ of equation \eqref{1.eqmain} satisfies
\begin{equation}\label{2.h2sm}
\|u(t)\|_{H^2}\le Ct^{-1}\(\|u(0)\|_H+1\),\ t\in(0,1].
\end{equation}
Moreover, if $u(0)\in H^2$ then
\begin{equation}\label{2.h2dis}
\|u(t)\|_{H^2}^2\le C\|u(0)\|_{H^2}^2e^{-\alpha t}+C_*,
\end{equation}
for some positive $C$, $\alpha$ and $C_*$.
\end{lemma}
\begin{proof} Indeed, differentiating equation \eqref{1.eqmain} in time and denoting $v=\Dt u$, we have
\begin{equation}\label{2.dt}
\Dt v+Av=F'(u)v.
\end{equation}
Taking the scalar product of this equation with $v(t)$ and using the second assumption of \eqref{2.Fass}, we end up with
\begin{equation}\label{2.dtgr}
\frac12\frac d{dt}\|v(t)\|^2_H+\|v(t)\|_{H^1}^2\le L\|v(t)\|_H^2.
\end{equation}
Applying the Gronwall inequality and using \eqref{1.dth2}, we get
\begin{equation}\label{2.dtgrr}
\|u(t)\|^2_{H^2}\le Ce^{2L t}(\|u(0)\|_{H^2}^2+1).
\end{equation}
This estimate is similar to the desired \eqref{2.h2dis}, but still growing in time. To remove this growth, we prove the smoothing estimate \eqref{2.h2sm}. To this end, we multiply equation \eqref{2.dtgr} by $t^2$, integrate in time and use that, due to \eqref{2.dis} and \eqref{2.Fass},
$$
\int_0^1\|v(t)\|^2_{H^{-1}}\,dt\le C(\|u(0)\|^2_H+1).
$$
Then, for $t\le1$ we have
\begin{multline}
t^2\|v(t)\|^2_H+\int_0^tt^2\|v(t)\|^2_{H^1}\,dt\le C\int_0^t t\|v(t)\|^2_H\,dt\le \int_0^tt^2\|v(t)\|^2_{H^1}\,dt+\\+C^2\int_0^t \|v(t)\|^2_{H^{-1}}\,dt\le \int_0^t\|v(t)\|^2_{H^1}\,dt +C(\|u(0)\|^2_H+1),
\end{multline}
where we have used the interpolation $\|v\|_H^2\le\|v\|_{H^{1}}\|v\|_{H^{-1}}$. This estimate
 together with \eqref{1.dth2} gives \eqref{2.h2sm}. It only remains to note that the dissipative estimate \eqref{2.h2dis} follows from \eqref{2.dtgrr} (which is used for $t\le1$ only), the smoothing estimate \eqref{2.h2sm} (which is used for estimating $\|u(t+1)\|_{H^2}$ in terms of $\|u(t)\|_H$) and the dissipative estimate \eqref{2.dis}. Thus, the lemma is proved.
\end{proof}
Thus, the solution semigroup $S(t)$ associated with equation \eqref{1.eqmain} possesses a bounded  absorbing set in $H^2$. Indeed, if we take
$$
B_{H^2}:=\{u\in H^2, \ \|u\|^2\le 2C_*\},
$$
then, due to \eqref{2.h2dis} and \eqref{2.h2sm}, for every bounded set $B$ is $H$, there exists time $T=T(B)$ such that
$$
S(t)B\subset B_{H^2}
$$
for all $t\ge T$. Moreover, taking
\begin{equation}\label{2.abs}
\Cal B=\cup_{t\ge0}S(t)B_{H^2},
\end{equation}
we obtain the absorbing set which remains bounded (and closed) in $H^2$ (and therefore is compact in $H$) and, in addition, will be semi-invariant
\begin{equation}\label{2.inv}
S(t)\Cal B\subset\Cal B.
\end{equation}
The next lemma establishes the smoothing property for the differences of two solutions which will be crucial for establishing the finite-dimensionality of the associated global attractor.

\begin{lemma}\label{Lem2.smdif} Let the above assumptions hold. Then, for any two solutions $u_1(t)$ and $u_2(t)$, the following estimate is valid:
\begin{equation}\label{2.lip}
\|u_1(t)-u_2(t)\|_H^2+\int_0^t\|u_1(t)-u_2(t)\|_{H^1}^2\,dt\le Ce^{2Lt}\|u_1(0)-u_2(0)\|^2_H
\end{equation}
for some positive constant $L$. Moreover, if $u_1(0),u_2(0)\in\Cal B$, we have also the following estimate:
\begin{equation}\label{2.smdif}
\|u_1(t)-u_2(t)\|_{H^2}\le C t^{-1}\|u_1(0)-u_2(0)\|_H,\ \ t\in(0,1].
\end{equation}
\end{lemma}
\begin{proof} Indeed, let $v(t)=u_1(t)-u_2(t)$. Then, this function solves
\begin{equation}\label{2.dif}
\Dt v+Av=l(t)v,\ \ l(t):=\int_0^1F'(su_1(t)+(1-s)u_2(t))\,dt.
\end{equation}
Taking the scalar product of this equation with $v(t)$ and using the sscond assumption of \eqref{2.Fass}, we have
$$
\frac12\frac d{dt}\|v(t)\|^2_H+\|v(t)\|^2_{H^1}\le L\|v(t)\|^2_H
$$
and the Gronwall inequality gives the desired estimate \eqref{2.lip}. To prove the smoothing property \eqref{2.smdif}, we take the scalar product of equation \eqref{2.dif} with $A^2v(t)$ and use that the trajectories $u_1(t)$ and $u_2(t)$ are uniformly bounded in $H^2$ together with the 3rd assumption of \eqref{2.Fass}. This gives
$$
\frac12\frac d{dt}\|v(t)\|_{H^2}^2+\|v(t)\|_{H^3}^2\le\|l(t)v\|_{H^1}\|v(t)\|_{H^3}\le C\|v(t)\|_{H^1}^2+\frac12\|v(t)\|_{H^3}^2.
$$
Multiplying this inequality by $t^2$, integrating in time and using \eqref{2.lip} together with the interpolation $\|v\|_{H^2}^2\le \|v\|_{H^1}\|v\|_{H^3}$, we arrive at
\begin{multline}
t^2\|v(t)\|^2_{H^2}+\int_0^tt^2\|v(t)\|^2_{H^3}\,dt\le C\int_0^t t\|v(t)\|^2_{H^2}\le\\\le \int_0^tt^2\|v(t)\|_{H^3}^2\,dt+C^2\int_0^t\|v(t)\|^2_{H^1}\,dt\le \int_0^tt^2\|v(t)\|_{H^3}^2\,dt+C\|v(0)\|^2_H.
\end{multline}
Thus, \eqref{2.smdif} is also verified and the lemma is proved.
\end{proof}
We are now ready to define the global attractor and verify its existence.

\begin{definition}\label{Def2.attr} A set $\Cal A$ is a global attractor of a semigroup $S(t)$ acting in $H$ if the following conditions are satisfied:
\par
1. $\Cal A$ is a compact set of $H$;
\par
2. $\Cal A$ is strictly invariant: $S(t)\Cal A=\Cal A$ for $t\ge0$;
\par
3. $\Cal A$ attracts the images of all bounded sets in $H$, i.e., for any bounded $B\subset H$ and any neighbourhood $\Cal O(\Cal A)$ of the attractor $\Cal A$, there exists time $T=T(B,\Cal O)$ such that
$$
S(t)B\subset\Cal O(\Cal A)
$$
for all $t\ge T$.
\end{definition}
In order to verify the existence of such an attractor, we use the following abstract result.
\begin{proposition}\label{Prop2.attr} Let the semigroup $S(t):H\to H$ satisfy the following two conditions:
\par
1. There exists a compact absorbing set $\Cal B\subset H$;
\par
2. The operators $S(t)$ are continuous on $\Cal B$ for every fixed $t\ge0$.
\par
Then, the semigroup $S(t)$ possesses a global attractor $\Cal A\subset \Cal B$ and this attractor consists of all complete bounded trajectories of $S(t)$:
\begin{equation}\label{2.k}
\Cal A=\Cal K\big|_{t=0},
\end{equation}
where $\Cal K:=\{u\in L^\infty(\R,H),\ u(t+s)=S(t)u(s),\ s\in\R,\ t\ge0\}$.
\end{proposition}
The proof of this proposition can be found, e.g., in \cite{BV}. Mention here only that $\Cal A$ is constructed as a standard  $\omega$-limit set for the absorbing set $\Cal B$:
$$
\Cal A=\omega(\Cal B):=\cap_{T\ge0}\overline{\cup_{t\ge T}S(t)\Cal B},
$$
where $\overline{C}$ is a closure of $C$ in $H$.

\begin{corollary}\label{Cor2.attr} Let the assumptions of Lemma \ref{Lem2.dis} hold. Then, the solution semigroup associated with equation $S(t)$ possesses a global attractor $\Cal A$ in $H$ which is a bounded subset of $H^2$ and consists of all solutions of \eqref{1.eqmain} which are defined for all $t\in\R$ and bounded, so \eqref{2.k} holds.
\end{corollary}
Indeed, the existence of a compact absorbing set as well as the continuity (even global Lipschitz continuity) has been verified in previous lemmas.

\begin{remark}\label{Rem2.good} On the one hand, the global attractor contains all non-trivial dynamics of the considered system and, on the other hand, it is essentially smaller than the initial phase space. Indeed, due to the compactness, the attractor is nowhere dense in the phase space (in the infinite-dimensional case) and is "almost finite-dimensional". Moreover, as will be shown in the next paragraph, it usually has finite Hausdorff and fractal dimension and, therefore, realizes (in a sense) the reduction of the considered infinite-dimensional dynamical system to the finite-dimensional one. Although, as will be discussed below, this reduction is essentially weaker than provided by the inertial manifolds theory, but as an advantage, the existence of a global attractor does not require any restrictive assumptions and is usually much easier to establish. So, verifying the existence of a global attractor is a natural "first step" in the study of the dissipative dynamics and its finite-dimensional reduction.
\end{remark}
\begin{remark}\label{Rem2.ass} We emphasize once more that our global boundedness assumptions \eqref{2.Fass} are not very realistic and posed only for simplicity. However, the analogue of the above three lemmas are verified for many classes of equations of mathematical physics, see \cite{tem,BV,CV,MirZe} and references therein. Mention also that once these lemmas and the existence of the absorbing ball in $H^2$ are established, assumptions \eqref{2.Fass} are no more restrictive since we may cut-off the nonlinearity outside of the absorbing ball and achieve the global boundedness.
\end{remark}
\subsection{Fractal/Hausdorff dimension and Man\'e projections}\label{s2.2} The aim of this paragraph is to establish and discuss the finite-dimensionality of a global attractor. To this end, we first recall the definitions of the Hausdorff and fractal dimension.
\begin{definition}\label{Def2.fractal} Let $K$ be a pre-compact set in a metric space $X$. Then, by the Hausdorff criterion, for any $\eb>0$, $K$ can be covered by finitely many $\eb$-balls in $X$. Denote by $N(K,X)$ the minimal number of such balls. By definition, the Kolmogorov's $\eb$-entropy of $K$ in $X$ is defined via
\begin{equation}\label{2.kol}
\Bbb H_\eb(K,X):=\log N_\eb(X,K)
\end{equation}
and the fractal (box-counting) dimension of $K$ in $X$ is the following number:
\begin{equation}
\dim_F(K,X):=\limsup_{\eb\to0}\frac{\Bbb H_\eb(K,X)}{\log\frac1\eb}.
\end{equation}
\end{definition}
\begin{definition}\label{Def2.H} Let $K$ be a separable subset of a metric space $X$. For any $\eb>0$ and $d\ge0$, define
$$
\mu_d(K,X,\eb):=\inf\biggl\{\sum_{n=1}^\infty \eb_i^d\,:\ \ K\subset \cup_{n=1}^\infty B(\eb_i,x_i),\ x_i\in X,\ \eb_i\le\eb\biggr\},
$$
where $B(\eb,x)=B(\eb,x,X)$ is an $\eb$-ball of $X$ centered at $x$, so the infinum is taken over all countable coverings of $K$ by balls with radiuses less than $\eb$. Obviously, this function is monotone increasing as $\eb\to0$, so the following limit exists (finite or infinite) and is called Hausdorff $d$-measure of the set $K$ in $X$:
$$
\mu_d(K,X):=\lim_{\eb\to0}\mu_d(K,X,\eb).
$$
Finally, the Hausdorff dimension of $K$ in $X$ is defined as follows:
\begin{equation}\label{2.dH}
\dim_H(K,X):=\sup\{d\,:\ \ \mu_d(K,X)=\infty\}.
\end{equation}
Roughly speaking, the difference with the fractal dimension is that the coverings with the balls of different radii are now allowed.
\end{definition}
More details about these dimensions can be found, e.g., in \cite{28}, we only state below some basic facts about them which are important for what follows and which can be deduced from the definitions as an easy exercise:
\par
1. If $K$ is compact $n$-dimensional Lipschitz submanifold of $X$, then
$$
\dim_H(K,X)=\dim_F(K,X)=n,
$$
so both of them generalize the concept of a "usual" dimension. However, they may be different and both non-integer if $K$ is not a manifold. Nevertheless, the equality
$$
\dim_H(K,X)=\dim_F(K,X)
$$
holds for many interesting examples of fractal sets. For instance, both of them equal to $\frac{\log2}{\log3}<1$ for the ternary Cantor set in $[0,1]$.
\par
2. The inequality
$$
\dim_H(K,X)\le\dim_F(K,X)
$$
always holds.
\par
3. If $\Phi:X\to Y$ is a Lipschitz map, then
\begin{equation}\label{2.dim-lip}
\dim_H(\Phi(K),Y)\le\dim_H(K,X),\ \ \dim_F(\Phi(K),Y)\le\dim_F(K,X)
\end{equation}
and, particularly, both dimensions preserve under bi-Lipschitz  homeomorphisms.
\par
4. If $K\subset\cup_{i=1}^\infty K_n$, then
\begin{equation}\label{2.countable}
\dim_H(K,X)\le\sup_{i\in\Bbb N}\dim_H(K_n,X)
\end{equation}
and, particularly, the Hausdorff dimension of any countable set is zero. This is clearly not true for the fractal dimension.
\par
5. One always has $\dim_F(\overline{K},X)=\dim_F(K,X)$. Moreover,
\begin{equation}\label{2.cart}
\dim_F(K\times K,X\times X)\le 2\dim_F(K,X)
\end{equation}
and none of these properties hold in general for the Hausdorff dimension.
\par
The next theorem is one of the main results of the attractors theory.
\begin{theorem}\label{Th2.fdim} Let the assumptions of Lemma \ref{Lem2.dis} hold. Then, the fractal dimension of the global attractor $\Cal A$ of equation \eqref{1.eqmain} is finite:
\begin{equation}
\dim_F(\Cal A,H)=\dim_F(\Cal A,H^2)<\infty.
\end{equation}
\end{theorem}
\begin{proof} Indeed, the equality of the dimensions in $H$ and $H^2$ follows from the invariance of the attractor with respect to the map $S(1)$ which, due to the smoothing property \eqref{2.smdif}, can be considered as a Lipschitz map from $H$ to $H^2$ and property \eqref{2.dim-lip}. Thus, the dimension in $H^2$ does not exceed the dimension in $H$. The inverse inequality is obvious since the identity map is Lipschitz as the map from $H^2$ to $H$. The finiteness of the dimension also follows from the parabolic smoothing property \eqref{2.smdif} and the following standard proposition.
\begin{proposition}\label{Prop2.dim} Let $H$ and $H^2$ be two Banach spaces such that $H^2$ is compactly embedded to $H$ and let $\Cal A$ be a bounded set in $H$. Assume also that $\Cal A$ is invariant with respect to some map $S$ ($S(\Cal A)=\Cal A$) which satisfies the following property:
\begin{equation}\label{2.smabs}
\|Sx_1-Sx_2\|_{H^2}\le L\|x_1-x_2\|_H,\ \ \forall x_1,x_2\in H.
\end{equation}
Then, the fractal dimension of $\Cal A$ in $H$ is finite.
\end{proposition}
\begin{proof} As usual, in order to prove the finiteness of fractal dimension, we need to construct the proper $\eb$-coverings of $\Cal A$ with sufficiently small number of balls. We do that by the inductive procedure. Since $\Cal A$ is bounded, there is $R>0$ and $x_0\in \Cal A$ that one $R$-ball $B(R,x_0,H)$ covers $\Cal A$. Assume that the covering of $\Cal A$ by $R2^{-n}$-balls in H is already constructed and let $N_n$ be the number of balls in that covering:
$$
\Cal A\subset \cup_{i=1}^{N_n}B(R2^{-n},x_i,H),\ \ x_i\in\Cal A.
$$
Then, due to \eqref{2.smabs}, the $H^2$-balls of radii $LR2^{-n}$ centered at points $S(x_i)$ cover the set $S\Cal A=\Cal A$:
$$
\Cal A\subset\cup_{i=1}^{N_n}B(LR2^{-n},Sx_i,H^2).
$$
We now utilize the compactness of the embedding $H^2\subset H$. Since every of that $H^2$-balls is pre-compact in $H$, we may cover every of them by
\begin{multline*}
N_{R2^{-n-2}}(B(LR2^{-n},Sx_i,H^2),H)=\\=N_{R2^{-n-2}}(B(LR2^{-n},0,H^2),H)=N_{1/4L}(B(1,0,H^2),H):=N
\end{multline*}
balls of radii $R2^{n-2}$. Moreover, increasing the radii by the factor of 2, we may assume that the centers of these balls belong to $\Cal A$ again. Thus, we have constructed the $R2^{-n-1}$-net for $\Cal A$ such that the number of balls $N_{n+1}\le N N_n$. Thus, by induction, we have constructed
the $\eb_n:=R2^{-n}$-coverings $\Cal A$ with the number of balls $N_n:=N^n$. Let now $\eb>0$ be small and arbitrary. Fix $n\in\Bbb N$ such that
$$
\eb_n\le\eb\le\eb_{n+1}.
$$
Then, $n\sim(\log 2)^{-1}\log\frac1\eb$ and
$$
\Bbb H_{\eb}(\Cal A,H)\le\Bbb H_{\eb_{n+1}}(\Cal A,H)\le (n+1)\log N\sim \frac{\log N}{\log2}\log\frac1\eb.
$$
We see that $\dim_F(\Cal A,H)\le \frac{\log N}{\log2}$ and the proposition is proved.
\end{proof}
To deduce the finite-dimensionality of the attractor $\Cal A$ from this proposition, it is sufficient to apply it with $S=S(1)$ and use \eqref{2.smdif}. Thus, the theorem is also proved.
\end{proof}
\begin{remark}\label{Rem2.opt} Although the given proof of the finite-dimensionality of the attractor is probably the simplest and most transparent, very often the alternative scheme based on the so-called volume contraction method is used since as believed it gives better estimates for the fractal dimension in terms of physical parameters of the system considered, see \cite{tem} for the details. The above given scheme is however typical for the so-called {\it exponential} attractors theory and has an advantage in comparison with the volume contraction method since the differentiability of the semigroup with respect to the initial data is not required, so it can be successfully applied, e.g.,  in the case of sigular/degenerate equations where such differentiability does not take place, see \cite{MirZe} and references therein for details.
\end{remark}
We conclude the paragraph by discussing the finite-dimensional reduction induced by the above finite-dimensionality of the attractor and the following classical theorem of Man\'e, see \cite{man}, which is the analogue of the so-called Whitney embedding theorem of compact manifolds to $\R^N$, see \cite{dub} for the details.

\begin{theorem}\label{Th2.man} Let a compact set $\Cal A\subset H$ has finite fractal dimension $\dim_F(\Cal A)$ in a Hilbert space $H$. Let also
$N\in\Bbb N$ be such that
\begin{equation}\label{2.big}
N>2\dim_F(\Cal A).
\end{equation}
Then, there exists an $N$-dimensional plane $H_N\subset H$ such that the orthoprojector $P:H\to H_N$ on that plane is one-to-one on $\Cal A$ and, therefore, is a homeomorphism of $\Cal A$ to the finite-dimensional set $P\Cal A\subset H_N=\R^N$. Moreover, the set of all $N$-dimensional projectors with this property is generic in the proper sense.
\end{theorem}
\begin{proof}[Sketch of the proof] Assume for simplicity that $\Cal A$ is already embedded into $H=\R^N$ for some big $N$ satisfying
\begin{equation}\label{2.bigg}
N>2\dim_F(\Cal A)+1
\end{equation}
and show how the dimension $N$ can be reduced. Namely, we want to show that for almost all $(N-1)$-dimensional planes in $H=\R^N$, the orthoprojector
$P$ to it will be one-to-one on $\Cal A$. Such projector $P=P_l$ is completely determined by the line $l\in \Bbb R\Bbb P^{N-1}$ ($(N-1)$-dimensional projective space),  orthogonal to the $(N-1)$-dimensional plane. For any $x,y\in\Cal A$, $x\ne y$, define the map
$$
\operatorname{Line}: \Cal A^2_0:=\Cal A\times\Cal A\backslash\{(x,x), \ x\in\Cal A\}\to  \Bbb R\Bbb P^{N-1}
$$
by formula $\operatorname{Line}(x,y)=$ "the direction of line passing through $x$ and $y$". Then, the "exceptional" projectors which are not one-to-one on $\Cal A$ are determined by the points of $\Bbb R\Bbb P^{N-1}$ belonging to $\operatorname{Line}(\Cal A_0^2)$. We claim that
\begin{equation}\label{2.Hdim}
\dim_H(\operatorname{Line}(\Cal A_0^2),\Bbb R\Bbb P^{N-1})\le 2\dim_F(\Cal A)
\end{equation}
and, since $\dim(\Bbb R\Bbb P^{N-1})=N-1$ and \eqref{2.bigg} holds, almost all directions of $\Bbb R\Bbb P^{N-1}$ are not exceptional and the dimension reduction works. Indeed, the line from $\operatorname{Line}(\Cal A_0^2)$ is uniquely determined by the pair $(x,y)\in\Cal A^2$ but the map is not Lipschitz continuous since the line passing through $x$ and $y$ depends very sensitively on $x$ and $y$ when $\|x-y\|$ is small. To overcome this difficulty, we introduce the sets
$$
\Cal A_n^2:=\{(x,y)\in\Cal A_0^2\,:\ \|x-y\|\ge1/n\}.
$$
Then, as not difficult to verify, $\operatorname{Line}$ is globally Lipschitz on $\Cal A_n$ and, therefore,
$$
\dim_H(\operatorname{Line}(\Cal A_n^2),\Bbb R\Bbb P^{N-1})\le \dim_H(\Cal A_n^2,H\times H)\le \dim_F(\Cal A\times\Cal A,H\times H)\le 2\dim_F(\Cal A).
$$
On the other hand,
$$
\operatorname{Line}(\Cal A_0^2)=\cup_{n=1}^\infty\operatorname{Line}(\Cal A_n^2)
$$
and \eqref{2.Hdim} follows from the property \eqref{2.countable}. Thus, the reduction of the dimension works and we have proved the theorem in the
finite-dimensional case. The infinite-dimensional case can be treated analogously and the same dimension reduction works, but one should be a bit more careful and use, for instance, the Zorn lemma in order to guarantee the "convergence".
\end{proof}
\begin{remark}\label{Rem2.holder} In order to give the flavor of the theory, we state and prove the Man\'e projection theorem in its simplest form although much stronger  versions are available nowadays. In particular, the underlying space may be only Banach, "generic" set of projectors is usually understood in terms of measure theory and the so-called prevalence and the Borel-Cantelli lemma is used instead of the Zorn lemma, see \cite{hunt,28} and references therein.
\par
 Recall also that $\Bbb R\Bbb P^{N-1}=\Bbb S^{N-1}/\Bbb Z_2$, where $\Bbb S^{N-1}$ is an $(N-1)$-dimensional sphere and $\Bbb Z_2$ is a symmetry $x\to-x$, so the map $\operatorname{Line}$ can be written in the form
$$
\operatorname{Line}(x,y)=\frac{x-y}{\|x-y\|},\ \ x,y\in\Cal A.
$$
Thus, the map  is actually defined on the set $\Cal A-\Cal A\subset H$ rather than on $\Cal A\times \Cal A$. That explains why various dimensions of the set $\Cal A-\Cal A$ are often involved in Man\'e type theorems, see \cite{28} for more details.
\par
Moreover, as also known  the projector in Theorem \eqref{Th2.man} can be chosen (again generically) in such way that the inverse will be not only continuous, but {\it H\"older continuous} on $P\Cal A$ with some H\"older exponent $\kappa>0$. The value of the exponent $\kappa$ depends on how good the set $\Cal A$ can be approximated by finite-dimensional planes (Lipschitz manifolds) and is determined by the so-called thickness exponent or Lipschitz deviation, see \cite{3,foo,hunt,23,26} and references therein. In fact, the attractor $\Cal A$ of equation \eqref{1.eqmain} can be  nicely approximated by finite-dimensional Lipschitz manifolds (at least when $\lambda_N$ grow faster than $\log N$), see the so-called approximative inertial manifolds of exponential order in e.g., \cite{tem}, so the Lipschitz deviation of that attractor is zero and, by this reason, the H\"older exponent $\kappa$ in the H\"older Man\'e projection theorem can be made arbitrary close to one if $N$ is large enough, see  \cite{28}.
\end{remark}
Thus, Theorems \ref{Th2.fdim} and \ref{Th2.man} allow us to project one-to-one the attractor $\Cal A$ to a finite dimensional set $\overline{\Cal A}\subset \R^N$ and the inverse projector $P^{-1}$ can be made H\"older continuous. Projecting the solution semigroup
\begin{equation}\label{2.sem-finite}
\bar S(t):=P\circ S(t)\circ P^{-1},\ \ \bar S(t):\overline{\Cal A}\to\overline{\Cal A},
\end{equation}
to $\overline{\Cal A}$, we see that the dynamics of \eqref{1.eqmain} on the attractor $\Cal A$ is conjugated by the H\"older continuous homeomorphism to the dynamics $\bar S(t)$ on the subset $\overline{\Cal A}$ of $\R^N$. That explains why, at least on the level of topological dynamics, there are no principal difference between the dynamics generated by \eqref{1.eqmain} and the classical finite-dimensional dynamics, see \cite{kap,katok} for more details.
\par
Moreover, the reduced dynamics on $\overline{\Cal A}$ can be at least formally determined by a system of ODEs analogous to the inertial form \eqref{1.inertial}. Indeed, since the attractor $\Cal A$ has finite dimension in $H$,  we may fix the Man\'e projector $P$ such that  $P^{-1}$ is H\"older continuous as a map from $\overline{\Cal A}$ to $H$  and, therefore, for any trajectory $u(t)\in \Cal A$ of equation \eqref{1.eqmain}, the function $v(t):=Pu(t)$ solves
\begin{equation}\label{2.inertial}
\frac d{dt}v(t)=-P AP^{-1}v(t)+PF(P^{-1}v(t)):=W(v(t))
\end{equation}
which is a system of ODEs on $\overline{\Cal A}\subset\R^N$. Moreover, as not difficult to see using that the identity map on the attractor $\Cal A$ is continuous as the map of $H$ to $H^2$ (due to compactness of $\Cal A$ in $H^2$), that the vector field $W$ is continuous on $\overline{\Cal A}$. Slightly more delicate arguments based, e.g., on the "almost equivalence" of the $H$ and $H^2$ norms on the attractor, see Proposition \ref{Prop2.Alip} below, show that $W$ is also H\"older continuous and the H\"older exponent can be chosen arbitrarily close to one, see \cite{EFNT,28} and references therein.
\par
Finally, the inertial form \eqref{2.inertial} initially defined on $\overline{\Cal A}\subset\R^N$ can be extended on the whole $\R^N$ using, e.g., the following extension result.

\begin{proposition}\label{Prop2.ext} Let $H_1$ and $H_2$ be two Hilbert spaces, $S\subset H_1$ be an arbitrary set and a map $W: S\to H_2$ satisfy
\begin{equation}\label{2.hol}
\|W(x_1)-W(x_2)\|_{H_2}\le L\|x_1-x_2\|^\alpha_{H_1},\ \ x_1,x_2\in S
\end{equation}
for some $\alpha\in(0,1]$ and $L>0$. Then, the map $W$ can be extended to the map from $H_1$ to $H_2$ in such way that inequality \eqref{2.hol} (with the same constants) will hold for all $x_1,x_2\in H_1$.
\end{proposition}
Indeed, in the simplest case $\alpha=1$ and $H_2=\R$, one of possible  extensions is given by the explicit formula
$$
\tilde W(x):=\sup_{y\in S}\,\biggl\{W(y)-L\|x-y\|\biggr\}.
$$
In a general case one should be a bit more accurate, see, e.g.,  \cite{min} for the detailed proof.

\begin{corollary}\label{Cor2.emb} Under the assumptions of Theorem \eqref{Th2.fdim} the attractor $\Cal A$ can be embedded to the finite-dimensional manifold which is a graph of a H\"older continuous function
$$
\Phi: H_+=\R^N\to H_-=H_+^\bot,\ \Cal A\subset \Cal M:=\{u_++\Phi(u_+),\ u_+\in H_+\}
$$
and
$$
\|\Phi(x_1)-\Phi(x_2)\|_H\le k\|x_1-x_2\|_H^\alpha
$$
for some $k>0$ and $\alpha>0$ which can be made arbitrarily close to one.
\end{corollary}
Indeed, let $P:\Cal A\to\overline{\Cal A}\subset H_+=\R^N$ be the Man\'e projection. Then applying Proposition \ref{Prop2.ext} to the map
$\Phi(u_+):=(1-P)P^{-1}u_+$, $S=\overline{\Cal A}$, $H_1:=H_+$ and $H_2:=(1-P)H$, we construct the desired manifold and prove the corollary.

\begin{remark}\label{Rem2.differ} Although the finite-dimensional manifold constructed in Corollary \ref{Cor2.emb} looks similar to the inertial manifold \eqref{1.graph} and can be indeed interpreted as a "non-smooth" version of the inertial manifold, there is a principal difference between these two cases. Namely, in the case of inertial manifolds, we can guarantee that the manifold $\Cal M$ is invariant also in the neighbourhood of the attractor and that the attractor of the inertial form \eqref{1.inertial} is {\it exactly} the same as for the initial problem \eqref{1.eqmain}. In the case of the inertial form constructed from the Man\'e projections, the manifold is in general not invariant and the attractor of the extended inertial form \eqref{2.inertial} may be larger than the initial attractor. Constructing of the inertial forms with continuous vector field $W$ (or reduced semigroups) with exactly the same attractor is a non-trivial open problem even on the level of discrete time, see \cite{rob-top} and references therein for some partial results in that direction.
\par
One more essential drawback of the inertial form \eqref{2.inertial} is that we a priori know only that the vector field $W$ is H\"older continuous which is not enough even for the uniqueness. So we are unable restore the dynamics just by solving the reduced ODEs and need to use the initial system at least for the selection of the proper solution, see Example \ref{Ex2.non-unique} showing that extra pathological solutions which are not related with the initial system indeed may appear under the non-smooth projection. These drawbacks stimulate the further study of the reduction problem and, in particular, various attempts to improve the regularity of the inverse Man\'e projector which will be considered in next sections.
\end{remark}
\subsection{Bi-Lipschitz projections and Romanov theory}\label{s2.3} The aim of this paragraph is to study the cases when the attractor $\Cal A$ of problem \eqref{1.eqmain} possesses Man\'e projections with {\it Lipschitz} continuous inverse. The main result is the following theorem which is slight reformulation of the results obtained by Romanov, see \cite{rom-th, rom-th1}.

\begin{theorem}\label{Th2.rom} Let the assumptions of Theorem \ref{Th2.fdim} hold. Then the following assertions are equivalent:
\par
1. The attractor $\Cal A$ of problem \eqref{1.eqmain} possesses at least one Man\'e projection in $H$ such that the inverse $P^{-1}:\overline{\Cal A}\to\Cal A$ is  Lipschitz continuous.
\par
2. The solution semigroup $S(t):\Cal A\to\Cal A$ can be extended for negative times to the Lipschitz continuous group acting on the attractor. In particular,
\begin{equation}\label{2.mlip}
\|S(-1)u_1-S(-1)u_2\|_H\le L\|u_1-u_2\|_H,
\end{equation}
for all $u_1,u_2\in \Cal A$ and some constant $L$.
\par
3. The $H^2$ and $H$ norms are equivalent on the attractor:
\begin{equation}\label{2.equiv}
l\|u_1-u_2\|_H\le\|u_1-u_2\|_{H^2}\le L\|u_1-u_2\|_H
\end{equation}
for all $u_1,u_2\in \Cal A$ and some positive $l$ and $L$.
\par
4. There exists $N\in \Bbb N$ such that the spectral projector $P_N$ on first $N$ eigenvalues of $N$ satisfies
\begin{equation}\label{2.plip}
\|Q_N(u_1-u_2)\|_H\le L\|P_N(u_1-u_2)\|_H,\ \ Q_N:=1-P_N
\end{equation}
for all $u_1,u_2\in\Cal A$ and, therefore, $\Cal A$ can be embedded into the Lipschitz manifold which is a graph of a Lipschitz continuous function over $H_+:=P_NH$.
\end{theorem}
\begin{proof} 1) $\Rightarrow$ 2). To verify this, we first note that the set of Man\'e projectors with Lipschitz inverse is open. Indeed, let $P$ be such a projector and let $Q:H\to H$ be an arbitrary linear continuous map. Since $P^{-1}$ is Lipschitz on the attractor, we have
$$
\|P(u_1-u_2)\|_H\ge\kappa\|u_1-u_2\|,\ u_1,u_2\in \Cal A,
$$
for some positive $\kappa$. Therefore,
$$
\|(P+\eb Q)(u_1-u_2)\|_H\ge\|P(u_1-u_2)\|_H-\eb\|Q(u_1-u_2)\|_H\ge(\kappa-\eb\|Q\|)\|u_1-u_2\|_H,\ \ u_1,u_2\in \Cal A
$$
and we see that $P+\eb Q$ is one-to-one and bi-Lipschitz on the attractor if $\eb$ is small enough. Recall that any finite-dimensional orthoprojector $P$ in $H$ has a form
\begin{equation}\label{2.proj}
Pu:=\sum_{i=1}^N(u,\xi_i)\xi_i
\end{equation}
for some orthonormal system $\{\xi_i\}_{i=1}^N$ in $H$. Since the set of Man\'e projectors with Lipschitz inverse is open and $H^2$ is dense in $H$, we can assume without loss of generality that the projector $P$ has the form \eqref{2.proj} where all vectors $\xi_i\in H^2$. Then, obviously,
$$
\|Pu\|_{H^{-1}}\le C\|u\|_{H^{-1}}\sum_{i=1}^N\|\xi_i\|_{H^1}\|\xi_i\|_{H^{-1}}\le C_N\|u\|_{H^{-1}}
$$
and $P\in\Cal L(H^{-1},H^{-1})$. Using this together with the fact that all norms are equivalent on the finite-dimensional space $H_+:=PH$, for every $u_1,u_2\in H$, we have
\begin{equation}\label{2.m10eq}
\|u_1-u_2\|_H\le \kappa^{-1}\|P(u_1-u_2)\|_H\le C\kappa^{-1}\|P(u_1-u_2)\|_{H^{-1}}\le C_1\|u_1-u_2\|_{H^{-1}}.
\end{equation}
Thus, the $H^{-1}$ and $H$ norms on the attractor are equivalent. We are now ready to check \eqref{2.mlip}. Indeed, let $u_1(t)$ and $u_2(t)$ be two trajectories on the attractor and $v(t):=u_1(t)-u_2(t)$. Then $v(t)$ solves \eqref{2.dif}. Multiplying this equation by $A^{-1}v(t)$ and using that $F$ is globally Lipschitz in $H$ together with \eqref{2.m10eq}, we have
$$
\frac12\frac d{dt}\|v(t)\|_{H^{-1}}^2\ge -C\|v(t)\|_H^2\ge-C_1\|v(t)\|^2_{H^{-1}}.
$$
Applying the Gronwall inequality and using the equivalence of the $H^{-1}$ and $H$ norms again, we get
$$
\|u_1(-t)-u_2(-t)\|_H\le Ce^{-K t}\|u_1(0)-u_2(0)\|_H
$$
and the implication 1) $\Rightarrow$ 2) is verified.
\par
Implication 2) $\Rightarrow$ 3) is an immediate corollary of smoothing property \eqref{2.smdif}. Indeed,
$$
\|u_1(0)-u_2(0)\|_{H^2}\le C\|u_1(-1)-u_2(-1)\|_H\le C_1\|u_1(0)-u_2(0)\|_H,\ \ u_1(0),u_2(0)\in\Cal A.
$$
and the left inequality of \eqref{2.equiv} is obvious.
\par
Implication 3) $\Rightarrow$ 4) follows from the estimate
\begin{multline*}
\|Q_N(u_1-u_2)\|_H\le\lambda_{N+1}^{-1}\|Q_N(u_1-u_2)\|_{H^2}\le\lambda_{N+1}^{-1}\|u_1-u_2\|_{H^2}\le\\\le L\lambda_{N+1}^{-1}\|u_1-u_2\|_H\le L\lambda_{N+1}\|P_N(u_1-u_2)\|_H+L\lambda_{N+1}^{-1}\|Q_N(u_1-u_2)\|_H
\end{multline*}
if $L\lambda_{N+1}^{-1}<1$ which is always true for sufficiently large $N$.
\par
Finally, the implication 3) $\Rightarrow$ 4) is obvious since $P_N$ is a Man\'e projector with Lipschitz inverse if \eqref{2.plip} is satisfied. Thus, the theorem is proved.
\end{proof}
\begin{remark}\label{Rem2. surprice} As verified in the theorem, the existence of {\it any} finite-dimensional Man\'e projector with Lipschitz inverse implies that {\it all spectral} projectors $P_N$ will possess this property if $N$ is large enough and, therefore, the existence of any inertial form with Lipschitz continuous non-linearity implies that more simple inertial form with the spectral projectors also can be used. Moreover, as shown in \cite{rom-th}, the existence of a compact $C^1$-submanifold containing the global attractor also implies the assertions of Theorem \ref{Th2.rom} and, in particular, the existence of a manifold containing the attractor which is a graph of a Lipschitz function over the spectral space $P_NH$ (see also \cite{rom-th,rom-th1,MLW} and references therein for more conditions which guarantee the existence of finite-dimensional  bi-Lipschitz embeddings of the attractor). We however do not know whether or not something similar is true for the general case of H\"older continuous Man\'e projections.
\par
Mention also that the 4th assertion of the proved theorem looks close to the situation when the inertial manifold exists. Indeed, using the extension result of Proposition \ref{Prop2.ext}, we construct a Lipschitz map $\Phi: H_+:=P_NH\to Q_N H$ such that its graph contains the attractor $\Cal A$ and the dynamics on it is described by the inertial form
\begin{equation}\label{2.inn}
\frac d {dt}u_+=-Au_++P_NF(u_++\Phi(u_+)),\  \ u_+\in H_+:=P_NH.
\end{equation}
The difference is however that the manifold thus constructed is {\it not invariant} with respect to the solution semigroup $S(t)$ outside of the attractor and, in particular, the attractor of the inertial form \eqref{2.inn} maybe {\it larger} than the initial attractor. On the other hand, up to the moment, no examples where such finite-dimensional reduction is verified and an inertial manifold does not exist are known (see below the example where the inertial manifold existence is not known so far, but bi-Lipschitz embeddings work), so it is unclear whether or not bi-Lipschitz embeddings are indeed more general than inertial manifolds.
\par
Note also that the assumption of global Lipschitz continuity of $F$ as a map from $H$ to $H$ is posed only for simplicity and all the above results remain true (with the proofs repeated almost word by word) if we assume its Lipschitz continuity, say, from $H$ to $H^{-\beta}$ with $\beta<2$. We will use this "generalization" in the example below.
\end{remark}

\begin{example}\label{Ex2.grad} Consider the 1D reaction-diffusion equation with the nonlinearity containing spatial derivatives:
\begin{equation}\label{2.gradRDE}
u_t=u_{xx}-f(x,u,u_x)
\end{equation}
on the interval $\Omega=[0,\pi]$ endowed, say, by the Dirichlet boundary conditions. The classical example here is a Burgers equation ($f=uu_x+g(x)$) although other type of nonlinearities are also interesting. Then, under the natural dissipativity assumptions on the nonlinearity $f$, for instance due to the maximum principle, one can assume that $f$ is smooth and
$$
f(x,u,0)\cdot \sgn(u)\ge -C,\ \ |f(x,u,p)|\le Q(u)(1+|p|^2)
$$
for some monotone increasing $Q$ and constant $C$ independent of $x$ and $u$, equation \eqref{2.gradRDE} is well-posed and dissipative, say, in $C(\Omega)$ and possesses an absorbing set which is bounded in $C^1(\Omega)$, see, e.g., \cite{LSU}. By this reason, cutting off the nonlinearity, we may assume that $f$, $f'_u$ and $f'_{u_x}$ are globally bounded. Then, problem \eqref{2.gradRDE} is well-posed in $H=L^2(\Omega)$ and the nonlinearity $F(u):=f(x,u,u_x)$ is globally Lipschitz as a map from $H^1:=H^1_0(\Omega)$ to $H$. The operator $A:=-\frac {d^2}{dx^2}$  is self-adjoint and positive (with $D(A):=H^2(\Omega)\cap H^1_0(\Omega)$) and its eigenvalues are $\lambda_N:=N^2$. According to Theorem \ref{Th1.dmain}, for the straightforward existence of an inertial manifold, we need to verify the spectral gap condition with $\beta=-1$:
$$
\frac{(N+1)^2-N^2}{N+N+1}=1>L
$$
and we see that it is violated if $L\ge1$ although we are, in a sense, in the borderline case (it would be satisfied for all $L$  if $\beta$ would be greater than $-1$). Thus, Theorem \ref{Th1.dmain} does not guarantee the existence of an inertial manifold.
\par
Let us show that the Romanov theory works for that equation. Indeed, although formally the assumptions of Theorem \ref{Th2.rom} are not satisfied
(since $F$ now decreases the smoothness), it is easy to verify that all the estimates of Lemmas \ref{Lem2.dis}, \ref{Lem2.smo} and \ref{Lem2.smdif} hold for that case as well and, repeating word by word the proof of Theorem \ref{Th2.rom}, we see that it remains valid for that case as well.
\par
We check below that the second assertion of that theorem is satisfied for equation \eqref{2.gradRDE}. Indeed, let $u_1(t)$ and $u_2(t)$, $t\in\R$, be two trajectories belonging to the attractor $\Cal A$ and $v(t):=u_1(t)-u_2(t)$. Then, $u_i(t)$  are smooth and remain bounded as $t\to-\infty$ and
$v(t)$ solves
\begin{multline}\label{2.var}
v_t=v_{xx}+l_1(t)v+l_2(t)v_x, \ l_1(t):=\int_0^1f'_u(x,su_1+(1-s)u_2,s\partial_xu_1+(1-s)\partial_xu_2)\,ds, \\ l_2(t):=\int_0^1f'_{u_x}(x,su_1+(1-s)u_2,s\partial_xu_1+(1-s)\partial_xu_2)\,ds.
\end{multline}
Moreover, since the attractor $\Cal A$ is smooth, we know that functions $l_1(t)$ and $l_2(t)$ are also smooth and bounded, for instance,
\begin{equation}\label{2.un}
\|l_i\|_{C^1(\R\times\Omega)}\le K,\ \ i=1,2
\end{equation}
and this estimate is uniform with respect to $u_1,u_2\in\Cal A$. The key idea here is that the "bad" term in \eqref{2.var} can be removed by the proper linear transform. Namely, let
$$
w(t,x):=e^{\frac12\int_0^x l_2(t,x)\,dx}v(t,x).
$$
Then, on the one hand, due to \eqref{2.un}
\begin{equation}\label{2.equivw}
C^{-1}\|v(t)\|_{H}\le\|w(t)\|_{H}\le C\|v(t)\|_H,
\end{equation}
where $C>0$ is uniform with respect to $u_1,u_2\in\Cal A$, so we may check the backward Lipschitz continuity for $w(t)$ instead of $v(t)$. On the other hand, the new function $w$ solves
\begin{equation}\label{2.goood}
w_t=w_{xx}+\(l_1(t)+\frac14l_2(t)^2-\frac12\partial_x l_2(t)+\int_0^x \partial_t l_2(t)\,dx\)w=w_{xx}+R(t,x)w.
\end{equation}
Moreover, due to \eqref{2.un},
\begin{equation}
\sup_{u_1,u_2\in\Cal A}\sup_{t\in\R}\|R(t)\|_{L^\infty(\Omega)}\le L
\end{equation}
for some finite constant $L$.
\par
Let us now fix $N$ such that $(N+1)^2-N^2=2N+1>2L$. Then, the spectral gap condition of Theorems \ref{Th1.main} and \ref{Th1.manex} is satisfied and, therefore, there is a linear map $M_+(t):H_+:=P_NH\to H_-$ such that $\|M_+(t)\|_{\Cal L(H_+,H_-)}\le C$ and the graph $\Cal M_+(t)$ of this map is an invariant subspace for equation \eqref{2.goood} (see Remark \ref{Rem1.gen} for the extension of the inertial manifold theorem to the non-autonomous case). Moreover, since $w(t)$ remains bounded as $t\to-\infty$, this trajectory belongs to $\Cal M_+(t)$ for all $t\in\R$. The inertial form corresponding to that manifold is a linear finite-dimensional system of ODEs on $H_+$, so the associated flow is Lipschitz continuous backward in time:
\begin{equation}\label{2.last}
\|w(-1)\|_H\le C\|w(0)\|_H
\end{equation}
and the constant $C$ depends only on $L$ and the norm of $M_+$. Since both of them are uniform with respect to $u_1,u_2\in\Cal A$, \eqref{2.last} is also uniform with respect to $u_1,u_2\in\Cal A$ and \eqref{2.mlip} is verified. Thus, by Theorem \ref{Th2.rom}, we have bi-Lipschitz Man\'e projections for equation \eqref{2.gradRDE} and the Romanov theory indeed works. This result has been firstly established in \cite{rom-th}, see also \cite{kuk1}.
\end{example}
\begin{remark}\label{Rem2.manifold} The example of a 1D reaction-diffusion equation considered above is the key example where the Romanov theory works, but the existence of an inertial manifold is not verified yet, so it potentially illustrates the advantages of the theory. However, it is hard to believe that the inertial manifold does not exists here, it looks like one just need a little more advanced theory in order to verify the existence of an inertial manifold. Indeed, as it actually shown in Example \ref{Ex2.grad} the equation of variations which corresponds to \eqref{2.gradRDE}
$$
v_t=v_{xx}-f'_u(x,u,u_x)v+f'_{u_x}(x,u,u_x)v_x
$$
possesses a family of invariant cones $K_u^+$ defined with the help of function $w$:
$$
\|Q_N\(e^{-1/2\int_0^xf'_{u_x}\,dx}v\)\|\le \|P_N\(e^{-1/2\int_0^xf'_{u_x}\,dx}v\)\|.
$$
The only difference with the situation considered in paragraph \ref{s.cone} is that now the cones are not the same for every point of the phase space, but depend explicitly on this point $K^+=K^+_u$. As known, the explicit dependence of cones on the point of the phase space is not a big obstruction to the existence of invariant manifolds, see e.g., \cite{fen,bates} especially if this dependence is regular, so we may expect that the proper modification of Theorem \ref{Th1.maincon} will be enough to verify the existence of an inertial manifold for equation \eqref{2.gradRDE}.
\end{remark}
\subsection{Log-Lipschitz conjecture and the uniqueness problem}\label{s2.4}
As we have seen before, the inertial form of equation \eqref{1.eqmain} built up based on the Man\'e projections of the attractor is a finite-dimensional system of ODEs with at least {\it H\"older continuous} vector field $W$. However, if the vector field is  only H\"older continuous, we cannot guarantee even the uniqueness of solutions of the reduced system of ODEs and that is a serious drawback. Indeed, in that situation, in order to be able to select the "correct" solution of the reduced equations, we need either to use the initial non-reduced system (which makes the reduction senseless) or to endow the reduced inertial form by the proper "selection principle" to restore the uniqueness which is also not an easy task and, to the best of our knowledge, no results in that direction are available so far. Mention also that without solving the uniqueness problem it is difficult to approximate the non-smooth inertial form by smooth ones (e.g., for doing the numerical simulations). The next simple example demonstrates that the uniqueness indeed can be lost under the non-smooth projections

\begin{example}\label{Ex2.non-unique} Let $H=\R^2$ and the attractor $\Cal A$ is a segment on the curve $y=x^3$ which corresponds to $x\in[-,1,1]$. Assume also that the dynamics on the curve near $x=0$ is given by
$$
x'(t)=1,\ y'(t)=3x^2.
$$
It is not difficult to construct a smooth system of ODEs with such attractor, so we leave this to the reader as an elementary exercise.
\par
Consider now the Man\'e projection $P$ which is the orthoprojector to the $y$-axis: $P(x,y)=y$. This projector is one-to-one on the attractor, but  the inverse is not smooth and only H\"older continuous with H\"older exponent $1/3$. The corresponding inertial form near $x=0$ reads
\begin{equation}\label{2.bad-inertial}
y'(t)=3x^2(t)=3\sqrt[3]{y^2(t)}.
\end{equation}
We see that the inertial form \eqref{2.bad-inertial} is indeed only H\"older continuous and the uniqueness is indeed lost despite the uniqueness of the initial non-reduced system. In fact, a new artificial equilibrium at $y=0$ is born and together with the correct solution $y(t)=t^3$, we have a family of wrong solutions including $y(t)\equiv0$.
\end{example}
\par
Thus, the uniqueness of solutions for the reduced equations \eqref{2.inertial} is indeed non-trivial and it is  one of the key problems in the theory of Man\'e projections. Of course, the problem will be solved if a bi-Lipschitz Man\'e projection of the attractor exists. However, according to Theorem \ref{Th2.rom} this requires rather strong conditions to be satisfied and as will be shown in the next section they are not always satisfied. Moreover, Example \ref{Ex2.grad} is one of very few examples where this theory works and there are no more or less general methods to verify assumptions of Theorem \ref{Th2.rom}.
\par
The aim of that paragraph is to discuss the relatively recent attempts to solve the uniqueness problem exploiting the so-called log-Lipschitz Man\'e projections.  We recall that a map $W:X\to Y$ between two metric space $X$ and $Y$ is $\alpha$-log-Lipschitz (=log-Lipschitz with the exponent $\alpha>0$) if
\begin{equation}\label{2.log}
d(W(x),W(y))\le C d(x,y)\(\log \frac {2K}{d(x,y)}\)^\alpha,\ \ x,y\in X,\ d(x,y)\le K
\end{equation}
for some positive $K$ and $C$. The following version of the so-called Osgood theorem explains why the log-Lipschitz functions may be important for solving the uniqueness problem.

\begin{proposition}\label{Prop2.osg} Let the vector field $W:\R^N\to\R^N$ be log-Lipschitz with exponent $\alpha\le1$. Then, the solution of the corresponding system of ODEs
\begin{equation}\label{2.ODE-un}
\frac d{dt}v=W(v),\  \ v(0)=v_0
\end{equation}
is unique and depends continuously on the initial data.
\end{proposition}
\begin{proof}[Sketch of the proof] It is enough to verify the uniqueness for the worst case $\alpha=1$ only. Indeed, let $v_1(t)$ and $v_2(t)$ be two solutions of \eqref{2.ODE-un} and $\bar v(t):=v_1(t)-v_2(t)$. Then, this function solves
$$
\frac d{dt}\bar v=W(v_1)-W(v_2).
$$
Multiplying this equation by $\bar v$, denoting $y(t):=\|\bar v(t)\|^2$ and using the log-Lipschitz assumption on $W$, we have
\begin{equation}\label{2.logest}
y'(t)=2(\bar v'(t),v(t))=2(W(v_1)-W(v_2),v_1-v_2)\le C\|\bar v\|^2\log\frac {2K}{\|\bar v\|}\le Cy(t)\log\frac {2K}{y(t)}.
\end{equation}
Solving the corresponding ODE, we get
$$
y(t)\le 2K\(\frac{y(0)}{2K}\)^{e^{-Ct}}
$$
which give both uniqueness and the continuous dependence on the initial data.
\end{proof}
\begin{remark} Analogous arguments show the continuous dependence on the function $W$ as well, so under the assumptions of Proposition \ref{Prop2.osg}, there are no problems with approximation the vector field $W$ by smooth ones. Note also that the restriction $\alpha\le1$ is sharp.
\end{remark}
The next result based on the logarithmic convexity  for elliptic/parabolic equations, see \cite{AN,MiZe} and references therein indicates that the log-Lipschitz continuity naturally appears under the study of equations \eqref{1.eqmain}.

\begin{proposition}\label{Prop2.Alip} Let the assumptions of Theorem \ref{Th2.fdim} hold. Then, the $H$ and $H^2$-norms are almost equivalent on the global attractor $\Cal A$ in the following sense:
\begin{equation}\label{2.Alog}
\|u_1-u_2\|_{H^2}\le C\|u_1-u_2\|_H\(\log\frac {2K}{\|u_1-u_2\|_H}\)^{1/2}, \ \ u_1,u_2\in\Cal A
\end{equation}
for some positive constants $C$ and $K$.
\end{proposition}
\begin{proof} Let $u_1(t)$ and $u_2(t)$, $t\in\R$, be two trajectories on the attractor $\Cal A$ and let $v(t):=u_1(t)-u_2(t)$. Then, this function solves
\begin{equation}\label{2.again}
\Dt v+Av=l(t)v,\ \ l(t):=\int_0^1F'(su_1(t)+(1-s)u_2(t))\,ds
\end{equation}
and since $F$ is globally Lipschitz, we know that $\|l(t)\|_{\Cal L(H,H)}\le L$. Following \cite{AN}, introduce a function
$$
\alpha(t):=\log\|v(t)\|_H-\int_0^t\Phi(s)\,ds,\ \ \Phi(t):=(l(t)v(t),v(t))/\|v(t)\|^2_H.
$$
Then, $|\Phi(t)|\le L$ and elementary calculations show that
$$
\alpha'(t)=\frac{(-Av,v)}{\|v\|^2_H},\ \ \alpha''(t)=\frac{2}{\|v\|^2_H}\(\|\eta\|^2_H-(\eta, l(t)v)\)\ge-\frac12\frac{\|l(t)v\|^2_H}{\|v\|^2_H}\ge-\frac12L^2,
$$
where $\eta:=Av-\frac{(Av,v)}{\|v\|^2_H}v$ (the calculations have sense since $v\in H^2$). Thus, the function $t\to\alpha(t)+L^2t^2/4$ is {\it convex}
and, for any $T>0$ and any $-t\in[-T,0]$, we may write
\begin{multline}\label{1.conv}
\alpha(-t)\le \frac{t}{T}\alpha(-T)+\frac{(T-t)}{T}\alpha(0)+\frac{t}T L^2T^2/4-L^2t^2/4\le\\\le\frac{t}{T}\log\|v(-T)\|_H+\frac{(T-t)}T\log\|v(0)\|_H+\frac{t}TLT+\frac{t(T-t)L^2}4,
\end{multline}
where we have implicitly used that $|\Phi(t)|\le L$. Taking the exponent from this inequality, we end up with the following estimate:
\begin{equation}\label{2.logconstan}
\|v(-t)\|_H\le e^{2Lt+L^2t(T-t)/4}\|v(-T)\|_H^{\frac tT}\|v(0)\|_H^{\frac{T-t}T}.
\end{equation}
We are now ready to finish the proof of estimate \eqref{2.Alog}. To this end, we remind that $v(-T)$ remains bounded as $T\to-\infty$ (since $v_1$ and $v_2$ are on the attractor), so taking $t=1/T$, we get
$$
\|v(-1/T)\|_H\le C\|v(0)\|_H^{1-1/T^2},
$$
where $C$ is independent of $T$ and $u_1,u_2\in\Cal A$. Combining this estimate with the smoothing property \eqref{2.smdif}, we end up with
$$
\|v(0)\|_{H^2}\le CT\|v(-1/T)\|_H\le CT\|v(0)\|^{1-1/T^2}_H.
$$
Note that $T$ is still arbitrary in that estimate. Optimising the left-hand side with respect to $T$ (=taking $T=\(\log\frac{2K}{\|v(0)\|_H}\)^{1/2}$
with $K:=2\|\Cal A\|_H$), we obtain the desired estimate \eqref{2.Alog} and finish the proof of the proposition.
\end{proof}
\begin{remark}\label{Rem2.stan} Although the presented proof looks a bit special and even a bit artificial, the log-convexity estimates similar to \eqref{2.logconstan} are standard (and very powerful) technical tools for the qualitative  theory of semilinear elliptic and parabolic equations, see \cite{lan} and references therein. The log-Lipschitz exponent $\alpha=1/2$ in estimate \eqref{2.Alog} is sharp, the corresponding example will be considered in the next section, see also \cite{kuch}.
\par
Note also that, in contrast to the most results of this section, the exponent $\alpha=1/2$ in \eqref{2.Alog} is strongly related with the fact that $F$ is globally Lipschitz as a map from $H$ to $H$, in more general situation when only \eqref{1.dlip} holds, say, with $\gamma=0$, this constant becomes larger. In fact, despite many results in that direction, see \cite{Gi,kuk,btiti, rom-sur} and the references therein, it is still not clear what is the best log-Lipschitz constant in \eqref{2.Alog} for the case when $F$ decreases the regularity.
\end{remark}
According to Theorem \ref{Th2.rom}, the Lipschitz estimate \eqref{2.equiv} is enough for the existence of the bi-Lipschitz Man\'e projections and according to Proposition \ref{Prop2.Alip}, its slightly weaker analogue \eqref{2.Alog} holds under the reasonable assumptions. Thus, it looks natural   to pose the question whether or not the log-Lipschitz Man\'e projections exist under reasonable generality. This, leads to the following conjecture, see also \cite{24,25,26,27,28}.

\begin{conjecture}\label{Con2.loglip} Under the reasonable assumptions on the non-linearity $F$, the attractor $\Cal A$ possesses at least one Man\'e  projector $P\in\Cal L(H,H)$ such that the inverse $P^{-1}:\overline{\Cal A}\to\Cal A\subset H^2$ is log-Lipschitz with exponent $\alpha\le1$:
\begin{equation}\label{2.loggood}
\|P^{-1}v_1-P^{-1}v_2\|_{H^2}\le C\|v_1-v_2\|_H\(\log\frac{2K}{\|v_1-v_2\|_H}\)^\alpha,\ \ v_1,v_2\in\Cal A
\end{equation}
for some positive constants $C$ and $K$.
\end{conjecture}
Indeed, if this conjecture were true, the uniqueness problem would be immediately solved since \eqref{2.loggood} guarantees that the nonlinearity $W$ in the inertial form \eqref{2.inertial} is log-Lipschitz with the exponent $\alpha\le1$. But, unfortunately, as recently shown in \cite{EKZ}, the conjecture is actually {\it wrong} (the corresponding counterexample is discussed in the next section). Nevertheless, we discuss below some results obtained under the attempts to prove this conjecture and introduce a number of interesting quantities which occurred to be helpful, in particular,  in order to disprove the conjecture.
We start with the so-called doubling factor and Bouligand dimension (which is also often referred as Assouad dimension).

\begin{definition} Let $K$ be a compact set in a metric space $H$. Then, the doubling factor $D(K,H)$ is defined as follows:
$$
D(K,H):=\sup_{\eb\ge0}\sup_{x\in K}N_{\eb/2}(B(\eb,x,H)\cap K,H).
$$
Roughly speaking $D(K,H)$ gives the number of $\eb/2$-balls which are necessary to cover any $\eb$-ball in $K$.
A set $K$ is called $s$-homogeneous, for some $s\ge0$ if there exists a constant $M$ such that
$$
N_{\eb_1}(B(\eb_2,x,H)\cap K,H)\le M\(\frac{\eb_2}{\eb_1}\)^s
$$
for all $x\in K$ and for all $\eb_2\ge\eb_1>0$. A Bouligand (Assouad) dimension of $K$ is defined as follows:
$$
\dim_B(K,H)=\inf\{s\,: \ \text{$K$ is $s$-homogeneous}\}.
$$
\end{definition}
The following properties of the introduced dimension are straightforward, see \cite{28} for the detailed proofs.
\par
1) The Bouligand dimension is finite iff the doubling factor is finite:
$$
\dim_B(K,H)<\infty\  \ \Leftrightarrow\ \ D(K,H)<\infty.
$$
\par
2) The Bouligand dimension preserves under the bi-Lipschitz homeomorphisms, in particular, the Bouligand dimension of an $n$-dimensional Lipschitz sub-manifold is exactly $n$.
\par
3) If $K$ is a subset of $\R^n$ ($n$-dimensional Lipschitz sub-manifold) then its Bouligand dimension is less or equal to $n$.
\par
4) If $K_1\subset K_2$ then $\dim_B(K_1)\le\dim_B(K_2)$ and the fractal dimension is always bounded by the Bouligand dimension:
$$
\dim_f(K,H)\le\dim_B(K,H)
$$
Moreover, $\dim_B(K\times K,H\times H)\le 2\dim_B(K,H)$.

\begin{remark} We see that the finiteness of the Bouligand dimension is a {\it necessary} condition for the existence of bi-Lipschitz Man\'e projections although not sufficient, see counterexamples in \cite{28}. The properties of the Bouligand dimension stated above looks similar to the analogous ones for the fractal/Hausdorff dimension discussed before. However, there is a principal difference which makes this dimension essentially
"less friendly" than the fractal or Hausdorff dimension. Namely, it may {\it increase} greatly and even become infinite if the set is just orthogonally projected to a linear subspace. In a fact, the counterexample to the log-Lipschitz conjecture which will discussed in the next section utilizes this property in an essential way. Mention also that, in contrast to the case of fractal dimension, we do not know how the Bouligand dimensions of the attractor $\Cal A$ in $H$ and $H^2$ are related (if related at all) or how the dimension of the set $\Cal A-\Cal A$ is related with the dimension of $\Cal A$ (the examples of sets, not attractors, where $\Cal A-\Cal A$ is infinite-dimensional for $\Cal A$ being finite-dimensional in the sense of Bouligand dimension are  given in \cite{28}). That is why the log-Lipschitz analogue of the Man\'e projection theorem stated below utilizes the dimension of $\Cal A-\Cal A$.
\end{remark}
The next theorem shows that the finiteness of the Bouligand dimension guarantees the existence of log-Lipschitz Man\'e projections.

\begin{theorem}\label{Th2.log} Let $K$ be a compact set in a Hilbert space $H$ such that
\begin{equation}\label{2.mbd}
\dim_B(K-K,H)=s<\infty.
\end{equation}
Let also $\alpha>0$ and $N>s$ satisfy
\begin{equation}\label{2.expos}
\alpha>\frac12+\frac{2+s}{2(N-s)}.
\end{equation}
Then, for a generic set of orthogonal projectors $P$ to $N$-dimensional planes in $H$, $P:K\to PK$ is one-to-one and the inverse map is $\alpha$-log-Lipschitz on $PK$.
\end{theorem}
For the proof of this theorem see \cite{28} in a more general Banach setting. The examples which show that the condition \eqref{2.expos} on the exponent $\alpha$ is sharp are also given there.
\par
Note that even if we know that the  attractor $\Cal A$ of equation \eqref{1.eqmain} satisfies
\begin{equation}\label{2.unknown}
\dim_B(\Cal A-\Cal A,H)<\infty,
\end{equation}
it is still not enough to verify \eqref{2.loggood} since we need the $H^2$-norm in the left-hand side of \eqref{2.loggood} and Theorem \ref{Th2.log} gives only the weaker estimate with the $H$-norm, but we really need the $H^2$-norm in order to handle the term $P A(P^{-1}v)$ in the definition of the vector field $W$ in \eqref{2.inertial}. The attempt to obtain the required $H^2$-norm estimate combining Theorem \ref{Th2.log} with Proposition \ref{Prop2.Alip}  fails since as a result, we will have the log-Lipschitz exponent $\frac12+\alpha>1$ (obviously, $\alpha>1/2$ in Theorem \ref{Th2.log} and taking a composition of two log-Lipschitz maps we need to sum the log-Lipschitz exponents).
Alternatively, we may assume that
\begin{equation}\label{2.unknown1}
\dim_B(\Cal A-\Cal A,H^2)<\infty.
\end{equation}
Then \eqref{2.loggood} is an immediate corollary of Theorem \ref{Th2.log} (we recall that all norms in a finite-dimensional space are equivalent, so there is no difference between $H^2$ and $H$ norms in the right-hand side of \eqref{2.loggood}). However, this also does not solve the uniqueness problem since in that case the projector $P$ is a priori well-defined in $H^2$ and not in $H$ (which is necessary to treat the term $PAP^{-1}v$ in the inertial form \eqref{2.inertial}) and, to the best of our knowledge the attempt to add the extra assumption that the Man\'e projector $P$ is bounded in a weaker space increases the exponent $\alpha$ in \eqref{2.expos} and as a result is not helpful as well.

\begin{remark}\label{Rem2.badd} We see that although the idea with log-Lipschitz projections and Bouligand dimension was a priori attractive, after the proper investigation it does not look promising any more. Indeed, even if the finiteness of the Bouligand dimension in the form \eqref{2.unknown} or \eqref{2.unknown1} is verified it is still {\it not sufficient} for solving the uniqueness problem and, since both Theorem \ref{Th2.log} and Proposition \ref{Prop2.Alip} are sharp, there is no room here to improve the log-Lipschitz exponents to gain the uniqueness. In addition, there are examples with very regular non-linearities $F$ when the Bouligand dimension of the attractor is infinite (see \cite{EKZ} and next section) and, finally,  no reasons and no mechanisms which can guarantee the finiteness of this dimension without the inertial manifold existence are known. Thus, it seems natural to make a conclusion that the concept of Bouligand dimension and related technical tools  are {\it not sufficient} to solve the uniqueness problem.
\end{remark}
To conclude this paragraph, we introduce following \cite{EKZ} one more quantity which will be used in order to disprove the log-Lipschitz conjecture.
\begin{definition}\label{Def2.logD} Let $K$ be a compact set in a metric space $H$. Then the logarithmic doubling factor is defined as follows
\begin{equation}
D_{log}(K,H):=\sup_{\eb\ge0}\frac{\log D_\eb(K,H)}{\log\log\frac1\eb},\ \ D_\eb(K,H):=\sup_{x\in K}N_{\eb/2}(B(\eb,x,H)\cap K),H).
\end{equation}
\end{definition}
Recall that in the case of homogeneous spaces (where the usual doubling factor and Bouligand dimension are finite) the number of $D_\eb(K,H)$ of $\eb/2$-balls which are necessary to cover any $\eb$-ball of $K$ does not grow as $\eb\to0$ and this fits to the case when $K$ is a subset of $\R^N$ or the finite-dimensional {\it Lipschitz} submanifold. However, if the manifold is less regular (e.g., only log-Lipschitz as in the most interesting for us case of log-Lipschitz Man\'e projections) this number may grow as $\eb\to0$. However, as not difficult to see, the bi-log-Lipschitz embedding
does  not allow this quantity to grow too fast as $\eb\to0$ and as the following proposition shows, the estimate $D_\eb\le C\log\frac1\eb$ must hold in that case
\begin{proposition}\label{Prop2.logfin} Let $K$ be a compact set in a metric space $H$. Assume there is a bi-log-Lipschitz homeomorphism of $K$ on  $\overline{K}\subset\R^N$ for  $N\in\Bbb N$ and some log-Lipschitz exponents (which are not necessarily less or equal to one). Then the log-doubling factor of $K$ is finite:
\begin{equation}
D_{log}(K,H)<\infty.
\end{equation}
\end{proposition}
The proof of this proposition is an elementary exercise and is given in \cite{EKZ}, so we leave it to the reader.

\begin{remark}The introduced log-doubling factor is one of the key technical tools for disproving the log-Lipschitz conjecture. As will be shown in the next section, if the spectral gap condition is violated there are nonlinearities $F$ such that the corresponding global attractor $\Cal A$ has infinite doubling factor and by this reason does not possess any log-Lipschitz Man\'e projection.
\end{remark}

\section{Counterexamples}\label{s3}
The aim of this section is to present the recent counterexamples of \cite{EKZ}. In Paragraph \ref{s4.1}, we show that the $C^1$-smooth inertial manifold may not exist if the spectral gap condition is violated. In Paragraph \ref{s.Floquet}, we discuss the Floquet theory for linear parabolic equations with time-periodic coefficients and, using the proper modification of the known counterexample to that Floquet theory, we show that even the Lipschitz inertial manifold may not exist if the spectral gap condition is violated. Finally, in Paragraph \ref{s4.3}, we  show that the spectral gap condition is in a sense responsible also for the existence of bi-log-Lipschitz Man\'e projections.

\subsection{Absence of $C^1$-smooth inertial manifolds}\label{s4.1}
In this section, following \cite{EKZ}, we show that the spectral gap condition \eqref{1.gap} is sharp in a sense that we can always find the globally Lipschitz and smooth non-linearity $F$ such that equation \eqref{1.eqmain} does not possess a finite-dimensional inertial manifold if the spectral gap condition is violated.
In fact, the counterexample given in a theorem below exploits the ideas from  \cite{rom-counter} and can be considered as a slight modification of the result given there.

\begin{theorem}\label{Th3.1} Let the operator $A$ be such that
\begin{equation}\label{3.3}
L_0:=\sup_{n\in\Bbb N}(\lambda_{n+1}-\lambda_n)<\infty.
\end{equation}
Then, for every $L>\max\{\frac12 L_0,\lambda_1\}$, there exists  a  nonlinear smooth operator $F\in C^\infty(H,H)$ such that
\par
1) $F$ is globally Lipschitz on $H$ with Lipschitz constant $L$:
\begin{equation}\label{3.4}
\|F(u)-F(v)\|_{H}\le L\|u-v\|_H, \ \ \forall u,v\in H,
\end{equation}
\par
2) The problem \eqref{1.eqmain} possesses a compact global attractor $\Cal A$ in $H$,
\par
3) There are no finite-dimensional invariant $C^1$-submanifolds in $H$ containing the global attractor $\Cal A$.
\end{theorem}
\begin{proof} To construct the counterexample, we need the following simple Lemma.
\begin{lemma}\label{Lem3.2} Consider the following two dimensional linear system of ODEs:
\begin{equation}\label{3.5}
\begin{cases}
\frac d{dt}u_n+\lambda_n u_n=Lu_{n+1},\\
\frac{d}{dt}u_{n+1}+\lambda_{n+1}u_{n+1}=-Lu_n.
\end{cases}
\end{equation}
Then, if $2L>\lambda_{n+1}-\lambda_n$, the associated characteristic equation does not have any real roots.
\end{lemma}
\begin{proof}
Indeed, the characteristic equation reads
\begin{equation}\label{3.6}
\det\(\begin{matrix} -\lambda_n-\lambda &L\\-L&-\lambda_{n+1}-\lambda\end{matrix}\)=0, \ \ \lambda^2+(\lambda_n+\lambda_{n+1})\lambda+\lambda_n\lambda_{n+1}+L^2=0
\end{equation}
and the roots are $\lambda=\alpha_n\pm i\omega_n$ with
 \begin{equation}\label{3.7}
\alpha_n:=-\frac{\lambda_n+\lambda_{n+1}}2,\ \ \omega_n=\frac12\sqrt{4L^2-(\lambda_{n+1}-\lambda_n)^2}>0
\end{equation}
and the lemma is proved.
\end{proof}
We are now ready to construct the desired equation \eqref{1.eqmain}. According to \cite{rom-counter}, it is enough to find the  equilibria $u_+$ and $u_-$ of the
problem \eqref{1.eqmain} such that there are no real eigenvalues in the spectrum $\sigma(-A+F'(u_-))$, but there exists {\it exactly}
one real eigenvalue in the  spectrum $\sigma(-A+F'(u_+))$ which is {\it unstable} (and all other eigenvalues are complex-conjugate).
\par
Indeed, assume that the $C^1$-invariant manifold $\Cal M$ exists and $\dim \Cal M=n$. Then, from the invariance, we conclude that
\begin{equation}\label{3.8}
\sigma_{\Cal M}(u_\pm):=\sigma\((-A+F'(u_\pm))\big|_{\Cal T\Cal M(u_\pm)}\)\subset \sigma(-A+F'(u_\pm)),
\end{equation}
where $\Cal T\Cal M(u_\pm)$ are  tangent planes to $\Cal M$ at $u=u_\pm$ (which belong to the invariant manifold  since these equilibria belong to the attractor). In particular, since the equation is real-valued, analyzing the equilibrium $u_-$, we see that $\dim \Cal M$ must be {\it even}
(if $\lambda\in \sigma_{\Cal M}(u_-)$ then $\bar\lambda\in\sigma_{\Cal M}(u_-)$ and $\lambda\ne\bar\lambda$ since there are no real eigenvalues).
\par
Let us now consider the equilibrium $u_+$. We have exactly one real eigenvalue which should belong to $\sigma_{\Cal M}(u_+)$ since the unstable manifold of $u_+$ belongs to the attractor. Since all other eigenvalues must be complex-conjugate, we conclude that the dimension of $\Cal M$ must be odd. This contradiction proves that the $C^1$-invariant manifold does not exist.
\par
Now we fix two equilibria $u_\pm=\pm N e_1$ where $N$ is a sufficiently large number and define the maps $F^\pm(u)$ via the coordinates
$F^\pm_n(u):=(F(u),e_n)$, $u=\sum_{n=1}^\infty u_ne_n$,
in the orthonormal basis $\{e_n\}$ as follows:
\begin{multline}\label{3.9}
F_1^-(u):=-\lambda_1N+Lu_2,\ F^-_2(u):=-L(u_1+N),\\ F_{2n+1}^-(u):=Lu_{2n+2},\ F_{2n+2}^-(u):=-Lu_{2n+1},\ n\ge1
\end{multline}
and
\begin{equation}\label{3.10}
F_1^+(u):=\lambda_1N+L(u_1-N),\ \ F^+_{2n}(u):=-Lu_{2n+1},\ F_{2n+1}^+(u):=Lu_{2n},\ n\ge1.
\end{equation}
Both $F^-$ and $F^+$ are smooth and Lipschitz continuous (in fact, linear) with Lipschitz constant $L$ and we may construct the smooth nonlinear map $F(u)$ such that
\begin{equation}\label{3.11}
F(u)\equiv F^\pm(u) \ \text{if $u$ is close to $u_\pm$}
\end{equation}
and the global Lipschitz constant of $F$ is less than $L+\eb$ (where $\eb=\eb(N)$ can be made arbitrarily small by increasing $N$, see the next paragraph for the more detailed construction of this map). Finally, we may cut-off the nonlinearity $F$ outside of a large ball making it dissipative without expanding the Lipschitz norm (see \cite{tem} for the details). This guarantees the existence of a global attractor $\Cal A$.
\par
Let find now the spectrum $\sigma(-A+F'(u_\pm))$ at equilibria $u_\pm$. Indeed, according to the construction of $F$, the linearization of \eqref{1.eqmain}
at $u=u_-$ looks like
 \begin{equation}\label{3.12}
\frac d{dt}v_{2n-1}=-\lambda_{2n-1} v_{2n-1}+Lv_{2n},\ \ \frac d{dt}v_{2n}=-\lambda_{2n}v_{2n}-Lv_{2n-1}, \ \ n=1,2,\cdots
\end{equation}
and, due to condition $L>\frac12 L_0$, there are no real eigenvalues in $\sigma(-A+F'(u_\pm))$, see Lemma \ref{Lem3.2}.
In contrast to that, the linearization near $u=u_+$ reads
\begin{multline}\label{3.13}
\frac d{dt} v_1=(L-\lambda_1)v_1,\\
\frac d{dt}v_{2n}=-\lambda_{2n} v_{2n}+Lv_{2n+1},\ \ \frac d{dt}v_{2n+1}=-\lambda_{2n+1}v_{2n+1}-Lv_{2n}, \ \ n=1,2,\cdots.
\end{multline}
Thus, since $L>\lambda_1$, we have exactly one positive eigenvalue and all other eigenvalues are complex conjugate by Lemma \ref{Lem3.2}. This shows the absence of any finite-dimensional $C^1$-invariant manifold for this problem and finishes the proof of the theorem.
\end{proof}
\begin{remark}\label{Rem1.1} Recall that the extra assumption  $L>\lambda_0$ is necessary in order to have the instability in the equation \eqref{1.eqmain}, otherwise the nonlinearity will be not strong enough to compensate the dissipativity of $A$, and the attractor will consist of a single exponentially stable point (which can be naturally treated as a zero dimensional inertial manifold for the problem).
\par
Mention also that the proof of Theorem \ref{Th3.1} is a typical proof ad absurdum which gives the non-existence result, but does not clarify much what happens with the dynamics on the attractor and how the obstructions to the inertial manifold look like, see also \cite{sell-counter} for similar counterexamples in the gradient case of a reaction-diffusion equation in 4D. In the next paragrphs, we make this counterexample more constructive and show that similar structure forbid the existence not only $C^1$-smooth manifolds, but also Lipschitz and even log-Lipschitz ones.
\end{remark}

\subsection{Floquet theory and absence of Lipschitz inertial manifolds}\label{s.Floquet}

In this section, following \cite{EKZ}, we refine the counterexample from the previous
section to show that without the spectral gap condition, not only
$C^1$, but also Lipschitz invariant manifolds containing the
attractor may not exist. Actually, we will find two trajectories $u$
and $v$ on the attractor such that
\begin{equation}\label{4.1}
\|u(t)-v(t)\|_H\le Ce^{-\kappa t^2}, \ \ t\ge0
\end{equation}
for some positive $C$ and $\kappa$. This will imply by Theorem \ref{Th2.rom} that the attractor cannot be  bi-Lipschitz projected to any finite-dimensional plane, so neither finite-dimensional Lipschitz inertial manifold, nor bi-Lipschitz Man\'e projections will exist in that example.
In addition, according to \cite{rom-th,rom-th1}, in this situation the attractor also cannot be embedded to any (not-necessarily invariant) $C^1$-submanifold.
\par
Thus, the main result of this section is the following theorem.

\begin{theorem}\label{Th4.1} Let $A$ be a self-adjoint positive operator with compact inverse acting in a Hilbert space $H$ and let the spectral gap exponent
$L_0<\infty$, see \eqref{3.3}. Then, for every $L>\max\{\frac12 L_0,\lambda_2\}$, there exists   a smooth nonlinearity $F(u)$ satisfying \eqref{3.4}
 such that problem \eqref{1.eqmain} possesses a global attractor $\Cal A$  which contains at least two trajectories $u(t)$ and $v(t)$
 satisfying~\eqref{4.1}.
\end{theorem}

The construction of the nonlinearity $F(u)$ with such properties is strongly based on  the known counterexample to the Floquet theory, see \cite{kuch,DFKM}. By this reason, we first briefly remind what is this theory about starting from the finite-dimensional case. Consider the following linear equation with time-periodic coefficients:
\begin{equation}\label{4.2}
\Dt w+Aw=\Phi(t)w,
\end{equation}
where, for a moment, $w\in H=\R^N$ and $\Phi(t)$ is a bounded time-periodic operator (matrix) with period $T$: $\Phi(t)\in\Cal L(H,H)$, $\Phi(t+T)=\Phi(t)$. Denote by $U(t,s)\in\Cal L(H,H)$, $s\in\R$, $t\ge s$, the solution operator generated by this equation:
$$
U(t,s)w(s):=w(t).
$$
Then, the following classical result (which is often referred as a Floquet-Lyapunov theorem, see e.g., \cite{hart}) holds.

\begin{proposition}\label{Prop3.fl} Let $H$ be finite-dimensional. Then, there exists a $T$-periodic complex valued  linear change of variables $z=C(t)w$
which transforms \eqref{4.2} to the autonomous equation:
\begin{equation}\label{3.auto}
\frac d{dt} z=\Bbb A z,
\end{equation}
where
\begin{equation}\label{3.log}
\Bbb A:=\frac1T\log U(T)
\end{equation}
and $U(T):=U(t+T,t)$ is a period map (=Poincare map or monodromy matrix) associated with equation \eqref{4.2}.
\end{proposition}
\begin{proof} Indeed, since \eqref{4.2} is uniquely solvable, $\det U(T)\ne0$ and we may find a complex valued logarithm of this matrix using the Jordan normal form (of course, it is not unique, but only the existence is important). Define
$$
C(t):=e^{\Bbb At}[U(t,0)]^{-1}.
$$
Then, on the one hand, this matrix is $T$-time periodic:
$$
C(t+T)=e^{\Bbb A(t+T)}[U(T+t,0)]^{-1}=e^{\Bbb At}U(T)[U(t+T,T)U(T,0)]^{-1}=e^{\Bbb At}[U(t,0)]^{-1}=C(t),
$$
where we have used that $U(t+T,s+T)=U(t,s)$ due to the time-periodicity. On the other hand,
$$
z(t):=C(t)w(t)=e^{\Bbb A(t)}[U(t,0)]^{-1}U(t,0)w(0)=e^{\Bbb At}w(0)
$$
and, therefore, $z(t)$ solves the autonomous ODE \eqref{3.auto} and the desired $T$-time periodic complex valued change of variables is found and the proposition is proved.
\end{proof}
\begin{remark}\label{Rem3.real}
If we are interested in the {\it real}-valued change of variables,  then there is a  problem that one (or more generally, odd number) of the eigenvalues of the matrix $U(T)$ may be real and {\it negative} which leads to $\det U(T)<0$. In that case, the real-valued matrix logarithm does not exist and we are unable to use the above arguments. Moreover, since $\det e^{\Bbb At}>0$ if $\Bbb A$ is real, we see that $T$-periodic real-valued change of variables does not exist in that case.
\par
However, if we double the period and consider the matrix $U(2T)=[U(T)]^2$ this problem disappears since $\det U(2T)=[\det U(T)]^2>0$ and the $2T$-periodic real-valued change of variables always exists, see \cite{hart}.
\end{remark}
Thus, in the generic case where all eigenvalues of $\Bbb A$ are different, using the transform $z=C(t)w$, we may write any solution of \eqref{4.2} in the form of
$$
w(t)=C(t)\(\sum_{i=1}^N C_ie^{\mu_i t}e_i\)=\sum_{i=1}^NC_i e^{\mu_i t}e_i(t),
$$
where $C_i\in\Bbb C$, $\mu_i\in\sigma(\Bbb A)$, $e_i$ - are the corresponding eigenvalues and $e_i(t):=C(t)e_i$ are $T$-periodic functions. The exponents $\mu_i$ are usually referred as {\it Floquet} exponents and the corresponding solutions $e^{\mu_it}e_i(t)$ are the Floquet (or Floquet-Bloch) solutions.
Thus, the main result of the Floquet theory in finite-dimensions is that any solution $w(t)$ of a linear equation \eqref{4.2} is a sum of the Floquet solutions which has the form $e^{\mu_it}e_i(t)$ with time-periodic functions $e_i(t)$ (in the case of multiple Floquet exponents, there are also solutions of the form $t^ke^{\mu_it}\tilde e_i(t)$). Note also that the Floquet exponents can be found without computing the matrix logarithm via
$$
\mu_i=\frac1T\log\nu_i,
$$
where $\nu_i\in\sigma(U(T))$ and, therefore, they are determined by the spectrum of the Poincare map $U(T)$. One more consequence of the Floquet theory is there are no solutions of linear ODEs with time-periodic coefficients which  decay/grow faster than exponential.
\par
We now return to the infinite-dimensional case where the operator $A$ satisfies the assumptions of Section \ref{s1} and $\Phi(t)$ is globally Lipschitz in $H$ uniformly in time. The immediate difficulty which we meet with is that the Poincare map $U(T)$ is a {\it compact} operator (due to the parabolic smoothing property, $U(T):H\to H^2$ and $H^2$ is compactly embedded in $H$), so the spectrum $\sigma(U(T))$ {\it always} contains zero
as a point of essential spectrum (of course, we do not have zero eigenvalues due to the backward uniqueness which ia actually established in the proof of Proposition \ref{Prop2.Alip}, see estimate \eqref{2.logconstan}), so it is unclear how to define the logarithm of such operator in \eqref{3.log}. One more difficulty is the absence of a  Jordan normal form for infinite dimensional operators, so the logarithm may not exist even in the case where $0$ does not belong to the spectrum, but $0$ and $\infty$ are in different connected components of the resolvent set, see \cite{shef} for nice examples.
\par
Nevertheless, as not difficult to see, any non-zero $\nu\in\sigma(U(T))$ still generates a Floquet exponent $\mu_i=\frac1T\log\nu_i$ and the corresponding Floquet solution $e^{\mu_it}e_i(t)$ with periodic $e_i(t)$ (to verify this, it is enough to use the fact that $\nu_i\in\sigma(U(T))$ is an eigenvalue of finite multiplicity since $\nu\ne0$ and $U(T)$ is compact). Thus, instead of trying to compute the operator logarithm \eqref{3.log}, it looks more natural to pose and study the questions related with the completeness of the eigenvalues of $U(T)$ (= the completeness of the finite sums of Floquet solutions in the space of all solutions). A lot of results in that direction are given in \cite{kuch}.
\par
 However, in general we basically know only that $U(T)$ is compact and that is not enough to have a reasonable spectral theory. For instance, it is well-known that the spectrum of a compact operator may coincide with zero and the eigenvectors may not exist at all! Moreover,  that even does not look like as a pathology and happens, e.g,  for the natural class of the  so-called Volterra-type integral operators, see \cite{dan}. The simplest example of such operators is the following integration operator:
 \begin{equation}\label{3.int}
 Uf:=\int_0^xf(s)\,ds,\ \ H=L^2(0,1).
 \end{equation}
Indeed, as easy to see \eqref{3.int} does not have any eigenvectors and looks like an infinite dimensional Jordan sell in the basis $e_k:=x^k$:
\begin{equation}\label{3.jordan}
Ue_k=\frac1{k+1}e_{k+1},\ \ k=0,1,\cdots.
\end{equation}
We remind this simple example since in the counterexample to Floquet theory given in \cite{kuch} as well as in the related example of \cite{EKZ} discussed below, the structure of the Poincare map $U(T)$ is similar to \eqref{3.jordan} and, in particular, $\sigma(U(T))=\{0\}$ and all solutions decay faster than exponentially  as $t\to\infty$.
\par
 Roughly speaking, in our case, equation \eqref{4.2} will coincide with system \eqref{3.12} on a half period, say, for $t\in[0,T/2]$ and with system \eqref{3.13} on the other half period $t\in[T/2,T]$ and the parameters $L$ and $T$ are chosen in such way that following equations \eqref{3.12}, $U(T/2,0)e_{2n-1}=K_n^+ e_{2n}$, $n=1,2,\cdots$, and, following equations \eqref{3.13}, $U(T,T/2)e_{2n}=K_n^-e_{2n+1}$, for some contraction factors $K_n^\pm$ so finally we will have $U(T)e_{2n-1}=K_n^+K_n^-e_{2n+1}$ and we indeed see something similar to \eqref{3.jordan}. However, we need to be a bit more accurate since also the smoothness of the map with respect to $t$ is required.
\par
To make the above arguments precise, we fix a scalar periodic function $x(t)$ with a  period $2T$ satisfying the following assumptions:

 1) $x(-t)=-x(t)$ and $x(T-t)=x(t)$ for all $t$;

 2) $x(T/2):=N\ge1$ is the maximal values of $x(t)$;

 3) $x''(t)<0$ for $0<t<T$ and $x'(t)>0$ for $0<t<T/2$.

 For instance, one may take $x(t)=\sin(\pi t/T)$. To simplify the notations, we switch here to the $2T$-periodic case instead of $T$-periodic considered before.

Also, without loss of generality,
  we may assume that
 \begin{equation}\label{4.3}
 c_2k\le \lambda_k\le c_1 k
 \end{equation}
 for some positive $c_1$ and $c_2$. Indeed, the absence of the spectral gap ($L_0<\infty$) gives the upper bound for $\lambda_k$ and the lower bound can be achieved just by dropping out the unnecessary modes.
 \par
 The following proposition is crucial for what follows.
 \begin{proposition}\label{Prop4.1} Let the assumptions of Theorem \ref{Th3.1} hold. Then, for every $L>\frac12L_0$ and every periodic function $x(t)$ of sufficiently large period $2T$, satisfying the above assumptions, there exists a smooth map $\Cal R:\, \R\to \Cal L(H,H)$ such that
 \begin{equation}\label{4.4}
 \|\Cal R(x)\|_{\Cal L(H,H)}\le L.
 \end{equation}
 Moreover, if $\Phi(t):=\Cal R(x(t))$ is the $2T$-periodic map, then the Poincare map $P:=U(2T,0)$ associated with equation \eqref{4.2} satisfies the following properties:
 \begin{equation}\label{4.5}
 Pe_{2n-1}=\mu_{2n-1} e_{2n+1},\ n\in\Bbb N,\ \ Pe_{2n}=\mu_{2n-2} e_{2n-2},\ n>1,\ \ Pe_2=\mu_0e_1,
 \end{equation}
 where the positive multipliers $\mu_0$, $\mu_n$ defined via
 \begin{equation}\label{4.6}
 \mu_n:=e^{-T(\lambda_{n}+2\lambda_{n+1}+\lambda_{n+2})/2},\ \  \mu_0:=e^{-T(2\lambda_{1}+\lambda_{2})/2}.
 \end{equation}
 In particular, all solutions of \eqref{4.2} decay super-exponentially, namely,
 \begin{equation}\label{4.7}
 \|w(t)\|_{H}\le Ce^{-\beta t^2}\|w(0)\|_H,
 \end{equation}
 where positive $C$ and $\beta$ are independent of $w(0)\in H$.
 \end{proposition}
 \begin{proof}
 Introduce a pair of smooth nonnegative cut-off functions $\theta_1(x)$ and $\theta_2(x)$ such that
\par
1) $\theta_1(x)\equiv1$, $x\in[N/2,N]$
 and $\theta_1(x)\equiv0$ for $x\le N/4$;
 \par
2) $\theta_2(x)\equiv 1$ for $x\ge N/4$ and $\theta_2(x)\equiv 0$ for $x\le0$;
\par
and define  the linear operators $F_\pm$ as follows:
\begin{multline}\label{4.8}
F^-_{2n-1}(x)w=\frac12(\lambda_{2n-1}-\lambda_{2n})w_{2n-1}\theta_2(x)+\eb\theta_1(x) w_{2n},\\ F^-_{2n}(x)w=-\frac12(\lambda_{2n-1}-\lambda_{2n})w_{2n}\theta_2(x)-\eb w_{2n-1}\theta_1(x),\ \ n\in\Bbb N
\end{multline}
and
\begin{multline}\label{4.9}
F^+_1(x)w=0,\ F^+_{2n}(x)w=\frac12(\lambda_{2n}-\lambda_{2n+1})w_{2n}\theta_2(-x)+ \eb w_{2n+1}\theta_1(-x),\\
 F^+_{2n+1}(x)w=-\frac12(\lambda_{2n}-\lambda_{2n+1})w_{2n+1}\theta_2(-x)- \eb w_{2n}\theta_1(-x),\ n\in\Bbb N,
\end{multline}
where $\eb>0$ is a positive number which will be specified later. Finally, we introduce
the desired operator $\Phi(t)$ via
\begin{equation}\label{4.10}
\Phi(x(t))w:=F^+(x(t))w+F^-(x(t))w.
\end{equation}
We claim that there exists a small $\eb=\eb(T)>0$ such that the operator defined satisfies all assumptions of the proposition. Indeed, let $P_-$ and $P_+$ be the solution operators which map $w(0)$ to $w(T)$ and $w(T)$ to $w(2T)$ respectively:
\begin{equation*}
P_-:=U(T,0), \ \ P_+:=U(2T,T).
\end{equation*}

 Then, the spaces
\begin{equation}\label{4.11}
V_{n}^-:=\sppan\{e_{2n-1},e_{2n}\} \ \ \text{and}\ \ \ V_{n}^+:=\sppan\{e_{2n},e_{2n+1}\},\ n\in\Bbb N
\end{equation}
are invariant subspaces for the linear maps $P_-$ and $P_+$ respectively.
\par
We need to look at $P_\pm e_n$. To this end, we introduce $T_0$ such that $x(T_0)=N/4$ and note that, by the construction of the cut-off functions,
all $e_n$'s are invariant with respect to the solution maps $U(T_0,0)$, $U(T,T-T_0)$, $U(T+T_0,T)$ and $U(2T,2T-T_0)$, so we only need to study the maps
$U(T-T_0,T_0)$ and $U(2T-T_0,T+T_0)$. On these time intervals the cut-off function $\theta_2\equiv1$ and the equations read:
\begin{multline}\label{4.12}
\frac d{dt}w_{2n-1}=-\frac12(\lambda_{2n-1}+\lambda_{2n})w_{2n-1}+\eb\theta_1(x(t)) w_{2n}, \\ \frac d{dt}w_{2n}=-\frac12(\lambda_{2n-1}+\lambda_{2n})w_{2n}-\eb\theta_1(x(t)) w_{2n-1},
\end{multline}
for $t\in[T_0,T-T_0]$ and
\begin{multline}\label{4.13}
\frac d{dt}w_{2n}=-\frac12(\lambda_{2n}+\lambda_{2n+1})w_{2n}+\eb\theta_1(-x(t)) w_{2n+1}, \\ \frac d{dt}w_{2n+1}=-\frac12(\lambda_{2n}+\lambda_{2n+1})w_{2n+1}-\eb\theta_1(-x(t)) w_{2n},
\end{multline}
for $t\in[T+T_0,2T-T_0]$. To study these equations, we introduce the polar coordinates
\begin{equation}\label{4.14}
 w_{2n-1}+iw_{2n}=R_n^-e^{i\varphi_n^-},\ \ w_{2n}+iw_{2n+1}=R^+_ne^{i\varphi_n^+}
 \end{equation}
 for problems \eqref{4.12} and \eqref{4.13} respectively. Then the phases $\varphi_n^\pm$ solve the equations
 \begin{equation}\label{4.15}
\frac d{dt}\varphi_n^- =-\eb\theta_1(x(t))\ \text{and}\ \ \frac d{dt}\varphi_n^+ =-\eb\theta_1(-x(t))
\end{equation}
for $t\in[T_0,T-T_0]$ and $t\in[T+T_0,2T-T_0]$ respectively.
\par
Finally, if we fix
\begin{equation}\label{4.16}
\eb:=-\frac{\pi}{2\int_{T_0}^{T-T_0}\theta_1(x(t))\,dt}=-\frac{\pi}{2\int_{T+T_0}^{2T-T_0}\theta_1(-x(t))\,dt},
\end{equation}
then
both $U(T-T_0,T_0)$ and $U(2T-T_0,T+T_0)$ restricted to $V_{n}^-$ and $V_{n}^+$ respectively will be compositions of the rotation
 on the angle $\pi/2$ and the proper scaling. Namely,
 \begin{multline*}
 U(T-T_0,T_0)e_{2n-1}=e^{-(T-2T_0)/2(\lambda_{2n-1}+\lambda_{2n})}e_{2n},\\
 U(T-T_0,T_0)e_{2n}=e^{-(T-2T_0)/2(\lambda_{2n-1}+\lambda_{2n})}e_{2n-1},\\
 U(2T-T_0,T+T_0)e_{2n}=e^{-(T-2T_0)/2(\lambda_{2n}+\lambda_{2n+1})}e_{2n+1},\\
  U(2T-T_0,T+T_0)e_{2n+1}=e^{-(T-2T_0)/2(\lambda_{2n}+\lambda_{2n+1})}e_{2n}.
 \end{multline*}
Furthermore, on the time intervals $t\in[0,T_0]$ and $t\in[T-T_0,T_0]$, we have the decoupled equations
\begin{multline*}
\frac{d}{dt}w_{2n-1}+\lambda_{2n-1}w_{2n-1}=1/2(\lambda_{2n-1}-\lambda_{2n})\theta_2(x(t))w_{2n-1},\\
\frac{d}{dt}w_{2n}+\lambda_{2n}w_{2n}=-1/2(\lambda_{2n-1}-\lambda_{2n})\theta_2(x(t))w_{2n},
\end{multline*}
and, therefore,
\begin{multline*}
U(T_0,0)e_{2n-1}=e^{-\lambda_{2n-1}T_0}e^{1/2(\lambda_{2n-1}-\lambda_{2n})\int_0^{T_0}\theta_2(x(t))\,dt}e_{2n-1},\\
U(T_0,0)e_{2n}=e^{-\lambda_{2n}T_0}e^{-1/2(\lambda_{2n-1}-\lambda_{2n})\int_0^{T_0}\theta_2(x(t))\,dt}e_{2n},\\
U(T,T-T_0)e_{2n-1}=e^{-\lambda_{2n-1}T_0}e^{1/2(\lambda_{2n-1}-\lambda_{2n})\int_0^{T_0}\theta_2(x(t))\,dt}e_{2n-1},\\
U(T,T-T_0)e_{2n}=e^{-\lambda_{2n}T_0}e^{-1/2(\lambda_{2n-1}-\lambda_{2n})\int_0^{T_0}\theta_2(x(t))\,dt}e_{2n}.
\end{multline*}
  Thus, for the operator $P_-=U(T,T-T_0)U(T-T_0,T_0)U(T_0,0)$, we have
\begin{multline}\label{4.17}
 P_-e_{2n-1}=e^{-\lambda_{2n-1}T_0}e^{1/2(\lambda_{2n-1}-\lambda_{2n})\int_0^{T_0}\theta_2(x(t))\,dt} e^{-(T-2T_0)/2(\lambda_{2n-1}+\lambda_{2n})}\times\\\times
 e^{-\lambda_{2n}T_0}e^{-1/2(\lambda_{2n-1}-\lambda_{2n})\int_0^{T_0}\theta_2(x(t))\,dt}e_{2n}=e^{-T(\lambda_{2n-1}+\lambda_{2n})/2}e_{2n},\\ P_-e_{2n}=e^{-\lambda_{2n}T_0}e^{-1/2(\lambda_{2n-1}-\lambda_{2n})\int_0^{T_0}\theta_2(x(t))\,dt}
 e^{-(T-2T_0)/2(\lambda_{2n-1}+\lambda_{2n})}\times\\\times
 e^{-\lambda_{2n-1}T_0}e^{1/2(\lambda_{2n-1}-\lambda_{2n})\int_0^{T_0}\theta_2(x(t))\,dt}
 e_{2n-1}=e^{-T(\lambda_{2n-1}+\lambda_{2n})/2}e_{2n-1}
 \end{multline}
 and, analogously, for $P_+=U(2T,2T-T_0)U(2T-T_0,T+T_0)U(T+T_0,T)$,
\begin{multline}\label{4.18}
 P_+e_1=e^{-T\lambda_1/2} e_1,\  P_+e_{2n}=e^{-T(\lambda_{2n}+\lambda_{2n+1})/2}e_{2n+1},\\ P_+e_{2n+1}=e^{-T(\lambda_{2n}+\lambda_{2n+1})/2}e_{2n},\ \ n\in\Bbb N.
\end{multline}
This proves the desired spectral properties of the Poincare map $P:=P_+\circ P_-$ (see \eqref{4.5}). Note that, due to \eqref{4.16}, $\eb\to0$ if $T\to\infty$ and the norm of the operator $\Phi(x(t))$ with $\eb=0$ clearly does not exceed $\frac12L_0$. Moreover, the constant $\eb$ defined by
\eqref{4.16} can be made  arbitrarily small by increasing the period $T$. Indeed, from the convexity assumption on $x(t)$, we know that $x(t)\ge N/2$ for $t\in[T/4,3T/2]$ and, therefore,
$$
|\eb|\le\frac\pi T.
$$
Thus,  \eqref{4.4} will be satisfied if $T$ is large enough.
\par
Let us check \eqref{4.7}. To this end, it suffices to verify that
\begin{equation}\label{4.19}
\|P^Ne_{2n}\|_H\le Ce^{-\beta_1 N^2}
\end{equation}
uniformly with respect to $n\in\Bbb N$. Indeed, according to \eqref{4.6} and \eqref{4.3}, the multipliers $\mu_{n}$ satisfy
\begin{equation}\label{4.20}
 e^{-C_2Tn}\le \mu_n\le e^{-C_1Tn}.
\end{equation}
Then, for $N\ge n$,
\begin{equation*}
\|P^Ne_{2n}\|_H\le e^{-CT(\sum_{k=0}^n 2k+\sum_{k=0}^{N-n} 2k)}=e^{-CT(n(n+1)+(N-n)(N-n+1))}\le e^{-CTN^2/2}.
\end{equation*}
Since for $N\le n$, \eqref{4.19} is obvious, \eqref{4.7} is verified and Proposition \ref{Prop4.1} is proved.
\end{proof}
\begin{proof}[Proof of the theorem]
It is now not difficult to construct the desired counterexample. To this end, we generate the $2T$-periodic trajectory $(x(t),y(t))$,
such that the first component $x(t)$ satisfies all the above assumptions, as a solution
of the 2D system of ODEs:
\begin{equation}\label{4.21}
\frac d{dt}x=f(x,y),\ \ \frac d{dt} y=g(x,y)
\end{equation}
with smooth functions $g$ and $f$ (cut off for large $x$ and $y$ in order to have the dissipativity). By scaling time (and increasing the period $T$), we can also make the Lipschitz norms of $f$, $g$ and $x$ arbitrarily small. Then, we consider the coupled system
for $u=(x,y,w)$:
 \begin{equation}\label{4.22}
\frac d{dt}x=f(x,y),\ \ \frac d{dt} y=g(x,y),\ \ \Dt w+Aw=\Cal R(x)w.
\end{equation}
Obviously, the system \eqref{4.22} is of the form \eqref{1.eqmain} (we only need to reserve the first two modes $e_1$ and $e_2$ for $x$ and $y$ and
re-denote $Q_3A$ by $A$). It is also not difficult to see that the Lipschitz norm of the nonlinearity in \eqref{4.22} can be made arbitrarily close to $L>\frac12L_0$ (since the Lipschitz constants of $f$, $g$ and $x$ are small), but in order to produce the periodic trajectory $x(t)$, $y(t)$ we should be able to destabilize the first two modes (compensate the extra terms $\lambda_1 x$ and $\lambda_2y$ which come from the linear part of the equation \eqref{1.eqmain}) and this leads to the extra condition $L>\lambda_2$.
\par
Finally, in order to finish the construction, we need to guarantee that at least one of the trajectories of the form $v(t):=(x(t),y(t),w(t))$, $t\ge0$ with non-zero $w(t)$ belongs to the attractor. To this end, we fix the trajectory $v(t)$ of \eqref{4.22} such that $w(0)=e_1$ and
 $w_1(t):=(w(t),e_1)<1$ for all $t\ge0$. After this, we introduce a smooth coupling $R(x,y,w_1)=(R_1,R_2,R_3)$ and the perturbed version of \eqref{4.22}
\begin{multline}\label{4.23}
\frac d{dt} x=f(x,y)+R_1(x,y,w_1),\ \frac d{dt} y(t)=g(x,y)+R_2(x,y,w_1),\\  \frac d{dt} w+Aw=\Cal R(x)w+R_3(x,y,w_1)e_1
\end{multline}
such that $R\equiv0$ if $w_1\le2$ and such that, in addition, the 3D system
\begin{multline}\label{4.24}
\frac d{dt} x=f(x,y)+R_1(x,y,w_1),\ \frac d{dt} y(t)=g(x,y)+R_2(x,y,w_1),\\  \frac d{dt} w_1+\lambda_1w_1=R_3(x,y,w_1)
\end{multline}
possesses a saddle point (somewhere in the region $w_1>2$, $x<0$), such that the point $(x(0),y(0),1)$ belongs to the unstable manifold of this equilibrium and the corresponding trajectory $v(t)=(\bar x(t),\bar y(t),w_1(t))$ satisfies $\bar x(t)<0$ for $t<0$.
Such smooth coupling obviously exists and, using the trick with scaling the time, one can guarantee that the global Lipschitz constant for the nonlinearity in \eqref{4.23} does not exceed $L$. Moreover, the assumption that $\bar x(t)<0$ for $t<0$  guarantees that $F_-(\bar x(t))\equiv0$, for $t\le0$, see definition \eqref{4.8} and, therefore, the line $\R e_1$ remains invariant for the operator $\Cal R(\bar x(t))$ if $t<0$. By this reason, the backward solution $(\bar x(t),\bar y(t),w_1(t))$, $t\le0$, generates also a backward solution $(\bar x(t),\bar y(t),w(t))$, $t\le0$, for
the infinite-dimensional system \eqref{4.24} (just by setting $w_i(t)=0$ for $i>1$). Furthermore, since $x(0)=\bar x(0)$ and $y(0)=\bar y(0)$ and the nonlinearity $R$ vanishes on that trajectory for $t\ge0$, we have $\bar x(t)=x(t)$ and $\bar y(t)=y(t)$ for $t\ge0$.
\par
Finally,
since any unstable manifold as well as any periodic orbit belong to the attractor, both of the trajectories $v(t)$ and  $u(t):=(x(t),y(t),0)$  belong to the attractor. But, according to Proposition \ref{Prop4.1},
\begin{equation}\label{4.25}
\|u(t)-v(t)\|_H=\|w(t)\|_H\le Ce^{-\beta t^2}
\end{equation}
and, according to Theorem \ref{Th2.rom}, the Lipschitz inertial manifold containing the global attractor cannot exist.
Thus, Theorem \ref{Th4.1} is proved.
\end{proof}
\begin{remark}\label{Rem2.sharp} Estimate \eqref{4.25} shows also that Proposition \ref{Prop2.Alip} is sharp. Indeed, if $u(t)$ and $v(t)$ are two solutions on the attractor $\Cal A$, then $w(t)=u(t)-v(t)$ solves \eqref{2.again}. Multiplying this equation on $w(t)$ and using \eqref{2.Alog} together with the Lipschitz continuity of $F$, we get
$$
\frac12\frac d{dt}\|w(t)\|^2_H\ge-\|w(t)\|_H\|Aw(t)\|_H-L\|w(t)\|^2_H\ge-K\|w(t)\|^2_H\log^{1/2}\frac{2C}{\|w(t)\|_H}
$$
for some $K>0$. Therefore, the function $y(t):=\|w(t)\|_H$ satisfies the following differential inequality
$$
y'(t)\ge -K y(t)\log^{1/2}\frac{2C}{y(t)}
$$
and, solving this inequality, we get the estimate, opposite to \eqref{4.25}
\begin{equation}\label{4.ssarp}
\|u(t)-v(t)\|_H\ge C_1e^{-\beta_1 t^2}
\end{equation}
for some positive $C_1$ and $\beta$. This shows, in particular, that the exponent $1/2$ in \eqref{2.Alog} is sharp and cannot be improved.
\end{remark}

\subsection{An attractor with infinite log-doubling factor}\label{s4.3}

In this section, we explain, following to \cite{EKZ} how to construct (under the assumptions of  Theorem \ref{Th4.1}) the smooth nonlinearity $F(u)$ in such way that the corresponding attractor $\Cal A$ will be not embedded into any finite-dimensional log-Lipschitz manifold. Namely
the following result holds.

\begin{theorem}\label{Th5.1} Let the assumptions of Theorem \ref{Th4.1} holds. Then, for every $L>\max\{\frac12L_0,\lambda_2\}$ there exists a smooth nonlinearity $F(u)$ satisfying \eqref{3.4} such that the attractor $\Cal A$ of problem \eqref{1.eqmain} has an infinite log-doubling factor:
\begin{equation}\label{2.inf}
D_{log}(\Cal A,H)=\infty,
\end{equation}
see Definition \ref{Def2.logD}, and, due to Proposition \ref{Prop2.logfin}, the attractor $\Cal A$ does not possess any Man\'e projections with log-Lipschitz inverse.
\end{theorem}
The detailed proof of this theorem is given in \cite{EKZ}. Since it is rather technical, we will not reproduce it here. Instead, we restrict ourselves by the informal indication of  the main ideas behind this proof.
\par
As the first step, we discuss the what structures inside of a compact set $\Cal A$ may guarantee that the Bouligand dimension or logarithmic doubling factor is infinite. A natural choice here are the so-called orthogonal sequences and their modifications which play an essential role in various counterexamples in the dimension theory, see \cite{hunt,24,28} and references therein. Namely, let $\eb_n>0$ be a non-increasing sequence converging to zero. Consider the compact set in $H$
\begin{equation}\label{5.set}
K:=\cup_{n\in\Bbb N}\eb_n[-1,1]e_n.
\end{equation}
This set consists of a sequence of orthogonal segments $\eb_n[-1,1]e_n$. The Hausdorff dimension of $K$ is always one since $K$ is a countable sum of one-dimensional sets. The fractal dimension of $K$ depends on the rate of convergence of $\eb_n\to0$, but if that rate is sufficiently fast, namely,
\begin{equation}\label{5.fast}
\eb_n\le C n^{-\alpha},\ \  \alpha>1,
\end{equation}
then, as not difficult to see, $\dim_f(K,H)=1$. However, the following result holds.

\begin{proposition}\label{Prop5.inf} The Bouligand dimension of $K$ is always infinite no matter how fast is the decay rate of $\eb_n>0$.
\end{proposition}
\begin{proof}
Indeed, for every $n\in\Bbb N$, the vectors $\eb_ne_1,\eb_n e_2,\cdots,\eb_n e_n$ belong to $K$ (here we have used that $\eb_n$ is non-increasing. Obviously all these vectors are in the $\eb_n$-ball centered at zero. However, $\|\eb_ne_k-\eb_ne_l\|=\sqrt 2\eb_n\delta_{kl}$ and we need at least $n$ balls of radius $\eb_n/2$ are necessary to cover these points. Thus,
\begin{equation}\label{5.notgood}
D_{\eb_n}(K,H):=\sup_{x\in K}N_{\eb_n/2}(K\cap B(\eb_n,x,H),H)\ge n.
\end{equation}
Thus, the doubling factor of $K$ is infinite and, therefore, the Bouligand dimension of $K$ is also infinite.
\end{proof}
Thus, we see that the set $K$ cannot be embedded in a bi-Lipschitz way to any finite-dimensional smooth (Lipschitz) manifold, no matter how fast is the decay rate of $\eb_n$. This demonstrates the principal difference between the Bouligand and fractal dimensions and gives the simplest known obstruction to the existence of bi-Lipschitz embeddings. Let us now consider the log-doubling factor and the log-Lipschitz embeddings. The situation here is slightly different since it will be finite for $K$ if the decay rate of $\eb_n$ is too fast. But if the following condition holds:
\begin{equation}\label{5.toostr}
\limsup_{n\to\infty}\frac{\log n}{\log\log\eb_n^{-1}}=\infty,
\end{equation}
 then estimate \eqref{5.notgood} guarantees that the log-doubling factor is also infinite. This will be true, for instance if $\eb_n\sim e^{-(\log n)^\gamma}$ with $\gamma>1$. Note also that the presence of the whole orthogonal segments  in $K$ is not necessary, it is sufficient that the orthogonal vectors $\eb_n e_{l+1},\cdots,\eb_n e_{l+n}$ belong to $K$ for some $l=l(n)$ and any $n$.
 \par
 However, assumption \eqref{5.toostr} is in fact a great restriction which does not allow $\eb_n$ to decay exponentially fast and that is not appropriate for our purposes. By this reason, we slightly modify the idea with orthogonal sequences by allowing the vortices of the $n$-dimensional cubes of sizes $\eb_n$ to be in $K$. Namely, the following result holds.

\begin{proposition}\label{Prop5.double} Let $g_n$ be the set of $2^n$ vortices of the $n$-dimensional cube $[0,1]^n$ and let the sequence $\eb_n$ satisfy
\begin{equation}\label{5.good}
\limsup_{n\to\infty}\frac n{(\log n)(\log\log\eb_n^{-1})}=\infty.
\end{equation}
Assume that a compact set $K\subset H$ possesses isometric embeddings $\eb_n g_n\hookrightarrow K$ for all $n$.
Then, the logarithmic doubling factor of $K$ is infinite.
\end{proposition}
\begin{proof} Indeed, since $\eb_ng_n$ is isometrically embedded in $K$ all $2^n$-vortices are in the ball of radius $\eb_n\sqrt n$ in K and all pairwise distances between them are not less than $\eb_n$. Thus,
\begin{equation}\label{5.nice}
\sup_{x\in K}N_{\eb_n/2}(K\cap B(\eb_n\sqrt n,x,H),H)\ge 2^n.
\end{equation}
Using now the obvious formula that
$$
\sup_{x\in K}N_{\eb/4}(K\cap B(\eb,x,H),H)\le D_\eb(K,H)D_{\eb/2}(K,H),
$$
we can conclude from \eqref{5.nice} that there is at least one $\eb\in[\eb_n/2,\eb_n\sqrt n]$ such that
$$
D_\eb(K,H)\ge 2^{\frac n{\log_2 n}}
$$
and this estimate together with assumption \eqref{5.good} imply that the log-doubling factor of $K$ is infinite and finishes the proof of the proposition.
\end{proof}
We see that, in contrast to \eqref{5.toostr}, assumption \eqref{5.good} allows us to take $\eb_n\ge e^{-e^{n^\gamma}}$ for all $\gamma<1$ and the
typical size of $\eb_n$ in our counterexample will be $\eb_n\sim e^{-Cn^2}$, see below, so Proposition \ref{Prop5.double} is more than enough for our purposes.
\par
Now we turn to the more interesting and more challenging question: how to embed the sets like \eqref{5.set} or/and the vortices $\eb_n g_n$ of the $n$-dimensional cubes into the global attractor $\Cal A$ of equation \eqref{1.eqmain}? We start from the most straightforward construction which does not solve the problem, but nevertheless is an important step in that direction.

\begin{example}\label{Ex5.not}
Let us consider the following 2D system:
\begin{equation}\label{7.1}
\begin{cases}
x'=-x(x^2+y^2-1),\\
y'=-y(x^2+y^2-1).
\end{cases}
\end{equation}
This system can be rewritten in polar coordinates $x+i y=Re^{i\varphi}$ as follows:
\begin{equation}\label{7.2}
R'=-R(R^2-1),\ \ \varphi'=0.
\end{equation}
Thus, the global attractor $\Cal A_0\subset\R^2$ of this system consists of all points $R\in[-1,1]$, $\varphi\in[0,2\pi]$ and is a ball $x^2+y^2\le1$. Moreover, zero is an unstable equilibrium, $\varphi=const$ is the first integral, so the trajectories are the straight lines (radii) passing through zero and every point on a circle $x^2+y^2=1$ is a stable equilibrium.
\par
 We will couple this ODE with an equation in $H$
$$
\Dt w+Aw=F(x,y)
$$
by constructing the coupling  $F$ which excites every mode $e_n$ by  the "proper kick"
in order to embed the orthogonal segments in the directions of $e_n$ into the attractor.
To this end,
we now split the segment $\varphi\in[0,2\pi]$ into  infinite number of intervals $I_n$ with $|I_n|=E_n$, $n=1,2\cdots$ satisfying
 \begin{equation}\label{7.3}
\sum_{n=1}^\infty E_n<2\pi.
\end{equation}
Roughly speaking, the solutions of \eqref{7.1} with $\varphi\in I_n$ will be coupled with the $n$-th mode $e_n$ for producing the $n$th orthogonal segment. To be more precise, we fix $\varphi_n\in I_n$ and a family of the cut-off functions $\psi_n\in C_0^\infty(\R)$ such that $\psi_n(\varphi_k)=\delta_{nk}$ and
\begin{equation}\label{7.4}
\|\psi_n\|_{C^k(\R)}\le C_k E_n^{-k}, \ n =1,2,...,
\end{equation}
where the constant $C_k$ is independent of $n$, and the cut-off function $\theta \in C_0^\infty(\R)$  such that
\begin{equation}\label{7.5}
\theta(R)\equiv0,\ \ R\le1/4,\ \ \theta(R)\equiv1,\ \ R\in[1/2,1].
\end{equation}
Finally, we introduce the operator $F: \R^2\to H$ via
\begin{equation}\label{7.6}
F(x,y):=\sum_{n=1}^\infty B_n\theta(x,y)\psi_n(x,y)e_n,
\end{equation}
where the monotone decreasing to zero sequence $B_n$ will be specified below, and consider the coupled
system
 \begin{equation}\label{7.7}
\begin{cases}
x'=-x(x^2+y^2-1),\ \ y'=-y(x^2+y^2-1),\\
\Dt w+A w=F(x,y).
\end{cases}
\end{equation}
The elementary properties of this system are collected in the following lemma:

\begin{lemma}\label{Lem7.1} Let $H^s:=D(A^{s/2})$ be the scale of H-spaces generated by the operator $A$. Then,
\par
1) The map $F$ defined by \eqref{4.6} belongs to
$C^k(\R^2,H^s)$ iff
\begin{equation}\label{7.8}
\sup_{n\in\Bbb N}\{B_n\lambda_n^s E_n^{-k}\}<\infty.
\end{equation}
\par
2) The initial value problem for the equation \eqref{7.7} possesses a global attractor $\Cal A$ in $\R^2\times H$ which contains the following sequence of equilibria:
 \begin{equation}\label{7.9}
P_n:=(\cos\varphi_n,\sin\varphi_n,B_n\lambda_n^{-1}e_n)\in\Cal A
\end{equation}
for $n=1,2,\cdots$. In addition, let $Q_2:\R^2\times H\to H$ be the orthoprojector to the $w$-component of \eqref{7.7}. Then, the projection
$Q_2\Cal A$ contains the following sequence of segments:
 \begin{equation}\label{7.10}
S_n:=\{se_n,\ \ s\in[0,B_n\lambda_n^{-1}]\}\subset Q_2\Cal A.
\end{equation}
\end{lemma}
\begin{proof} Indeed, the first assertion of the lemma is an immediate corollary of the definition \eqref{7.6}, the estimate \eqref{7.4} and the fact that,
 for every fixed $(x,y)\in\Bbb R^2$, at most one term in \eqref{7.6} is non-zero. To verify \eqref{7.9} it remains to note that, due to the definition of the cut-off functions and the operator $F$, $P_n$ is an equilibrium of \eqref{7.7} for every $n\in\Bbb N$ and \eqref{7.10} follows from the fact that there is a heteroclinic orbit connecting $(\cos\varphi_n,\sin\varphi_n,0)$ and $P_n$, again for every $n\in\Bbb N$.
\end{proof}
Embedding \eqref{7.10} and Proposition \ref{Prop5.inf} give that
\begin{equation}\label{7.11}
D(Q_2\Cal A,H)=\infty
\end{equation}
if the decreasing sequence $B_n$ is strictly positive and, by this reason, $Q_2\Cal A$ cannot be embedded into any finite-dimensional Lipschitz manifold. To show that the log-doubling factor of $Q_2\Cal A$ can be also infinite, one just needs to be a bit more accurate with the choice of $E_n$ and $B_n$.
\par
Let $\lambda_n\sim n^\kappa$, $\kappa>0$ (which corresponds to the case where $A$ is an elliptic  differential operator in a bounded domain), $E_n:=1/n^2$ and
 $$
 B_n\sim e^{-\log^{2} n}=n^{-\log n}.
 $$
  Then, since $B_n$ decays faster than polynomially, according to \eqref{7.8}, the non-linearity $F$ in equations  \eqref{7.7} belongs to $C^\infty(\R^2,H^s)$ for any $s\in\R$. On the other hand, assumption \eqref{5.toostr} is satisfied for for segments \eqref{7.10}, so the log-doubling factor is indeed infinite:
 \begin{equation}\label{7.12}
D_{log}(Q_2\Cal A,H)=\infty.
\end{equation}

\begin{remark}\label{Rem7.3} Note that the equality \eqref{7.12} does not hold for
 the attractor $\Cal A$ itself, but only for its "bad" projection $Q_2\Cal A$. In addition, modifying equation \eqref{7.7} as follows
  \begin{equation}\label{7.14}
\begin{cases}
x'=-\beta x(x^2+y^2-1),\ \ y'=-\beta y(x^2+y^2-1),\\
\Dt w+A w=\beta F(x,y),
\end{cases}
\end{equation}
where $\beta\ll1$ is a small positive parameter, we see that the above arguments still work and the log-doubling factor of $Q_2\Cal A$ is infinite, but now the spectral gap condition is satisfied and there is a $C^1$-smooth inertial manifold diffeomorphic to $\R^2\ni(x,y)$ containing the global attractor $\Cal A$ (by choosing  $\beta>0$ small enough this manifold can be made of the class $C^k$ for any fixed $k>0$). Thus, the possibility to
embed the attractor smoothly into the finite-dimensional manifold indeed can be lost under the "badly chosen" orthogonal projection. The above described construction can be also used in order to give an example of equation \eqref{1.eqmain} with the nonlinearity $F$ of finite smoothness such that
$$
\dim_H(\Cal A,H)=\dim_F(\Cal A, H)=\dim_H(\Cal A,H^3)=2,\ \ \dim_F(\Cal A,H^3)=\infty,
$$
so at least the fractal dimension can indeed depend strongly on the choice of the phase space,
see \cite{EKZ} for the details.
\end{remark}
\end{example}
The previous example gives us the almost desired embedding of the orthogonal segments $S_n$ into the attractor. However, this is still not enough to verify that the log-doubling factor of $\Cal A$ is infinite. The problem is that $\Cal A$ now lives in the extended space $(x,y,w)\in\R^2\times H$ and, although in the $H$-component the orthogonal sequence is built up, the inverse image of every $S_n$ on the attractor has also the $(x,y)$-component and this $(x,y)$-component is dominating (essentially larger than the $H$-component), so the arguments of Proposition \ref{Prop5.inf} do not work and the doubling factor of $\Cal A$ is finite. Thus, the next natural question is how to suppress/remove the non-desirable $(x,y)$-component and to built up the true orthogonal sequences/segments in the attractor?
\par
The next example shows that it is possible to do in the case where the $(x,y)$-modes decay {\it faster} than all other modes $e_n$. Indeed, then even if at the initial time, say, $t=0$, the $(x,y)$-component is dominating, it nevertheless will be suppressed during the time evolution and the true orthogonal sequences will appear. Note that such behavior of the modes, in particular, the existence of infinitely many modes which decay slower than, say, the first two is typical for the case of damped hyperbolic equations, so the next example will be related with that class of equations and the analogous construction for the case of reaction-diffusion problems will be discussed later.

\begin{example}\label{Ex7.hyp} Consider the 2D system
\begin{equation}\label{7.ode}
x'=-x(x^2+y^2-1)^2,\ \ y'=-y(x^2+y^2-1)^2
\end{equation}
which differs form \eqref{7.1} by the presence of squares. As not difficult to see, the attractor $\Cal A_0$ of this system is exactly the same as for equations \eqref{7.1}: it is just a disk $x^2+y^2\le1$ but the time evolution on it is reversed. Namely, the circle $x^2+y^2=1$ is now filled by unstable equilibria, their unstable manifolds are the radii $\phi=const$ and zero is a stable equilibrium.
\par
We couple these equations with the following second order abstract hyperbolic equation:
\begin{equation}\label{7.pde}
\Dt^2 w+2\gamma\Dt w+Aw=F(x,y),
\end{equation}
where $F(x,y)$ is exactly the same as in Example \ref{Ex5.not} and $0<\gamma<1$ is a dissipation parameter.
\par
Then, exactly as in Lemma \ref{Lem7.1}, we have a family of equilibria \eqref{7.9} together with the family of heteroclinic orbits $u_n(t):=(x_n(t),y_n(t),w_n(t))$ connecting $P_n$ with zero. We fix time parametrization on that orbits in such way that
$$
R_n(0):=\sqrt{x_n^2(0)+y_n^2(0)}=1/2.
$$
Then, according to our construction of the cut-off functions,
$$
w_n(0)=B_n\lambda_n^{-1}e_n,\ \  \Dt w_n(0)=0.
$$
Moreover, since near zero equations \eqref{7.ode} and \eqref{7.pde} are decoupled and look like
$$
R'\sim-R,\ \ w_n''(t)+2\gamma w_n'(t)+\lambda_nw_n(t)=0,
$$
we conclude also that
\begin{equation}\label{7.decay}
R(t)\sim Ce^{- t},\ \ \sqrt{\lambda_n[w_n(t)]^2+[w_n'(t)]^2}\sim CB_n\lambda_n^{-1}e^{-\gamma t}
\end{equation}
as $t\to\infty$, where the constant $C$ is independent of $n$ (since the function $w_n(t)$ has only one non-zero component which corresponds to the mode $e_n$, in slight abuse of notations, we will not distinguish between $w_n\in H$ and its $n$th component $w_n\in\R$).
\par
We see that indeed initially at $t=0$ the $R$-component is essentially larger than the other ones, but since $\gamma<1$, the $R$-component becomes essentially smaller than the $w_n$-component if $t\ge T=T_n$ where
$$
e^{-T}\sim B_n\lambda_ne^{-\gamma T},\ \ T\sim \frac1{1-\gamma}\log(\lambda_n B_n^{-1}),\ B_n\lambda_n^{-1}e^{-\gamma T} \sim\(B_n\lambda_n^{-1}\)^{\frac1{1-\gamma}}.
$$
Thus, if we consider a sequence $\{u_1(t^1_n),u_2(t^2_n),\cdots, u_n(t^n_n)\}$ belonging to the attractor, where $t^k_n\ge T_n$ are chosen in such way that
$$
\sqrt{\lambda_k[w_k(t^k_n)]^2+[w_k'(t^k_n)]^2}=\(B_n\lambda_n^{-1}\)^{\frac1{1-\gamma}},
$$
then all $(x,y)$ components will be essentially smaller than the $H$-components, so they can be neglected and we have a sequence close to $\{\eb_n e_1,\cdots \eb_n e_n\}$ with $\eb_n\sim \(B_n\lambda_n^{-1}\)^{\frac1{1-\gamma}}$. Then, fixing  $B_n$, $E_n$ and $\lambda_n$ as in Example \ref{Ex5.not}, we see that assumption \eqref{5.toostr} is satisfied and, therefore, the corresponding attractor $\Cal A$ indeed has an infinite
log-doubling factor. Note also that, as in Example \ref{Ex5.not}, the nonliearity $F$ is $C^\infty$-smooth. Thus, the first example of the attractor with infinite log-doubling factor and without bi-log-Lipschitz Man\'e projections is constructed.
\end{example}
\begin{remark}\label{Rem7.hyp} Although the damped hyperbolic equations are also natural and important from both theoretical and applied points of view, their systematic study is out of scope of these notes. Mention only that their properties are essentially different from the abstract parabolic equations considered before, in particular, they do not possess any smoothing on a finite time interval, moreover, the solution operators are usually invertible and form a group, not only a semigroup, so the Romanov theory (see Theorem \ref{Th2.rom}) clearly does not work here, etc. We refer the reader to \cite{tem,BV,CV} and the references therein for more details on this type of dissipative PDEs.
\par
Thus, our reason to present Example \ref{Ex7.hyp} was not to give a rigorous exposition of the theory of hyperbolic equations and even not to give a rigorous statement of the result, but only to indicate the main idea how the almost orthogonal sequences can naturally appear in the attractor if the so-called normal hyperbolicity is sevearly violated and we have infinitely many modes which decay essentially slower than the fixed one. The damped hyperbolic equations are {\it ideal} from this point of view since the spectrum of \eqref{7.pde} with $F=0$ is given by
$$
\mu_n=-\gamma\pm\sqrt{\gamma^2-\lambda_n}
$$
and mostly belongs to the line $\Ree \lambda=-\gamma$. So, you just need to generate the first mode which decays faster than $e^{-\gamma t}$ which is done by adding the ODEs \eqref{7.ode}. One more advantage of this example is that it can be easily realized not only on the level of {\it abstract} operator-valued equations, but also on the level of hyperbolic PDEs with "usual" nonlinearities. The analogous question in the parabolic setting is much more difficult and is still an open problem.
\end{remark}
We now return to the case of abstract parabolic equations \eqref{1.eqmain}. At first glance, it looks extremely difficult/impossible to generate something similar to Example \ref{Ex7.hyp} on the level of parabolic equations. Indeed,  the structure of the linear part of the equation is completely different and we now have a sequence of real eigenvalues $-\lambda_n\to-\infty$, so how can we generate an infinitely many modes which will decay {\it slower} than a fixed mode?
\par
That is indeed impossible if we try to do that based on the linearization near an {\it equilibrium}, however, if we replace an equilibrium by the properly chosen {\it periodic orbit} it becomes possible and the surprising counterexample to the Floquet theory constructed in the previous paragraph is extremely helpful again. Indeed, as the next lemma shows, although all solutions of equation \eqref{4.2} decay super-exponentially (as $e^{-t^2}$), the decay rate is nevertheless essentially non-uniform (roughly speaking, the trajectories starting from $v(0)=e_{2n}$ decay essentially slower than the ones starting from $v(0)=e_{2n+1}$) and that will be enough to reproduce the main structure of Example \ref{Ex7.hyp} and to obtain the attractor with infinite log-doubling factor.

\begin{lemma}\label{Lem4.3} Let $P$ be the Poincare map associated with the problem \eqref{4.2}. Then, for every
 $n\in\Bbb N$, $n\le s_1,s_2\le n+k$, $k:=[\sqrt{n}]$ and $N=2n+k$, the following estimates hold:
\begin{equation}\label{7.non-uniform}
\frac{\|P^Ne_1\|}{\|P^Ne_{2s_1}\|}\le e^{-\beta n^2},\ \ e^{-\gamma n^{3/2}}\le\frac{\|P^Ne_{2s_1}\|}{\|P^Ne_{2s_2}\|}\le e^{\gamma n^{3/2}}
\end{equation}
for some positive constants $\beta$ and $\gamma$ and all $n$ large enough
instead of its integer part).
\end{lemma}
The proof of this lemma follows in a straightforward way from \eqref{4.5} and \eqref{4.6} and assumption \eqref{4.3} on the eigenvalues and, by this reason is omitted, see \cite{EKZ} for the details.
\par
We conclude this paragraph by a very brief exposition of the main construction involved in the proof of Theorem \ref{Th5.1}

\begin{proof}[Sketch of the proof of Theorem \ref{Th5.1}] We will base on the construction given in the proof of Theorem \ref{Th4.1} as well as on Example \ref{Ex7.hyp}. Indeed, according to Lemma \ref{Lem4.3}, the "first mode" which corresponds to $w(0)=e_1$ in equation \eqref{4.2} decays {\it faster} than infinitely many "other modes" associated with $w(0)=e_{2n}$, so we will use the kick in their directions in order to build up the orthogonal sequences. Namely, we fix the trajectory $v(t)=(x(t),y(t),w(t))\in\Cal A$ of coupled equations \eqref{4.23} which is constructed in the proof of Theorem \ref{Th4.1} and will treat it as a basic non-perturbed solution of Example \ref{Ex7.hyp}. Furthermore, we modify the 3D system \eqref{4.24} in such way that not only the point $v(0):=(x(0),y(0),e_1)$, but also the whole segment $v_s(0):=(x(0),y(0),se_1)$, $s\in[1,1-\kappa]$
belongs to the unstable manifold and, therefore, to the attractor as well. This parameter $s$ will play the role of phase $\varphi$ in Examples \ref{Ex5.not} and \ref{Ex7.hyp} and will parameterise the kicks in the directions of $e_{2n}$. Namely, we split the interval
$$
[1,1-\kappa]=\cup_{n=1}^\infty I_n
$$
into a sum of disjoint intervals $I_n$, after that, every interval $I_n$ will be further split into $2^{[\sqrt n]}$ intervals $I_{n,p}$, $p=1,\cdots,2^{[\sqrt{n}]}$, subintervals and finally, in the small neighbourhood of the trajectory $v_{s_{n,p}}(0)$, we construct a
kick on the modes $e_{2s}$, $s=n,n+1,\cdots, n+[\sqrt{n}]$ which after the properly chosen time evolution will reproduce the $p$th vortex of the $[\sqrt{n}]$-dimensional cube spanned on these eigenvectors. The accurate calculations, see \cite{EKZ} show that the size $\eb_n$ of that cube can be made not smaller than $e^{-\beta n^2}$ for some $\beta$ which is independent of $n$. Thus, assumption \eqref{5.good} of Proposition \ref{Prop5.double} is satisfied and, therefore, the log-doubling factor of the attractor $\Cal A$ is indeed infinite.
\end{proof}

\section{Concluding remarks and open problems}\label{s4}
In this concluding section, we briefly summarize the main achievements of the theory developed above and discuss the related open problems. We start with the most principal and important question on the validity of the finite-dimensional reduction for dissipative PDEs.
\par
{\it Is the dissipative dynamics generated by PDEs effectively finite-dimensional?}
\par
It is more or less clear that the answer on this question is positive in the case when the underlying PDE possesses an inertial manifold. Indeed, in that case, the reduced dynamics on the manifold is described by a system of ODEs (the so-called inertial form) and every trajectory which starts outside of the manifold is exponentially attracted to some trajectory on the manifold. Thus, up to a transient behavior, the dynamics is described by the finite system of ODEs on the inertial manifold. As a drawback, mention that even in that ideal situation, the finite-dimensional reduction drastically decreases the smoothness and starting from the $C^\infty$ or even analytic equations, we may guarantee in general that the reduced equations are only $C^{1+\eb}$-smooth for some small  $\eb>0$, see Paragraph \ref{s.smo}.
\par
However, as we have seen, the existence of an inertial manifold is a big restriction on the system considered, in particular, some kind of spectral gap conditions are necessary for that. In addition, although on the level of {\it parabolic} equations, all known examples where the inertial manifold does not exist look a bit artificial, such counterexamples are natural in the class of damped {\it hyperbolic} equations where, vise versa,
the equations with inertial manifolds are rare and artificial, see \cite{sola1,sola2}. Thus, the inertial manifolds are not sufficient to solve completely the finite-dimensional reduction problem.
\par
The positive answer on the above question on the finite-dimensional reduction is often referred as the main achievement of the attractors theory and the fractal/Hausdorff dimension of the attractors is treated as a quantity related with the effective number of degrees of freedom of the reduced system. This, in particular, motivates and justifies many modern works on improving the existing and finding new upper and lower bounds for the fractal dimension of attractors for various equations of mathematical physics, see \cite{tem,BV,CV} and references therein.
\par
In a sense, this interpretation looks indeed reasonable since, as we have seen, the finiteness of the fractal dimension of the attractor allows us to project it one-to-one to a generic finite-dimensional plane using the Man\'e projection theorem and this, in turn, gives the analogue of the inertial form, see \eqref{2.inertial}, which is a system of ODEs describing the dynamics on the global attractor. And indeed the fractal dimension of $\Cal A$ determines the minimal dimension of the plane when such reduction is possible.
\par
However, the reduction based on the Man\'e projection theorem has a huge drawback, namely, the non-linearity in the reduced ODEs obtained in that way is only {\it H\"older} continuous which is not enough even to establish the uniqueness (and as we have seen in Example \ref{Ex2.non-unique}, the non-smooth Man\'e projections indeed may generate extra pathological solutions). Thus, we actually reduced the initial infinite-dimensional, but "nice", say, $C^\infty$-smooth dynamical system to the finite-dimensional, but "ugly" equations with non-smooth nonlinearities, without uniqueness and with possible presence of extra pathological solutions. Clearly, it can hardly  be accepted as an "effective" reduction.
\par
This motivates the  attempts to improve the regularity of the inverse Man\'e projections which have been considered at Section \ref{s.bsg}. As we have seen there, despite of a number of interesting ideas and approaches developed, the uniqueness problem for the reduced equations is still far from being solved. Moreover, most of the obtained results are sharp and the corresponding counterexamples are given (e.g., the spectral gap condition is sharp and is responsible not only for $C^1$-smooth, but also for the Lipschitz and log-Lipschitz manifolds; the exponents in log-Lipschitz Man\'e projection theorem, see Theorem \ref{Th2.log},  as well as in the log-convexity result of Proposition \ref{Prop2.Alip} are sharp, etc.). So, the limitations of the theory are already achieved and there is no much room for further improvements, see also Remark \ref{Rem2.badd}. Thus, principally new ideas/approaches are required in order to solve the uniqueness problem or/and to find more  {\it effective} ways for the finite-dimensional reduction. Mention also that, as we have seen, the fractal dimension of the attractor {\it is not} responsible for the existence of smoother inertial forms (e.g., more delicate  Bouligand dimension is related with Lipschitz and log-Lipschitz projections, etc.), so its role in justifying the finite-dimensional reduction in dissipative systems looks exaggerated.
\par
On the other hand, the possibility of the {\it negative} answer on the above question about the finite-dimensional reduction deserves a serious attention. Indeed, the periodic orbit with super-exponential attraction constructed in Paragraph \ref{s.Floquet} does not look  as a finite-dimensional phenomenon as well  as the period map nearby which is conjugated to the shift operator in $l_2(\Bbb Z)$. Note that such shift operator models the infinite-dimensional Jordan sell and is a standard example of the {\it infinite-dimensional} behavior in the operator theory, see \cite{dan}. We also do not see there any higher and lower modes, lengthscales, etc. which are expected in the case when the finite-dimensional reduction works. So, the drastic loss of smoothness under any attempts to reduce the dynamics to finite system of ODEs may be       interpreted as a manifestation of the fact that the dynamics of dissipative PDEs is  {\it infinite-dimensional}. The proved theorems on the finite-dimensional reduction via the H\"older Man\'e theorem is naturally treated then as a mathematical "pathology" with restricted practical relevance (similar to, say, the Peano curve). This hypothesis will be confirmed if one construct an example where the dynamics on the attractor is not conjugated even {\it topologically} to any {\it smooth} finite-dimensional dynamics. The constructed counterexamples show that such conjugating homeomorphism cannot be bi-H\"older continuous (due to the super-exponential convergence to the limit cycle), but finding the obstructions to the topological equivalence with a {\it smooth} system of ODEs is an interesting and challenging open problem.
\par
{\it Open problems related with the spectral theory.}
\par
 As we have seen, in order to verify the existence of an inertial manifold (at least in a straightforward way), we need the spectral gap condition which is a condition on a difference of two subsequent eigenvalues of the operator $A$, see Theorems \ref{Th1.main} and \ref{Th1.dmain}. Thus, it becomes crucial to be able to check this condition for the concrete operators arising in applications, for instance, for the Laplacian in a bounded domain of $\R^n$, say, with Dirichlet or Newmann boundary conditions. However, despite the vast amount of various results on the distribution of eigenvalues and improvements of the classical Weyl asymptotics, see \cite{ivri,bolt} and references therein, almost nothing is known about the existence of spectral gaps even in the case of the Laplacian. In particular, for a bounded domain of $\R^2$, we do not know any examples when the spectral gaps of arbitrarily large size do not exist, but, to the best of our knowledge, there are also no results that such gaps exist for all/generic domains of $\R^2$. Clarifying the situation with these gaps would essentially increase the number of equations with inertial manifolds.
\par
Another interesting question is related with the Floquet theory for PDEs and the spectrum of the period map. As we have seen in Paragraph \ref{s.Floquet}, in the case of {\it abstract} parabolic equations, the spectrum may just coincide with zero and do not contain any Floquet exponents. However, to the best of our knowledge, there are still no examples of parabolic {\it PDEs} with such properties. In particular, it is even not clear how to realize two equilibria involved in the proof of Theorem \ref{Th3.1} on the level of PDEs. Finding the explicit examples with such properties or clarifying the nature of obstructions to that is important for understanding how natural are the related infinite-dimensional phenomena on the level of concrete PDEs arising in applications.
\par
{\it Navier-Stokes equations and inertial manifolds.}
\par
As known, the Navier-Stokes problem is one of the most important examples in the theory of dissipative PDEs and many concepts and methods of the modern theory  were inspired exactly by this problem and by the dream to understand the underlying turbulence. In particular, even the strange word  {\it inertial} in the title of {\it inertial manifold} (which replaces/duplicates  more precise and more natural notion of a {\it center manifold}) comes from the Navier-Stokes equations, namely, it is related with the so-called inertial term in the equations as well as the associated inertial scale in the empiric theory of turbulence, see \cite{firsh,FMRT} and reference therein. However, despite many attempts there are still  no progress  in finding the inertial manifolds for that problem or proving their non-existence, so the problem remains completely open. Mention here only the attempt of Kwak, see \cite{kwak,tem-wrong} and explanations  in Remark \ref{Rem1.gen} why this does not work. It also worth to mention that the transformation of the Navier-Stokes to semilinear reaction-diffusion equations given there nevertheless looks interesting and maybe useful. In particular, it would be interesting to investigate whether or not similar transforms with better properties (e.g., which lead to  self-adjoint operators $A$ and which is enough to repair the proof) exist.
\par
{\it Generalizations of invariant cones and spatial averaging.}
\par
To conclude, we discuss several important open problems which, in contrast to the previous ones, do not look insane and likely can be solved without inventing the principally new techniques. One of them is related to extending Theorem \ref{Th1.maincon} to the case where the invariant cones {\it depend} explicitly on the point of the phase space. The model equation here is the 1D semilinear heat equation considered in Example \ref{Ex2.grad}, e.g., 1D Burgers type equations. As we have seen in Remark \ref{Rem2.manifold}, these equations possess such invariant cones, so it would be nice to develop a general theory which gives the existence of an inertial manifold for these equations as a natural application.
\par
Another interesting problem is related to various generalizations of the spatial averaging principle. Some generalizations to the case of domains differ from the tori are given in \cite{kwean}, see also \cite{mar} for the generalizations to the coupled PDE-ODE systems. It would be perfect to establish something similar  for the 2D sphere $\Bbb S^2$. Since the Navier-Stokes equations on $\Bbb S^2$ are in a sense on a borderline, even slight weakening of the spectral gap condition would be extremely helpful for the inertial manifold existence.
\par
It would be also nice to generalize the spatial averaging to the cases where the averaged operator can be not the scalar one. For instance, the proper generalization which a priori looks almost straightforward should give the inertial manifold for the Cahn-Hilliard equation on the 3D torus.
A bit more challenging problem is to extend the technique to the case of complex Ginzburg-Landau equations on the 3D tori.

 \end{document}